\documentclass[11pt, a4paper]{amsart}
\pdfoutput=1 
\usepackage[dvips]{graphicx}
\usepackage{amssymb, amstext, amscd, amsmath, color}
\usepackage{epstopdf}
\usepackage{tikz,mathdots,enumerate}
\usepackage{pgfplots}
\usepackage{url}

\usepackage{kbordermatrix} 
\renewcommand{\kbldelim}{[} 
\renewcommand{\kbrdelim}{]}

\usepackage{hhline}

\usepackage{hyperref}

\begin{document}

\newtheorem{thm}{Theorem}[section]
\newtheorem{cor}[thm]{Corollary}
\newtheorem{prop}[thm]{Proposition}
\newtheorem{lem}[thm]{Lemma}
%

\theoremstyle{definition}
\newtheorem{rem}[thm]{Remark}
\newtheorem{defn}[thm]{Definition}
\newtheorem{note}[thm]{Note}
\newtheorem{eg}[thm]{Example}
\newcommand{\Prf}{\noindent\textbf{Proof.\ }}
\newcommand{\bx}{\hfill$\blacksquare$\medbreak}
\newcommand{\upbx}{\vspace{-2.5\baselineskip}\newline\hbox{}%
\hfill$\blacksquare$\newline\medbreak}
\newcommand{\eqbx}[1]{\medbreak\hfill\(\displaystyle #1\)\bx}

\newcommand{\FFock}{\mathcal{F}}
\newcommand{\kil}{\mathsf{k}}
\newcommand{\Hil}{\mathsf{H}}
\newcommand{\hil}{\mathsf{h}}
\newcommand{\Kil}{\mathsf{K}}
\newcommand{\Real}{\mathbb{R}}
\newcommand{\Rplus}{\Real_+}

\newcommand{\bC}{{\mathbb{C}}}
\newcommand{\bD}{{\mathbb{D}}}
\newcommand{\bN}{{\mathbb{N}}}
\newcommand{\bQ}{{\mathbb{Q}}}
\newcommand{\bR}{{\mathbb{R}}}
\newcommand{\bT}{{\mathbb{T}}}
\newcommand{\bX}{{\mathbb{X}}}
\newcommand{\bZ}{{\mathbb{Z}}}
\newcommand{\bH}{{\mathbb{H}}}
\newcommand{\BH}{{\B(\H)}}
\newcommand{\bsl}{\setminus}
\newcommand{\ca}{\mathrm{C}^*}
\newcommand{\cstar}{\mathrm{C}^*}
\newcommand{\cenv}{\mathrm{C}^*_{\text{env}}}
\newcommand{\rip}{\rangle}
\newcommand{\ol}{\overline}
\newcommand{\td}{\widetilde}
\newcommand{\wh}{\widehat}
\newcommand{\sot}{\textsc{sot}}
\newcommand{\wot}{\textsc{wot}}
\newcommand{\wotclos}[1]{\ol{#1}^{\textsc{wot}}}
 \newcommand{\A}{{\mathcal{A}}}
 \newcommand{\B}{{\mathcal{B}}}
 \newcommand{\C}{{\mathcal{C}}}
 \newcommand{\D}{{\mathcal{D}}}
 \newcommand{\E}{{\mathcal{E}}}
 \newcommand{\F}{{\mathcal{F}}}
 \newcommand{\G}{{\mathcal{G}}}
\renewcommand{\H}{{\mathcal{H}}}
 \newcommand{\I}{{\mathcal{I}}}
 \newcommand{\J}{{\mathcal{J}}}
 \newcommand{\K}{{\mathcal{K}}}
\renewcommand{\L}{{\mathcal{L}}}
 \newcommand{\M}{{\mathcal{M}}}
 \newcommand{\N}{{\mathcal{N}}}
\renewcommand{\O}{{\mathcal{O}}}
\renewcommand{\P}{{\mathcal{P}}}
 \newcommand{\Q}{{\mathcal{Q}}}
 \newcommand{\R}{{\mathcal{R}}}
\renewcommand{\S}{{\mathcal{S}}}
 \newcommand{\T}{{\mathcal{T}}}
 \newcommand{\U}{{\mathcal{U}}}
 \newcommand{\V}{{\mathcal{V}}}
 \newcommand{\W}{{\mathcal{W}}}
 \newcommand{\X}{{\mathcal{X}}}
 \newcommand{\Y}{{\mathcal{Y}}}
 \newcommand{\Z}{{\mathcal{Z}}}

\newcommand{\sgn}{\operatorname{sgn}}
\newcommand{\rank}{\operatorname{rank}}
\newcommand{\Isom}{\operatorname{Isom}}
\newcommand{\qIsom}{\operatorname{q-Isom}}
\newcommand{\Cknet}{{\mathcal{C}_{\text{knet}}}}
\newcommand{\Ckag}{{\mathcal{C}_{\text{kag}}}}
\newcommand{\rind}{\operatorname{r-ind}}
\newcommand{\lind}{\operatorname{r-ind}}

\setcounter{tocdepth}{1}

 \title[The rigidity of infinite graphs]{The  rigidity of infinite graphs}

\author[D. Kitson and S.C. Power]{D. Kitson and S.C. Power}
\thanks{Supported by EPSRC grant  EP/J008648/1.}
\address{Dept.\ Math.\ Stats.\\ Lancaster University\\
Lancaster LA1 4YF \\U.K. }
\email{d.kitson@lancaster.ac.uk, s.power@lancaster.ac.uk}


\begin{abstract}
A rigidity theory is developed for the Euclidean and non-Euclidean placements of countably infinite simple graphs in 
the normed spaces $(\bR^d,\|\cdot \|_q)$, for $d\geq 2$ and $1 <q < \infty$.
Generalisations are obtained for the Laman and Henneberg
combinatorial characterisations of generic infinitesimal rigidity for  finite graphs in  $(\bR^2,\|\cdot \|_2)$.  
Also Tay's multi-graph characterisation of the rigidity of generic finite body-bar frameworks in $(\bR^d,\|\cdot\|_2)$ is generalised to the non-Euclidean norms and to countably infinite graphs. 
For all dimensions and norms  it is shown that a generically rigid countable simple graph is the direct limit $G= \varinjlim G_k$ of an inclusion tower of finite graphs $G_1 \subseteq G_{2} \subseteq  \dots$ for which the inclusions satisfy a relative rigidity property. For $d\geq 3$ a countable graph which is rigid for generic placements in  $\bR^d$ may fail the stronger property of  sequential rigidity, while for $d=2$ the equivalence with sequential rigidity is obtained from the generalised Laman characterisations. 
Applications are given to the flexibility of
non-Euclidean convex polyhedra and to the infinitesimal and continuous rigidity of compact infinitely-faceted simplicial polytopes.
\end{abstract}
\date{\today}

\maketitle
\tableofcontents

\section{Introduction}
In 1864 James Clerk Maxwell initiated a combinatorial trend in the rigidity theory of finite bar-joint  frameworks in Euclidean space. In two dimensions this amounted to the observation that the underlying structure  graph $G=(V,E)$ must satisfy the simple counting rule $|E|\geq 2|V|-3$.
For minimal rigidity, in which any bar removal renders the framework flexible, equality must hold together with the inequalities
$|E(H)|\leq 2|V(H)|-3$ for subgraphs $H$ with at least two vertices.
The fundamental result that these two
necessary conditions are also sufficient for the minimal rigidity of a generic framework
was obtained by Laman in 1970 and this has given impetus to the development of matroid theory techniques. While corresponding counting rules
are necessary in three dimensions they fail to be sufficient and a purely combinatorial characterisation of generic rigidity is not available.  On the other hand many specific families of finite graphs are known to be generically rigid,  such as the edge graphs of triangular-faced convex polyhedra in three dimensions and the graphs associated with finite triangulations of general surfaces. See, for example, Alexandrov \cite{alex}, Fogelsanger \cite{fogel}, Gluck \cite{glu},
Kann \cite{kan} and Whiteley \cite{whi-poly1,whi-poly2}.

A finite simple graph $G$ is said to be generically $d$-rigid, or simply $d$-rigid, if its realisation as some generic bar-joint framework in the Euclidean space $\bR^d$ is infinitesimally rigid. 
Here \emph{generic} refers to the algebraic independence of the set of coordinates of the vertices and infinitesimal rigidity in this case is equivalent to continuous (nontrivial finite motion) rigidity (Asimow and Roth \cite{asi-rot,asi-rot-2}).
The rigidity analysis of bar-joint frameworks and related frameworks, such as
body-bar frameworks and  body-hinge frameworks, continues to be a focus of investigation, both in the generic case and in the presence of symmetry. For example  Katoh and Tanigawa \cite{kat-tan} have resolved the molecular conjecture
for generic structures, while
Schulze \cite{schulze} has obtained  variants of Laman's theorem for semi-generic symmetric
bar-joint frameworks. 
In the case of infinite frameworks however  developments have  centred mainly on periodic frameworks and the infinitesimal and finite motions which retain some form of periodicity. 
Indeed, periodicity hypotheses  lead to configuration spaces that are real algebraic varieties and so to methods from multi-linear algebra and finite combinatorics.
See, for example,
Borcea and Streinu \cite{bor-str}, Connelly et al. \cite{con-et-al}, Malestein and Theran \cite{mal-the},
Owen and Power \cite{owe-pow-crystal} and Ross, Schulze and Whiteley \cite{ros-sch-whi}. Periodic rigidity, broadly interpreted, is also significant in a range of applied settings, such as the mathematical analysis of rigid unit modes in crystals, as indicated in Power \cite{pow-poly} and  Wegner \cite{weg}, for example.

In the development below we consider general countable simple graphs and the flexibility and rigidity of their placements in the Euclidean spaces $\bR^d$ and in the non-Euclidean spaces $(\bR^d, \|\cdot\|_q)$ for the classical $\ell^q$ norms, for $1< q < \infty$. 
Also we consider several forms of rigidity, namely (absolute) infinitesimal rigidity, infinitesimal rigidity for continuous velocity fields,
and, perhaps the most intriguing, rigidity with respect to continuous time-parametrised flexes (finite flexes). The constraint conditions
for the non-Euclidean norms are no longer given by polynomial equations and  we adapt the Asimow-Roth notion of a regular framework to obtain the appropriate notion of  a generic framework. 
This strand of non-Euclidean rigidity theory for finite frameworks
was begun in Kitson and Power \cite{kit-pow} and here we develop this in parallel with the usual  Euclidean perspective.

In Theorem \ref{InfiniteLaman2} we obtain  generalisations of Laman's theorem in the form of a characterisation of the simple countable graphs which are generically  rigid in two dimensional spaces. In Theorem \ref{t:inductive} we obtain inductive Henneberg move style constructions for the minimally infinitesimally rigid graphs (isostatic graphs) in two dimensions. In the Euclidean case this entails that the minimally $2$-rigid 
graphs $G$ are those that can be given as the limit of a construction chain
\[
K_2 \overset{\mu_1}\longrightarrow G_2 \overset{\mu_2}\longrightarrow G_3 \overset{\mu_3}\longrightarrow \cdots
\]
where each of the moves $\mu_k$ is one of the two usual Henneberg moves
for the plane.
We also see that the infinitesimally rigid countable simple graphs that are infinitesimally rigid for two dimensions are necessarily {sequentially
rigid} in the sense that $G$ contains a spanning subgraph which is
a union of finite graphs, each of which is infinitesimally rigid.
This is the strongest form of infinitesimal rigidity and the corresponding equivalence fails in higher dimensions.

The inductive characterisations in two dimensions for the norms $\|\cdot\|_q, 1<q<\infty,q\neq 2,$ require the rigidity matroid associated with the sparsity count $|E|=2|V|-2$ and  additional construction moves. A countable graph is  infinitesimally rigid in this case if and only if it is inductively constructible from $K_1$ through an infinite sequence of $5$ types of  moves, namely, the Henneberg vertex and edge moves, the vertex-to-$K_4$ move, the vertex-to-$4$-cycle move, and edge additions. 

The results above rest in part on a general characterisation of infinitesimal rigidity in terms of what we refer to as the \emph{relative rigidity} of a finite  graph $G_1$ with respect to a containing finite graph $G_2$. Specifically, in all dimensions we show that a countable simple graph $G$ is infinitesimally rigid if and only if there is a  subgraph inclusion tower
\[
G_1\subseteq G_2 \subseteq G_3 \subseteq \cdots 
\]
which is vertex spanning and
in which for each $k$ the graph $G_k$ is relatively infinitesimally rigid in $G_{k+1}$ (Theorem \ref{t:IR}).
This principle is useful for determining the rigidity of infinite generic structures in a variety of other contexts and we show here that Tay's multi-graph characterisation of the rigidity of generic finite body-bar frameworks in $(\bR^d,\|\cdot\|_2)$, for all $d\geq 2$, can be generalised to the countably infinite setting. Also we obtain a similar infinite multi-graph characterisation of rigidity for the countably infinite body-bar frameworks in non-Euclidean spaces, for $1<q<\infty$. 

It is natural to determine variants of Cauchy's celebrated infinitesimal rigidity theorem for simplicial polytopes in $\bR^3$.
We see that a finite simplicial generic polytope in a non-Euclidean space is infinitesimally flexible but becomes isostatic after the addition of three internal nonincident shafts. 
We also define
a  family of  infinitely-faceted strictly convex compact polytopes in $\bR^3$ and the countable graphs defined by their edges. In the simplicial case such polytopes are associated with triangulations of a finitely punctured $2$-sphere
and we characterise the $3$-rigid graphs of this type. 

In the final section  we consider rigidity relative to continuous motions and give a non-Euclidean variant of the Asimow-Roth theorem on the equivalence of infinitesimal and continuous rigidity 
for  finite frameworks. Countable graphs can have both continuously flexible generic realisations and continuously rigid generic realisations, and the analysis of continuous  rigidity  presents new challenges, even for crystallographic frameworks. We show here that every simplicial polytope graph admits continuously rigid  placements (Theorem \ref{infinitepolytoperigid}) and we see that 
 a countable bar-joint framework
may be both infinitesimally rigid
and continuously flexible (Example \ref{InfRigCtsFlx}).


\section{Preliminaries}
\label{Preliminaries}
In this section we review the necessary background information on the rigidity of finite graphs in $\mathbb{R}^d$ with respect to the Euclidean norm and the non-Euclidean $\ell^q$ norms.
For further details of the Euclidean context we refer the reader to Graver, Servatius and Servatius \cite{gra-ser-ser} and the references therein. This section and each of the subsequent  sections is completed with brief notes.

\subsection{Continuous and infinitesimal rigidity.}
A bar-joint framework in a normed vector space $(X,\|\cdot\|)$ is a pair $(G,p)$ consisting of a simple graph $G=(V(G),E(G))$ and a mapping $p:V(G)\to X$, $v\mapsto p_v$ with the property that $p_v\not=p_w$ whenever $vw\in E(G)$. 
We call  $p$ a {\em placement} of $G$ in $X$ and  the collection of all placements of $G$ in $X$ will be denoted by $P(G,X)$ or simply $P(G)$ when the context is clear. If  $H$ is a subgraph of $G$  then the bar-joint framework $(H,p)$ obtained by restricting $p$ to $V(H)$ is called a {\em subframework} of $(G,p)$.

\begin{defn}\label{d:contsflex}
A {\em continuous flex} of  $(G,p)$  is  a family of continuous paths \[\alpha_v:[-1,1]\to X, \,\,\,\,\,\,\, v\in V(G)\] such that $\alpha_v(0)=p_v$ for all $v\in V(G)$ and  $\|\alpha_v(t)-\alpha_w(t)\|=\|p_v-p_w\|$ for all $t\in [-1,1]$ and all $vw\in E(G)$.
\end{defn}

A continuous flex is regarded as trivial if it results from a continuous isometric motion of the ambient space. Formally, a {\em continuous rigid motion} of $(X,\|\cdot\|)$ is  a mapping $\Gamma(x,t):X\times[-1,1]\to X$ which is isometric in the variable $x$ and continuous in the variable $t$ with $\Gamma(x,0)=x$ for all $x\in X$. 
Every continuous rigid motion gives rise to a continuous flex of $(G,p)$  by setting $\alpha_v:[-1,1]\to X$, $t\mapsto\Gamma(p_{v},t)$ for each $v\in V(G)$.
A continuous flex of $(G,p)$ is {\em trivial} if it can be derived from a continuous rigid motion  in this way. 
If every continuous flex of $(G,p)$ is trivial then we say that $(G,p)$ is {\em continuously rigid}, otherwise we say that $(G,p)$ is {\em continuously flexible}.
A bar-joint framework $(G,p)$ is {\em minimally continuously rigid} if it is continuously rigid and every subframework obtained by removing a single edge from $G$ is continuously flexible. 

\begin{defn}\label{d:infflex}
An {\em infinitesimal flex} of $(G,p)$ is a mapping $u:V(G)\to X$, $v\mapsto u_v$ which satisfies
\[\|(p_v+tu_v)-(p_w+tu_w)\|-\|p_v-p_w\|=o(t), \,\,\,\,\, \mbox{ as } t\to 0\]
for each edge $vw\in E(G)$. 
\end{defn}

We will denote the vector space of infinitesimal flexes of $(G,p)$  by $\mathcal{F}(G,p)$.
An {\em infinitesimal rigid motion} of $(X,\|\cdot\|)$ is a mapping $\gamma:X\to X$ derived from a continuous rigid motion $\Gamma$ by the formula $\gamma(x)=\frac{d}{dt}\Gamma(x,t)\vert_{t=0}$ for all $x\in X$.
The vector space of all infinitesimal rigid motions of $(X,\|\cdot\|)$ is denoted $\T(X)$.
Every infinitesimal rigid motion $\gamma\in\T(X)$ gives rise to an infinitesimal flex of $(G,p)$  by setting $u_v= \gamma(p_v)$ for all $v\in V(G)$.
We regard such infinitesimal flexes as trivial and the collection of all trivial infinitesimal flexes of $(G,p)$ is a vector subspace of $\F(G,p)$ which we denote by $\T(G,p)$. 
The {\em infinitesimal flexibility dimension} of $(G,p)$ is the vector space dimension of the quotient space,
\[
\dim_{\rm fl}(G,p):= \dim \F(G,p)/ \T(G,p)
\] 
If $\T(G,p)$ is a proper subspace  then $(G,p)$ is said to be an {\em infinitesimally flexible} bar-joint framework. 
Otherwise, we say that $(G,p)$ is {\em infinitesimally rigid} and we call $p$ an infinitesimally rigid placement of $G$. 
A bar-joint framework $(G,p)$ is {\em minimally infinitesimally rigid} if it is infinitesimally rigid and removing any edge results in a subframework which is infinitesimally flexible.

We will consider the rigidity properties of bar-joint frameworks in $\bR^d$ with respect to the family 
$\{\|\cdot\|_q:q\in(1,\infty)\}$ of $\ell^q$ norms,
\[\|\cdot\|_q:\mathbb{R}^d\to\mathbb{R},\,\,\,\,\,\,\,\|(x_1,\ldots,x_d)\|_q=\left(\sum_{i=1}^d|x_i|^q\right)^{\frac{1}{q}}\]
We  use a subscript $q$  to indicate the $\ell^q$ norm when referring to 
the collection of infinitesimal rigid motions $\T_q(\bR^d)$ and the infinitesimal flexes $\F_q(G,p)$ and trivial infinitesimal flexes $\T_q(G,p)$ of a bar-joint framework.  

In the Euclidean setting $q=2$ it is well-known that the space of infinitesimal rigid motions $\T_2(\mathbb{R}^d)$ has dimension $\frac{d(d+1)}{2}$.
In the non-Euclidean setting $(\bR^d,\|\cdot\|_q)$ where $q\in(1,2)\cup(2,\infty)$
 the  infinitesimal rigid motions are precisely the constant mappings and so $\T_q(\bR^d)$ is $d$-dimensional
(see \cite[Lemma 2.2]{kit-pow}). 

In the following proposition we write $h_v=(h_{v,1},\ldots,h_{v,d})$ for each $h_v\in \bR^d$. 

\begin{prop}
\label{q-NormFlex}
Let $(G,p)$ be a bar-joint framework in $(\mathbb{R}^d,\|\cdot\|_q)$ where $q\in(1,\infty)$.
Then a mapping $u:V(G)\to \mathbb{R}^d$ is an infinitesimal flex of $(G,p)$ if and only if 
\begin{eqnarray*}
\sum_{i=1}^d \sgn(p_{v,i}-p_{w,i})|p_{v,i}-p_{w,i}|^{q-1}(u_{v,i}-u_{w,i})=0
\end{eqnarray*}
for each edge $vw\in E(G)$. 
\end{prop}

\proof
Given  $u:V(G)\to \mathbb{R}^d$ and an edge $vw\in E(G)$ we define 
\[\zeta_{vw}(t)=\|(p_v+tu_v)-(p_w+tu_w)\|_q\]
Then  $u\in\F_q(G,p)$ if and only if $\zeta_{vw}'(0)=0$ for each $vw\in E(G)$. 
We have 
\[\zeta_{vw}'(0)=\frac{1}{q}\left[\sum_{i=1}^d|p_{v,i}-p_{w,i}|^q\right]^{\frac{1-q}{q}}\left[\sum_{i=1}^d g_i'(0)\right]\]
where 
$g_i(t)=|(p_{v,i}+tu_{v,i})-(p_{w,i}+tu_{w,i})|$.
If $p_{v,i}\not=p_{w,i}$ then 
\[g_i'(0) = q \sgn(p_{v,i}-p_{w,i})|p_{v,i}-p_{w,i}|^{q-1}(u_{v,i}-u_{w,i})\]
If  $p_{v,i}=p_{w,i}$ then $g_i'(0)=0$ since
\[\partial_+g(0)
= \lim_{t\to0^+} |u_{v,i}-u_{w,i}|t^{q-1}=0= \lim_{t\to0^-} -|u_{v,i}-u_{w,i}|t^{q-1} 
=\partial_-g(0)\]
The result  follows.
\endproof

If $G$ is a finite graph then the system of linear equations in Proposition \ref{q-NormFlex} can be expressed as a matrix equation $R_q(G,p)u=0$
where $R_q(G,p)$ is an $|E(G)|\times d|V(G)|$ matrix called the {\em rigidity matrix} for $(G,p)$.
The rows of $R_q(G,p)$ are  indexed by the edges of $G$ and the columns are indexed by the $d$ coordinates of $p_v$ for each vertex $v\in V(G)$. 
The row entries for a particular edge $vw\in E(G)$ are,
\[\kbordermatrix{& & & & p_v & && & p_w & & \\
vw & 0 & \cdots &0 & (p_v-p_w)^{(q-1)} &0& \cdots&0 &-(p_v-p_w)^{(q-1)}  &0& \cdots &0 }\]
where we use the notation  
\[(p_v-p_w)^{(q)}=(\sgn(p_{v,1}-p_{w,1})|p_{v,1}-p_{w,1}|^{q},\ldots,\sgn(p_{v,d}-p_{w,d})|p_{v,d}-p_{w,d}|^{q})\]
Evidently we have $\F_q(G,p)\cong\ker R_q(G,p)$ for all $q\in (1,\infty)$ and it immediately follows that
\[\rank R_q(G,p)\leq d|V(G)|-\dim\T_q(G,p)\]
with equality if and only if $(G,p)$ is infinitesimally rigid.

\begin{defn} 
A finite bar-joint framework $(G,p)$ is  {\em regular}  in $(\mathbb{R}^d,\|\cdot\|_q)$ if the function 
\[P(G,\bR^d)\to\mathbb{R}, \,\,\,\,\,\, x\mapsto \rank R_q(G,x)\]
achieves its maximum value at $p$.
\end{defn}

The equivalence of continuous and infinitesimal rigidity for regular finite bar-joint frameworks in Euclidean space was established by Asimow and Roth \cite{asi-rot,asi-rot-2}.

\begin{thm}[Asimow-Roth, 1978/9]\label{AsimowRoth}
If $(G,p)$ is a finite bar-joint framework in Euclidean space $(\bR^d,\|\cdot\|_2)$ 
then the following statements are equivalent. 
\begin{enumerate}[(i)]
\item $(G,p)$ is continuously rigid and regular.
\item $(G,p)$ is infinitesimally rigid.
\end{enumerate}
\end{thm}

In Section \ref{AREquivalence} we will extend this result to finite bar-joint frameworks in the non-Euclidean spaces $(\bR^d,\|\cdot\|_q)$ for $q\in (1,\infty)$.

We now formalise our meaning of a generic finite bar-joint framework in $(\mathbb{R}^d,\|\cdot\|_q)$ for $q\in(1,\infty)$.
The complete graph on the vertices $V(G)$ will be denoted $K_{V(G)}$. 

\begin{defn} 
A finite bar-joint framework $(G,p)$ is  {\em generic}   in $(\mathbb{R}^d,\|\cdot\|_q)$ if  
$p\in P(K_{V(G)},\bR^d)$ and every subframework of $(K_{V(G)},p)$ is regular.
\end{defn}

If $(G,p)$ is a finite bar-joint framework then $p$ will frequently be identified with a vector $(p_{v_1}, p_{v_2},\ldots,p_{v_{n}})\in \bR^{d|V(G)|}$ with respect to some fixed ordering of the vertices  $V(G)=\{v_1,v_2,\ldots,v_{n}\}$.
In particular, the collection of all generic placements of $G$ in $(\mathbb{R}^d,\|\cdot\|_q)$ is identified with a subset of $\bR^{d|V(G)|}$. 

\begin{lem} 
\label{OpenDense}
Let $G$ be a finite simple graph and let $q\in (1,\infty)$.
Then the set of  generic placements of $G$  in $(\mathbb{R}^d,\|\cdot\|_q)$ is an open and dense subset of 
$\mathbb{R}^{d|V(G)|}$.
\end{lem}

\proof
The set of regular placements of $G$ is an open set since the rank function is lower semi-continuous and the matrix-valued function $x\mapsto R_q(G,x)$ is continuous.

Let $\V_{nr}(G)$ denote the set of all non-regular placements of $G$ and let $\V(G)$ be  the variety
\[\V(G):=\left\{x\in\mathbb{R}^{d|V(G)|}: \prod_{vw\in E(G)}\,\prod_{i=1}^d \, (x_{v,i}-x_{w,i})=0\right\}\]
If $p\in \V_{nr}(G)\backslash \V(G)$ then there exists a neighbourhood $U$ of $p$ such that 
\[\V_{nr}(G)\cap U = \{x\in U: \phi_1(x)=\cdots=\phi_m(x)=0\}\]
where $\phi_1(x),\ldots,\phi_m(x)$ are the minors of $R_q(G,x)$ which correspond to its largest square submatrices.
The entries of  $R_q(G,x)$ when viewed as functions of $x$ 
are real analytic at all points in the complement of $\V(G)$ and so in particular we may assume that $\phi_1,\ldots,\phi_m$ are real analytic on $U$.
This shows that $\V_{nr}(G)\backslash\V(G)$ is a real analytic set in $\bR^{d|V(G)|}$ and so the set of regular placements of $G$ is a dense subset of $\bR^{d|V(G)|}$. 

Finally,  the set of generic placements of $G$ is obtained as a finite intersection of open and dense sets. 
\endproof

Note that the infinitesimal flexibility dimension $\dim_{\rm fl}(G,p)$ is constant on the set of generic placements of $G$
in $(\mathbb{R}^d,\|\cdot\|_q)$.  

\begin{defn} Let $G$ be a finite simple graph.
The {\em infinitesimal flexibility dimension of $G$}  in  $(\mathbb{R}^d,\|\cdot\|_q)$ is 
\[
\dim_{\rm fl}(G):= \dim_{d,q}(G) := \dim_{\rm fl}(G,p)=\dim \F_q(G,p)/ \T_q(G,p).
\] 
where $p$ is any generic placement of $G$.
\end{defn}

\begin{eg}
Let $(K_3,p)$ be the bar-joint framework in $(\mathbb{R}^2,\|\cdot\|_3)$ illustrated in Figure \ref{f:non-Euclideanflexes} where 
$K_3$ is the complete graph on the vertices $v,w,o$ and  \[p_o=(0,0),\,\,\,\,\,\, p_v=(-\sqrt{3},1),\,\,\,\,\,\,p_w=(\sqrt{3},1)\]
A non-trivial infinitesimal flex $u=(u_o,u_v,u_w)\in\F_3(K_3,p)$ is given by 
\[u_o=(0,0), \,\,\,\,\,\,u_v=(-\frac{1}{3},-1),\,\,\,\,\,u_w=(-\frac{1}{3},1)\]
and is  indicated by the dotted arrows which are tangent to a sphere in $(\mathbb{R}^2,\|\cdot\|_3)$.
The rigidity matrix shows that this is a generic placement and so $\dim_{2,3}  K_3 =1$. 

\begin{figure}[h]
\centering
  \begin{tabular}{  c   }

\begin{minipage}{.42\textwidth}
   \includegraphics[scale=0.18]{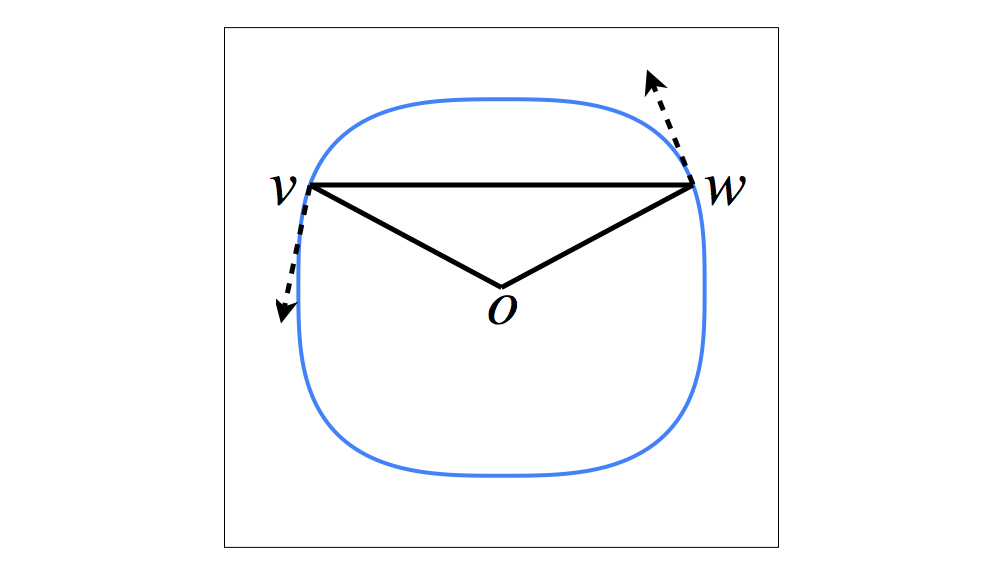}
\end{minipage}

\begin{minipage}{.7\textwidth}
\small{$  \kbordermatrix{
& p_{v,1} & p_{v,2} && p_{w,1} & p_{w,2} &  &  p_{o,1} & p_{o,2} \\
vw& -12  & 0 & \vrule & 12 & 0 & \vrule & 0 & 0 \\
vo& -3 & 1 & \vrule & 0 & 0 & \vrule &  3 & -1\\
ow& 0 & 0 & \vrule & 3 & 1 & \vrule & -3 & -1
}
$}
\end{minipage}

 \end{tabular}
\caption{A non-trivial infinitesimal flex and rigidity matrix for a generic placement of $K_3$ in $(\mathbb{R}^2,\|\cdot\|_3)$.}
\label{f:non-Euclideanflexes}
\end{figure}
\end{eg}

If $G$ has a (minimally) infinitesimally rigid placement then all generic placements of $G$ are (minimally) infinitesimally rigid in $(\mathbb{R}^d,\|\cdot\|_q)$.

\begin{defn} A finite simple graph $G$ is {\em (minimally) rigid} in  $(\mathbb{R}^d,\|\cdot\|_q)$ if 
the generic placements of $G$  are (minimally) infinitesimally rigid.
\end{defn}

One can readily verify that the complete graph $K_{d+1}$ on $d+1$ vertices satisfies 
$\dim_{d,2}(K_{d+1}) =0$ and that $K_{d+1}$ is minimally rigid for $\bR^d$ with the Euclidean norm. Also, in $d$ dimensions we have $\dim_{d,q}(K_{2d}) =0$, with   minimal rigidity, for each of the non-Euclidean $q$-norms.

As noted in Alexandrov \cite{alex} (page 125) Cauchy essentially proved \cite{cau} that if all faces of a closed convex polyhedron are infinitesimally rigid in the Euclidean space $\bR^3$ then the polyhedron itself is infinitesimally rigid.
The following theorem is occasionally referred to as the generic form
of Cauchy's rigidity theorem for convex polyhedra.

\begin{thm}[Cauchy, 1905]
\label{t:GenericCauchy}
Let $G$ be the edge graph of a convex polyhedron with triangular faces. Then 
$G$ is rigid in  $(\mathbb{R}^3,\|\cdot\|_2)$.
\end{thm}

In Section \ref{FiniteSimplicial} we extend Cauchy's rigidity theorem to finite non-Euclidean bar-joint frameworks and
in Section \ref{sub:SimplicialGraphs} we obtain a variant of Cauchy's rigidity theorem for a class of infinite
triangulated planar graphs.

\subsection{Sparsity and rigidity.} 
We recall the following classes of multi-graphs.
\begin{defn}
Let $k,l\in\mathbb{N}$. A multi-graph $G$ is 
\begin{enumerate}
\item {\em $(k,l)$-sparse} if $|E(H)|\leq k|V(H)|-l$ for each subgraph $H$ of $G$ which contains at least two vertices. 
\item {\em $(k,l)$-tight} if it is $(k,l)$-sparse and $|E(G)|=k|V(G)|-l$.
\end{enumerate}
\end{defn}

Our main interests are in the classes of simple $(2,2)$-sparse and $(2,3)$-sparse graphs and the class of  $(k,k)$-sparse multi-graphs for  $k\geq 2$.

\begin{eg}
 The complete graph $K_n$ is $(k,k)$-sparse for $1\leq n \leq 2k$, $(k,k)$-tight for $n\in\{1,2k\}$ and fails to be $(k,k)$-sparse for $n>2k$. Also, $K_2$ and $K_3$ are $(2,3)$-tight  while $K_n$ fails to be $(2,3)$-sparse for $n\geq 4$.
\end{eg}

Laman's theorem  (\cite{Lam}) provides a combinatorial characterisation of the class of finite simple graphs  which are rigid in the Euclidean plane and can be restated as follows.

\begin{thm}[Laman, 1970]\label{Laman}
If $G$ is a finite simple graph then the following statements are equivalent.
\begin{enumerate}[(i)]
\item $G$ is rigid in  $(\mathbb{R}^2,\|\cdot\|_2)$.
\item $G$ contains a $(2,3)$-tight spanning subgraph.
\end{enumerate}
\end{thm}

In particular, a generic bar-joint framework $(G,p)$ is minimally infinitesimally rigid in $(\mathbb{R}^2,\|\cdot\|_2)$ if and only if $G$ is $(2,3)$-tight.
In \cite{kit-pow} an analogue of Laman's theorem was obtained for bar-joint frameworks in the non-Euclidean spaces $(\mathbb{R}^2,\|\cdot\|_q)$.
We can restate this result as follows.

\begin{thm}[\cite{kit-pow}]\label{qNormLaman}
If $G$ is a finite simple graph  and  $q\in (1,2)\cup(2,\infty)$ then the following statements are equivalent. 
\begin{enumerate}[(i)]
\item $G$ is rigid in  $(\mathbb{R}^2,\|\cdot\|_q)$.
\item $G$ contains a $(2,2)$-tight spanning subgraph.
\end{enumerate}
\end{thm}

In Section \ref{LamanSection} we will extend Laman's combinatorial characterisation and its non-Euclidean analogue
to countable graphs in $(\mathbb{R}^2,\|\cdot\|_q)$ for all $q\in (1,\infty)$. 
The class of $(k,k)$-sparse graphs will be used in Section \ref{Body-bar} to obtain analogous results 
for generic finite and countable body-bar frameworks in $(\bR^d,\|\cdot \|_q)$ for all $d\geq 2$.

To end this section we collect the following facts about sparse graphs which will play a key role in later sections.

\begin{lem}
\label{SparseLemma3}
Let $G$ be a $(k,l)$-sparse multi-graph and let $H_1$ and $H_2$ be subgraphs of $G$  both of which  are $(k,l)$-tight. Suppose that one of the following conditions holds. 
\begin{enumerate}[(a)]
\item $k=2$, $l=3$ and $H_1\cap H_2$ contains at least two vertices, or,
\item $k=l$ and $H_1\cap H_2$ contains at least one vertex.
\end{enumerate} 
Then $H_1\cup H_2$ and $H_1\cap H_2$ are $(k,l)$-tight.
\end{lem}

\proof
Applying the sparsity counts for the subgraphs $H_1\cap H_2$ and $H_1\cup H_2$ we have
\begin{eqnarray*}
k|V(H_1\cup H_2)| 
&\leq& k|V(H_1\cup H_2)|  + k|V(H_1\cap H_2)| - l - |E(H_1\cap H_2)| \\
&=& (k|V(H_1)|-l)+(k|V(H_2)|-l)- |E(H_1\cap H_2)| +l\\
&=& |E(H_1)|+|E(H_2)|-|E(H_1\cap H_2)|+l \\
&=& |E(H_1\cup H_2)| +l
\end{eqnarray*}
Thus $H_1\cup H_2$ is $(k,l)$-tight and a similar argument can be applied for the intersection $H_1\cap H_2$.
\endproof

\begin{lem}
\label{SparseLemma2}
Let $G$ be a $(k,l)$-sparse multi-graph containing vertices $v,w\in V(G)$ with $vw\notin E(G)$ and let $G'=G\cup\{vw\}$. 
Then exactly one of the following conditions must hold.
\begin{enumerate}[(i)]
\item $G'$ is $(k,l)$-sparse, or,
\item there exists a $(k,l)$-tight subgraph of $G$ which contains both $v$ and $w$.
\end{enumerate}
\end{lem}

\proof
If $G'$ is not $(k,l)$-sparse then there exists a subgraph $H'$ of $G'$ which fails the sparsity count.
Now $H'\backslash\{vw\}$ is a $(k,l)$-tight subgraph of $G$ which contains both $v$ and $w$.
Conversely, if $H$ is a $(k,l)$-tight subgraph of $G$ which contains both $v$ and $w$ then $H\cup\{vw\}$ is a subgraph of $G'$ which fails the sparsity count.
\endproof

\begin{prop}
\label{SparseLemma1}
Let $G$ be a $(k,l)$-sparse multi-graph. 
Suppose that one of the following conditions holds. 
\begin{enumerate}[(a)]
\item $k=2$, $l=3$ and $G$ contains at least two vertices, or,
\item $k=l$ and $G$ contains at least $2k$ vertices.
\end{enumerate} 
Then the following statements are equivalent.
\begin{enumerate}[(i)]
\item $G$ is $(k,l)$-tight.
\item Every pair of distinct vertices $v,w\in V(G)$ with $vw\notin E(G)$ is contained in a $(k,l)$-tight subgraph of $G$.
\end{enumerate}
Moreover, $G$ is a spanning subgraph of a $(k,l)$-tight graph $G'$ obtained by adjoining edges
of the form $vw$ to $E(G)$ where $v$ and $w$ are distinct vertices of $V(G)$.  
\end{prop}

\proof
The implication $(i)\Rightarrow (ii)$ is immediate and so we prove $(ii)\Rightarrow (i)$.
If $G$ is a complete graph then since $G$ is $(k,l)$-sparse we have $G=K_2$ or $G=K_3$ when $(a)$ holds and $G=K_{2k}$ when $(b)$ holds. 
Since $K_{2k}$ is $(k,k)$-tight and $K_2$ and $K_3$ are both $(2,3)$-tight,
the conclusion of the lemma is clear in these cases.

We now assume that $G$ is not a complete graph.
Suppose that every pair of distinct vertices $v,w\in V(G)$ with $vw\notin E(G)$ is contained in a $(k,l)$-tight subgraph $H_{v,w}$.
Let $H$ be the subgraph of $G$ which is the union of all subgraphs $H_{v,w}$.

Consider the case when $k=2$ and $l=3$.
By Lemma \ref{SparseLemma3}, $H$ can be decomposed as a union of $(2,3)$-tight subgraphs $H_1,\ldots,H_n$ such that no pair $H_i$, $H_j$ shares more than one vertex.
 Note that \[|E(H)| = \sum_{j=1}^n |E(H_j)| =\sum_{j=1}^n (2|V(H_j)|-3) \geq 2|V(H)|-3n\]
Each of the subgraphs $H_{v,w}$ must contain at least four vertices since it is $(2,3)$-tight, contains the vertices $v$ and $w$ but does not contain the edge $vw$. 
Suppose $H_i\not= H_j$. Since $H_i$ and $H_j$ intersect in at most one vertex and each contains at least four vertices there exists distinct  $v_i,w_i\in V(H_i)\backslash  V(H_j)$ and 
$v_j,w_j\in V(H_j)\backslash  V(H_i)$.
By our definition of $H$ the four interconnecting edges $v_iv_j,v_iw_j,w_iv_j$ and $w_iw_j$ must  belong to $G$.
It follows that there are at least $4\left(\frac{n(n-1)}{2}\right)$ distinct interconnecting edges between the subgraphs $H_1,\ldots,H_n$.
Also if $v\in V(G)\backslash V(H)$ then again by the definition of $H$ we must have $vw\in E(G)\backslash E(H)$ for all $w\in V(G)\backslash \{v\}$.
Thus \[|E(G)\backslash E(H)|\geq (|V(G)|-1)|V(G)\backslash V(H)|+4\left(\frac{n(n-1)}{2}\right)\]
For $n\geq 1$  we have 
\begin{eqnarray*}
|E(G)| &\geq& |E(H)|+2n(n-1) +  (|V(G)|-1)|V(G)\backslash V(H)|\\
&\geq& 2|V(H)|-3n+2n(n-1) +(|V(G)|-1)|V(G)\backslash V(H)| \\
&=& 2|V(G)|-3n+2n(n-1) +(|V(G)|-3)|V(G)\backslash V(H)| \\
&\geq&  2|V(G)|-3
\end{eqnarray*}
Thus $G$ is $(2,3)$-tight.

Now consider the case when $k=l$.
By Lemma \ref{SparseLemma3}, $H$ can be decomposed as a vertex-disjoint union of $(k,k)$-tight subgraphs $H_1,\ldots,H_n$. 
Note that \[|E(H)| = \sum_{j=1}^n |E(H_j)| =\sum_{j=1}^n (k|V(H_j)|-k) = k|V(H)|-kn\]
Each of the subgraphs $H_{v,w}$ must contain at least $2k+1$ vertices since it is $(k,k)$-tight, contains the vertices $v$ and $w$ but does not contain the edge $vw$. 
It follows that there are at least $\tau(k,n):=(2k+1)^2\left(\frac{n(n-1)}{2}\right)$ distinct interconnecting edges between the subgraphs $H_1,\ldots,H_n$.

For $n\geq 1$  we have 
\begin{eqnarray*}
|E(G)| &\geq& |E(H)|+\tau(k,n)+ (|V(G)|-1)|V(G)\backslash V(H)| \\
&=& k|V(H)|-kn+\tau(k,n)+(|V(G)|-1)|V(G)\backslash V(H)| \\
&=& k|V(G)|-kn +\tau(k,n)+(|V(G)|-(k+1))|V(G)\backslash V(H)|\\
&\geq&  k|V(G)|-k
\end{eqnarray*}
Thus $G$ is $(k,k)$-tight. This completes the proof that $(ii)\Rightarrow (i)$.

To prove the final statement suppose that $G$ is not $(k,l)$-tight.
By the above arguments there exists a pair of distinct vertices $v,w\in V(G)$ with $vw\notin E(G)$ such that no $(k,l)$-tight subgraph of $G$ contains both $v$ and $w$.
Let $G_1=G\cup\{vw\}$.
By Lemma \ref{SparseLemma2}, $G_1$ is also $(k,l)$-sparse and so we can repeat this argument with $G_1$ in place of $G$. In this way we construct  $G_1,G_2,G_3,\ldots$ by adjoining edges to $G$.
This procedure must terminate after finitely many steps and so $G_{n}$ is $(k,l)$-tight for some $n\in\mathbb{N}$.  
\endproof

\subsection{Remarks.}
Accounts of the foundations of geometric rigidity theory can be found in
Alexandrov \cite{alex}, Graver \cite{gra}, Graver, Servatius and Servatius \cite{gra-ser-ser} and Whiteley \cite{whi-1}. Also \cite{gra-ser-ser} has a comprehensive guide to the literature up to 1993.  
Gluck \cite{glu} proved that every infinitesimally rigid finite bar-joint framework in the Euclidean space $\bR^3$ is continuously rigid and showed that for the usual algebraic notion of a generic framework $(G,p)$ (the vertex coordinates of the $p_i$ should be algebraically independent over $\bQ$) there is an equivalence  between infinitesimal rigidity, continuous rigidity and what might be called perturbational rigidity (all sufficiently nearby equivalent frameworks are congruent). 
The influential papers of Asimow and Roth
introduced regular frameworks as a more appropriate form of genericity. We have followed Graver \cite{gra} in our definition of ``generic", requiring that all frameworks supported by the $p_i$ should be regular. 


\section{Rigidity of countable graphs}
\label{CountableGraphs}
In this section we consider bar-joint frameworks $(G,p)$ in which $G$ is a countably infinite graph.
Our main result is to establish the general principle that infinitesimal rigidity is equivalent to 
\emph{local relative rigidity} in the sense that every finite subframework is rigid relative to some finite containing superframework. It follows that every countable bar-joint framework which can be expressed as a union of an increasing sequence of  infinitesimally rigid finite frameworks is itself infinitesimally rigid. In general this implication is only one-way and we illustrate this fact with an example  in $(\mathbb{R}^3,\|\cdot\|_2)$. However, we prove in Theorem \ref{InfiniteLaman} that the two notions are  equivalent for bar-joint frameworks in $(\mathbb{R}^2,\|\cdot\|_q)$ for all $q\in(1,\infty)$.

\subsection{Relative infinitesimal rigidity.}
We first define a weak form of induced rigidity for a subgraph and prove that in two dimensional $\ell^q$ spaces relative infinitesimal rigidity is equivalent to the existence of a rigid containing framework.

\begin{defn} Let $(G,p)$ be a bar-joint framework in a normed  space $(X,\|\cdot\|)$.
\begin{enumerate}
\item
A subframework $(H,p)$ is {\em relatively infinitesimally rigid} in $(G,p)$
if there is no non-trivial infinitesimal flex of $(H,p)$ which extends to an infinitesimal flex of $(G,p)$.
\item
A  subframework $(H,p)$ has an {\em infinitesimally rigid container} in $(G,p)$
if there exists an infinitesimally rigid subframework of $(G,p)$ which contains $(H,p)$ as a subframework.
\end{enumerate}
\end{defn}

If the complete bar-joint framework $(K_{V(H)},p)$ is infinitesimally rigid in $(X,\|\cdot\|)$ 
then relative infinitesimal rigidity is characterised by the property
\[\F(G,p) = \F(G\cup K_{V(H)},p)\]
It follows that relative infinitesimal rigidity is a generic property 
for bar-joint frameworks in $(\mathbb{R}^d,\|\cdot\|_q)$ for all $q\in (1,\infty)$ since if $p$ and $\tilde{p}$ are two generic placements of $G$ then
\[\F_q(G,\tilde{p})\cong \F_q(G,p) = \F_q(G\cup K_{V(H)},p) \cong \F_q(G\cup K_{V(H)},\tilde{p})\]
To ensure that $(K_{V(H)},p)$ is infinitesimally rigid in the Euclidean case we require that $H$ contains at least $d+1$ vertices while in the non-Euclidean cases $H$ must contain at least $2d$ vertices.
We will say that a subgraph $H$ is {\em relatively rigid} in $G$ with respect to $(\mathbb{R}^d,\|\cdot\|_q)$ if the subframework $(H,p)$ is relatively infinitesimally rigid in $(G,p)$ for some (and hence every) generic placement of $G$. Note that the existence of an infinitesimally rigid container is also a generic property for bar-joint frameworks in $(\bR^d,\|\cdot\|_q)$.
We will say that a subgraph $H$ has a  {\em rigid container} in $G$ with respect to $(\mathbb{R}^d,\|\cdot\|_q)$ if there exists a rigid subgraph of $G$ which contains $H$.

If $(H,p)$ has an infinitesimally rigid container in $(G,p)$ then $(H,p)$ is relatively infinitesimally
rigid in $(G,p)$. The converse statement is not true in general as the following example shows.

\begin{eg}
Figure \ref{Banana} illustrates a generic bar-joint framework $(G,p)$ in $(\bR^3,\|\cdot\|_2)$ 
with subframework $(H,p)$ indicated by the shaded region. Note that $(H,p)$ is relatively infinitesimally rigid in $(G,p)$ but does not have an infinitesimally rigid container in $(G,p)$.
\end{eg}
 
\begin{figure}[h]
\centering
    \begin{tikzpicture}[scale=0.24]
 
  \clip (-7.5,-4.2) rectangle (7.5cm,4.2cm); 
  
  \coordinate (A1) at (-6.5,0);
  \coordinate (A2) at (-4.5,0.7);
  \coordinate (A3) at (-4.5,-0.7);

  \coordinate (A4) at (6.5,0);
  \coordinate (A5) at (4.5,0.7);
  \coordinate (A6) at (4.5,-0.7);

  \coordinate (A7) at (0,4);
  \coordinate (A8) at (0,-4);

  \coordinate (A9) at (0,2.4);
  \coordinate (A10) at (-0.9,1.2);
  \coordinate (A11) at (0.9,1.2);

  \draw[fill=gray!12] (A7) -- (A1) -- (A2) -- (A3) -- cycle;
  \draw[fill=gray!12] (A3) -- (A7) -- (A2) -- cycle;
 
  \draw[fill=gray!12] (A8) -- (A1) -- (A2) -- (A3) -- cycle;
  \draw[fill=gray!12] (A3) -- (A8) -- (A2) -- cycle;

  \draw (A7) -- (A1) -- (A2) -- (A3) -- cycle;
  \draw (A3) -- (A7) -- (A2) -- cycle;
 
  \draw (A8) -- (A1) -- (A2) -- (A3) -- cycle;
  \draw (A3) -- (A8) -- (A2) -- cycle;
  \draw (A1) -- (A2) -- (A3) -- cycle;

  \draw[fill=gray!12] (A7) -- (A4) -- (A5) -- (A6) -- cycle;
  \draw[fill=gray!12] (A6) -- (A7) -- (A5) -- cycle;
  \draw (A6) -- (A4);
  \draw[fill=gray!12] (A8) -- (A4) -- (A5) -- (A6) -- cycle;
  \draw[fill=gray!12] (A6) -- (A8) -- (A5) -- cycle;

  \draw (A7) -- (A4) -- (A5) -- (A6) -- cycle;
  \draw (A6) -- (A7) -- (A5) -- cycle;
  \draw (A6) -- (A4);
  \draw (A8) -- (A4) -- (A5) -- (A6) -- cycle;
  \draw (A6) -- (A8) -- (A5) -- cycle;

  \draw (A3) -- (A9) -- (A10) -- (A11) -- (A3) -- (A10);
  \draw (A9) -- (A11);
  \draw (A6) -- (A9) -- (A10) -- (A11) -- (A6) -- (A10);

  \node[draw,circle,inner sep=1.2pt,fill] at (A1) {};
  \node[draw,circle,inner sep=1.2pt,fill] at (A2) {};
  \node[draw,circle,inner sep=1.2pt,fill] at (A3) {};
  \node[draw,circle,inner sep=1.2pt,fill] at (A4) {};
  \node[draw,circle,inner sep=1.2pt,fill] at (A5) {};
  \node[draw,circle,inner sep=1.2pt,fill] at (A6) {};
  \node[draw,circle,inner sep=1.2pt,fill] at (A7) {};
  \node[draw,circle,inner sep=1.2pt,fill] at (A8) {};
  \node[draw,circle,inner sep=1.2pt,fill] at (A9) {};
  \node[draw,circle,inner sep=1.2pt,fill] at (A10) {};
  \node[draw,circle,inner sep=1.2pt,fill] at (A11) {};

\end{tikzpicture}

  \caption{An example of a relatively rigid subgraph in the Euclidean space $\mathbb{R}^3$
which does not have a rigid container.}
\label{Banana}  
\end{figure}
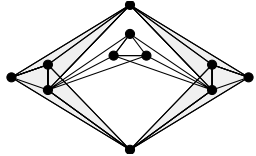

In the following we will say that a finite simple graph $G$ is {\em independent} in $(\mathbb{R}^d,\|\cdot\|_q)$ if the rigidity matrix $R_q(G,p)$ is independent for some (and hence every) generic placement $p:V(G)\to\mathbb{R}^d$. 

\begin{prop}
\label{IndepLemma}
Let $G$ be a finite simple graph and let $q\in(1,\infty)$.
Suppose that one of the following conditions holds.
\begin{enumerate}[(a)]
\item $q=2$, $l=3$ and $G$ contains at least two vertices, or,
\item $q\not=2$, $l=2$ and $G$ contains at least four vertices.
\end{enumerate}
Then the following statements are equivalent.
\begin{enumerate}[(i)]
\item $G$ is independent in $(\mathbb{R}^2,\|\cdot\|_q)$.
\item $G$ is $(2,l)$-sparse.
\end{enumerate}
\end{prop}

\proof
Let $p:V(G)\to \mathbb{R}^2$ be a generic placement of $G$ in $(\mathbb{R}^2,\|\cdot\|_q)$.
If $G$ is independent and  $H$ is a subgraph of $G$ then
  $|E(H)|=\rank R_q(H,p)\leq2|V(H)|-l$.
We conclude that $G$ is $(2,l)$-sparse.

Conversely, if $G$ is $(2,l)$-sparse then, by Proposition \ref{SparseLemma1}, $G$ is a subgraph of some $(2,l)$-tight graph $G'$ with $V(G)=V(G')$. By Laman's theorem and its analogue for the non-Euclidean case (Theorems \ref{Laman} and \ref{qNormLaman}), $(G',p)$ is minimally infinitesimally rigid and so $G$ is independent. 
\endproof

We now show that relative infinitesimal rigidity does imply the existence of an infinitesimally rigid container for generic bar-joint frameworks in $(\mathbb{R}^2,\|\cdot\|_q)$ for all $q\in (1,\infty)$.

\begin{thm}
\label{RelRigidProp}
Let $G$ be a finite simple graph and let $H$ be a subgraph of $G$.
Suppose that $q\in(1,\infty)$ and that one of the following conditions holds.
\begin{enumerate}[(a)]
\item $q=2$ and $H$ contains at least two vertices, or,
\item $q\not=2$ and $H$ contains at least four vertices.
\end{enumerate}
Then the following statements are equivalent.
\begin{enumerate}[(i)]
\item $H$ is relatively  rigid in $G$ with respect to $(\mathbb{R}^2,\|\cdot\|_q)$.
\item $H$ has a rigid container in $G$ with respect to $(\mathbb{R}^2,\|\cdot\|_q)$.
\end{enumerate}
\end{thm}

\proof
$(i)\Rightarrow(ii)$
Consider first the case when $G$ is independent with respect to $(\mathbb{R}^2,\|\cdot\|_q)$.  
By Proposition \ref{IndepLemma}, $G$ is $(2,l)$-sparse (where $l=3$ when $q=2$ and $l=2$ when $q\not=2$).
Since $K_{V(H)}$ is rigid in $(\mathbb{R}^2,\|\cdot\|_q)$ the relative rigidity property implies that 
\[\F_q(G,p)=\F_q(G\cup K_{V(H)},p)\] for every generic placement $p\in\P(G)$.
It follows that if $v,w\in V(H)$ and $vw\notin E(G)$ then $G\cup\{vw\}$ is dependent.
By Proposition \ref{IndepLemma}, $G\cup\{vw\}$ is not $(2,l)$-sparse.
Thus by Lemma \ref{SparseLemma2} there exists a $(2,l)$-tight subgraph $H_{v,w}$ of $G$ with $v,w\in V(H_{v,w})$.
By Theorems \ref{Laman} and \ref{qNormLaman}, $H_{v,w}$ is rigid $(\mathbb{R}^2,\|\cdot\|_q)$.
Let $H'$ be the subgraph of $G$ which consists of $H$ and the subgraphs $H_{v,w}$.
Then $H'$ is  rigid in $(\mathbb{R}^2,\|\cdot\|_q)$ and so
$H'$ is a rigid container for $H$ in $G$.

If $G$ is dependent then let $p:V(G)\to\mathbb{R}^2$ be a generic placement of $G$.
There exists an edge $vw\in E(G)$ such that 
 \[\ker R_q(G,p)=\ker R_q(G\backslash vw,p)\]
Let $G_{1}=G\backslash vw$ and $H_{1}=H\backslash vw$ and note that $H_{1}$ is relatively  rigid in $G_{1}$. Continuing to remove edges in this way we arrive after finitely many iterations at  subgraphs $H_{n}$ and $G_{n}$ such that $V(H_n)=V(H)$, $H_{n}$ is relatively  rigid in $G_{n}$ and $G_{n}$ is independent.
By the above argument there exists a rigid container $H'_{n}$ for $H_{n}$ in  $G_{n}$.
Now $H' = H'_{n} \cup H$ is a rigid container for $H$ in $G$. 

$(ii)\Rightarrow (i)$ Let $H'$ be a rigid container for $H$ in $G$ and let $p:V(G)\to\mathbb{R}^2$ be a generic placement of $G$ in $(\mathbb{R}^2,\|\cdot\|_q)$.
Then no non-trivial infinitesimal flex of $(H,p)$ can be extended to an infinitesimal flex of $(H',p)$ and so the result follows. \endproof

\subsection{Flex cancellation and relatively rigid towers.}
A tower of bar-joint frameworks in a normed vector space $(X,\|\cdot\|)$ is a sequence $\{(G_k,p_k):k\in\mathbb{N}\}$ of finite bar-joint frameworks in $(X,\|\cdot\|)$ such that $(G_k,p_k)$ is a subframework of $(G_{k+1},p_{k+1})$ for each $k\in\mathbb{N}$.
The linear maps  \[\rho_{j,k}:\F(G_{k},p_{k}) \to \F(G_j,p_j)\] 
defined for all $j\leq k$ by the restriction of flexes determine an inverse system $(\F(G_k,p_k),\rho_{j,k})$
with associated vector space inverse limit $\varprojlim \F(G_k,p_k)$.

\begin{defn}
A tower of bar-joint frameworks $\{(G_k,p_k):k\in\mathbb{N}\}$  has the {\em flex cancellation property} if for each $k\in\mathbb{N}$ and any non-trivial infinitesimal flex $u_k$ of $(G_{k},p_k)$ there is an $m>k$ such that $u_k$ does not extend to an infinitesimal flex of $(G_m,p_m)$. 
\end{defn}

If a bar-joint framework $(G_m,p_m)$ in a tower $\{(G_k,p_k):k\in\mathbb{N}\}$ has a non-trivial infinitesimal flex $u_m:V(G_m)\to X$ which can be extended to every containing framework in the tower then we 
call $u_m$ an enduring infinitesimal flex for the tower. 

\begin{lem}
\label{FlexProp}
Let $\{(G_k,p_k):k\in\mathbb{N}\}$ be a tower of bar-joint frameworks in a finite dimensional  normed  space $(X,\|\cdot\|)$ 
and let $u_1$ be an infinitesimal flex of $(G_1,p_1)$ which is an enduring flex for the tower. 
Then there exists a sequence $\{u_k\}_{k=1}^\infty$ such that, for each $k\in\mathbb{N}$, $u_{k}$ is an infinitesimal flex of $(G_k,p_k)$ and $u_{k+1}$ is an extension of $u_k$.
\end{lem}

\begin{proof}
If $u_1$ is a trivial infinitesimal flex of $(G_1,p_1)$ then there exists $\gamma\in \mathcal{T}(X)$ such that $u_1(v)=\gamma(p_1(v))$ for all $v\in V(G_1)$.
Define $u_k(v)=\gamma(p_k(v))$ for all $v\in V(G_k)$ and all $k\in\mathbb{N}$. Then $\{u_k\}_{k=1}^\infty$ is the required sequence.
Now suppose that $u_1$ is non-trivial.
Denote by $\mathcal{F}^{(k)}\subset\mathcal{F}(G_k,p_k)$ the vector space of all infinitesimal flexes $u\in\F(G_k,p_k)$ with the property that there exists a scalar $\lambda\in\mathbb{K}$ such that
$u(v)=\lambda u_1(v)$ for all $v\in V(G_1)$. 
Let $\rho_k:\mathcal{F}^{(k)}\to\mathcal{F}^{(2)}$ be the restriction map
and note that since $u_1$ is an enduring flex we have a decreasing chain of non-zero finite dimensional linear spaces
\[\mathcal{F}^{(2)}\supseteq\rho_3(\mathcal{F}^{(3)})\supseteq \rho_{4}(\mathcal{F}^{(4)})\supseteq\rho_{5}(\mathcal{F}^{(5)})\supseteq\cdots\]
Thus there exists $m\in\mathbb{N}$ such that $\rho_n(\mathcal{F}^{(n)})= \rho_{m}(\mathcal{F}^{(m)})$ for all $n>m$.
Since $u_1$ is non-trivial and enduring there is a necessarily non-trivial extension $\tilde{u}_{m}$ say in $\mathcal{F}^{(m)}$. 
Let $u_2$ be the restriction of $\tilde{u}_{m}$ to $(G_2,p_2)$.
Note that $u_2$ is an enduring flex since 
for each $n>m$ we have $u_2\in \rho_{m}(\mathcal{F}^{(m)})= \rho_{n}(\mathcal{F}^{(n)})$. 
Also $u_2$ is an extension of $u_1$. 
An induction argument can now be applied to obtain a sequence of consecutive extensions $u_{k}\in\F(G_{k}, p_{k})$. 
\end{proof}

A bar-joint framework  $(G,p)$ contains a tower $\{(G_k,p_k):k\in\mathbb{N}\}$ if  $(G_k,p_k)$ is a subframework of $(G,p)$ for each $k\in\mathbb{N}$.
A tower in $(G,p)$ is vertex-complete if $V(G)=\cup_{k\in\mathbb{N}} V(G_k)$ and edge-complete if
$E(G)=\cup_{k\in\mathbb{N}} E(G_k)$.
If a tower is edge-complete then the vector space $\F(G,p)$ of infinitesimal flexes is naturally isomorphic to the vector space inverse limit
\[\F(G,p)\cong\varprojlim\, \F(G_k,p_k)\]

\begin{prop}
\label{RelRigidLemma}
Let $(G,p)$ be a countable bar-joint framework in a finite dimensional normed space $(X,\|\cdot\|)$.
If $(G,p)$  is infinitesimally rigid then every edge-complete tower in $(G,p)$  has the flex cancellation property.
\end{prop}

\proof
Suppose there exists an edge-complete tower $\{(G_k,p_k):k\in\mathbb{N}\}$ of finite frameworks in $(G,p)$ which does not have the flex cancellation property. 
Then there exists a non-trivial infinitesimal flex of some $(G_k,p_k)$ which
is an enduring flex for the tower. We may assume without loss of generality that $k=1$.
By Lemma \ref{FlexProp} there is a sequence of infinitesimal flexes $u_1,u_2,u_3,\ldots$ for the chain with each flex extending the preceding flex. 
The tower is edge-complete and so this sequence defines an infinitesimal flex $u$ for $(G,p)$ by setting $u(v)=u_k(v)$ for all $v\in V(G_k)$ and all $k\in\bN$.
Since $u_1$ is a non-trivial infinitesimal flex of $(G_1,p_1)$ the flex $u$ is a non-trivial infinitesimal flex of $(G,p)$.  
\endproof

We can now establish the connection between relative rigidity, flex cancellation and infinitesimal rigidity for countable bar-joint frameworks.

\begin{defn}
A tower of bar-joint frameworks $\{(G_k,p_k):k\in\mathbb{N}\}$ is {\em relatively infinitesimally rigid} if $(G_{k},p_k)$ is relatively infinitesimally rigid in $(G_{k+1},p_{k+1})$ for each $k\in\mathbb{N}$.
\end{defn}

\begin{lem}
\label{RelRigidLemma2}
Let $\{(G_k,p_k):k\in\mathbb{N}\}$ be a framework tower in a finite dimensional normed  space $(X,\|\cdot\|)$.
If $\{(G_k,p_k):k\in\mathbb{N}\}$ has the flex cancellation property then there exists an increasing  sequence $(m_k)_{k=1}^\infty$ of natural numbers such that the tower $\{(G_{m_k},p_{m_k}):k\in\mathbb{N}\}$ is relatively infinitesimally rigid. 
\end{lem}

\proof
Let $\mathcal{F}^{(k)}\subset\mathcal{F}(G_1,p_1)$ denote the set of all infinitesimal flexes of $(G_1,p_1)$ which extend to $(G_{k},p_{k})$ but not $(G_{k+1},p_{k+1})$.
Suppose there exists an increasing sequence $(n_k)_{k=1}^\infty$ of natural numbers such that $ \mathcal{F}^{(n_k)}\not=\emptyset$ for all $k\in\mathbb{N}$.
Choose an element $u_k\in\mathcal{F}^{(n_k)}$ for each $k\in\mathbb{N}$ and note that $\{u_k:k\in\mathbb{N}\}$ is a linearly independent set in $\mathcal{F}(G_1,p_1)$.
Since $\mathcal{F}(G_1,p_1)$ is finite dimensional we have a contradiction. Thus there exists $m_1\in\mathbb{N}$ such that $\mathcal{F}^{(k)}=\emptyset$ for all $k\geq m_1$ and so $(G_1,p_1)$ is relatively infinitesimally rigid in $(G_{m_1},p_{m_1})$. The result now follows by an induction argument. 
\endproof

\begin{thm}
\label{t:IR}
Let $(G,p)$ be a countable bar-joint framework in a finite dimensional real normed linear space $(X,\|\cdot\|)$.
Then the following statements are equivalent.
\begin{enumerate}[(i)]
\item $(G,p)$ is infinitesimally rigid.
\item $(G,p)$ contains a vertex-complete tower   which has the flex cancellation property.
\item $(G,p)$ contains a vertex-complete tower   which is relatively infinitesimally rigid.
\end{enumerate}
\end{thm}

\proof
The implication $(i)\Rightarrow (ii)$ is a consequence of Proposition \ref{RelRigidLemma}. To prove $(ii)\Rightarrow(iii)$ apply Lemma \ref{RelRigidLemma2}.
We now prove $(iii)\Rightarrow(i)$.
Let $\{(G_k,p_k):k\in\mathbb{N}\}$ be a vertex-complete tower in $(G,p)$ which  is relatively infinitesimally rigid and suppose $u$ is a non-trivial infinitesimal flex of $(G,p)$.  
We will construct inductively a sequence $(\gamma_{n})_{n=1}^\infty$ of infinitesimal rigid motions of $X$ and an increasing sequence $(k_n)_{n=1}^\infty$ of natural numbers satisfying
\begin{itemize}
\item $u(v)=\gamma_{n}(p(v))$ for all $v\in V(G_{k_n})$
\item $u(v_{k_{n+1}})\not=\gamma_{n}(p(v_{k_{n+1}}))$ for some $v_{k_{n+1}}\in V(G_{k_{n+1}})$
\end{itemize}
Since the tower $\{(G_k,p_k):k\in\mathbb{N}\}$  is relatively infinitesimally rigid the restriction of $u$ to $(G_k,p_k)$ is  trivial for each $k\in\mathbb{N}$.
Thus there exists $\gamma_{1}\in \mathcal{T}(X)$ such that $u(v)=\gamma_{1}(p(v))$ for all $v\in V(G_1)$. Let $k_1=1$.
Since $u$ is non-trivial and the tower is vertex-complete  there exists $k_2>k_1$ such that 
$u(v_{k_2})\not=\gamma_1(p(v_{k_2}))$ for some $v_{k_2}\in V(G_{k_2})$.
Now the restriction of $u$ to $(G_{k_2},p_{k_2})$ is trivial and so there exists $\gamma_{2}\in \mathcal{T}(X)$ such that $u(v)=\gamma_{2}(p(v))$ for all $v\in V(G_{k_2})$.
In general, given $\gamma_n\in \mathcal{T}(X)$  and $k_n\in\mathbb{N}$  we construct $\gamma_{n+1}$ and $k_{n+1}$ using the same argument. 

Let $s_n=\gamma_{n+1}-\gamma_{n}\in \mathcal{T}(X)$. Then $s_n(p(v))=0$ for all $v\in V(G_{k_n})$ and $s_n(p(v_{k_{n+1}}))\not=0$  for some $v_{k_{n+1}}\in V(G_{k_{n+1}})$.
Thus $\{s_n:n\in\mathbb{N}\}$ is a linearly independent set in $\mathcal{T}(X)$ and since  $\mathcal{T}(X)$ is finite dimensional we have a contradiction.
We conclude that  $(G,p)$ is infinitesimally rigid.
\endproof

Theorem \ref{t:IR} gives useful criteria for the determination of infinitesimal rigidity of a countable framework $(G,p)$. 

\begin{defn}
A countable bar-joint framework $(G,p)$ is {\em sequentially infinitesimally rigid} if there exists a vertex-complete tower of bar-joint frameworks $\{(G_k,p_k):k\in\mathbb{N}\}$ in $(G,p)$ such that $(G_k,p_k)$ is infinitesimally rigid for each $k\in\mathbb{N}$. 
\end{defn}

\begin{cor}
\label{sequential}
Let $(G,p)$ be a countable bar-joint framework in a  finite dimensional normed space $(X,\|\cdot\|)$.
If $(G,p)$ is sequentially infinitesimally rigid then $(G,p)$ is infinitesimally rigid.
\end{cor}

\proof
If there exists a vertex-complete tower  $\{(G_k,p_k):k\in\mathbb{N}\}$ in $(G,p)$ such that $(G_k,p_k)$ is infinitesimally rigid for each $k\in\mathbb{N}$
then this framework tower is relatively infinitesimally rigid. The result now follows from Theorem \ref{t:IR}.
\endproof

\subsection{A characterisation of rigidity for countable graphs.}
\label{LamanSection}
Let $G$ be a countably infinite simple graph and let $q\in (1,\infty)$.

\begin{defn}
A placement $p:V(G)\to \mathbb{R}^d$ is {\em generic} in $(\mathbb{R}^d,\|\cdot\|_q)$ if every finite subframework of $(G,p)$ is generic.
\end{defn}

A tower of graphs is a sequence of finite graphs $\{G_k:k\in\mathbb{N}\}$ such that $G_k$ is a subgraph of $G_{k+1}$ for each $k\in\mathbb{N}$.  
A countable graph $G$ contains a vertex-complete tower $\{G_k:k\in\mathbb{N}\}$ if each $G_k$ is a subgraph of $G$ and $V(G)=\cup_{k\in\mathbb{N}} V(G_k)$.

\begin{prop}
Every countable simple graph $G$ has a generic placement in $(\mathbb{R}^d,\|\cdot\|_q)$ for $q\in(1,\infty)$.
\end{prop}

\proof
Let $\{G_k:k\in\mathbb{N}\}$ be a vertex-complete tower in $G$ and let \[\pi_{j,k}:\bR^{d|V(G_k)|}\to\bR^{d|V(G_j)|}, \,\,\,\,\,\,\, (x_v)_{v\in V(G_k)}\mapsto (x_v)_{v\in V(G_j)}\] be the natural projections whenever $G_j\subseteq G_k$.
By Lemma \ref{OpenDense} the set of generic placements of each $G_k$ is an open and dense subset of
$\bR^{d|V(G_k)|}$. 
It follows by an induction argument that for each $k\in\bN$ there exists  an open ball $B(p_k,r_k)$ in $\bR^{d|V(G_k)|}$ consisting of generic placements of $G_k$ such that $r_{k+1}<r_{k}$ and
the projection $\pi_{k,k+1}(p_{k+1})$ is contained in the open ball $B(p_{k},\frac{r_k}{2^{k}})$. 
 For each $j\in\mathbb{N}$ the sequence $\{\pi_{j,k}(p_k)\}_{k=j}^\infty$  is a Cauchy sequence of points in  $B(p_j,\frac{r_j}{2})\subset \mathbb{R}^{d|V(G_j)|}$
and hence converges to a point in  $B(p_j,r_j)$.
Define $p:V(G)\to\mathbb{R}^d$ by setting \[p(v)=\lim_{k\to\infty, \,k\geq j} p_k(v), \,\,\,\,\,\,\,\, \forall\, v\in V(G_j), \,\,\,\,\,\,\forall \, j\in \mathbb{N}\]
The restriction of $p$ to $V(G_j)$  is a generic placement of $G_j$ for all $j\in\mathbb{N}$ and so $p$ is a generic placement of $G$.
\endproof

We now show that infinitesimal rigidity and sequential infinitesimal rigidity are generic properties for countable bar-joint frameworks in $(\mathbb{R}^d,\|\cdot\|_q)$ for all  $q\in (1,\infty)$.

\begin{prop}
Let $(G,p)$ be a generic countable bar-joint framework in $(\mathbb{R}^d,\|\cdot\|_q)$ where  $q\in (1,\infty)$.
\begin{enumerate}[(i)]
\item The infinitesimal flex dimension $\dim_{\rm fl}(G,p)$ is constant on the set of all generic placements of $G$. 
\item If $(G,p)$ is infinitesimally rigid then every generic placement of $G$   is infinitesimally rigid. 
\item If $(G,p)$ is sequentially infinitesimally rigid then every generic placement of $G$   is sequentially infinitesimally rigid. 
\end{enumerate}
\end{prop}

\proof 
To show $(i)$ choose an edge complete tower $\{(G_k,p):k\in\mathbb{N}\}$ in $(G,p)$. Then 
$\F_q(G,p)$ is isomorphic to the inverse limit of the inverse system 
$(\F_q(G_k,p),\rho_{j,k})$ and similarly $\T_q(G,p)$ is isomorphic to the inverse limit of the inverse system 
$(\T_q(G_k,p),\rho_{j,k})$ where $\rho_{j,k}$ are restriction maps. 
If $p'$ is another generic placement of $G$ then $\F_q(G_k,p)$ is isomorphic to $\F_q(G_k,p')$
and $\T_q(G_k,p)$ is isomorphic to $\T_q(G_k,p')$ for each $k$. 
Hence the corresponding inverse limits are isomorphic,
\[\F_q(G,p)\cong \varprojlim\, \F_q(G_k,p)\cong \varprojlim\, \F_q(G_k,p')\cong\F_q(G,p')\]
\[\T_q(G,p)\cong \varprojlim\, \T_q(G_k,p)\cong \varprojlim\, \T_q(G_k,p')\cong\T_q(G,p')\]
 In particular the 
infinitesimal flex dimensions agree,
\[\dim_{\rm fl}(G,p) =\dim \F_q(G,p)/ \T_q(G,p) =\dim \F_q(G,p')/\T_q(G,p')= \dim_{\rm fl}(G,p')\]

Statement $(ii)$ follows immediately from $(i)$ and 
$(iii)$  holds since  infinitesimal rigidity is a generic property for finite bar-joint frameworks.

\endproof

The infinitesimal flex dimension of a countable graph  and the classes of countable rigid and sequentially rigid graphs are now defined.

\begin{defn} Let $G$ be a countable simple graph. 
\begin{enumerate}[(i)]
\item
$G$ is {\em (minimally) rigid} in $(\mathbb{R}^d,\|\cdot\|_q)$ if the generic placements of $G$  are (minimally) infinitesimally rigid.
\item
$G$ is {\em sequentially rigid} in $(\mathbb{R}^d,\|\cdot\|_q)$  if the generic placements of $G$ are sequentially
infinitesimally rigid.
\item
The {\em infinitesimal flexibility dimension of $G$}  in  $(\mathbb{R}^d,\|\cdot\|_q)$ is 
\[
\dim_{\rm fl}(G):= \dim_{d,q}(G) := \dim_{\rm fl}(G,p)=\dim \F_q(G,p)/ \T_q(G,p).
\] 
where $p$ is any generic placement of $G$.
\end{enumerate}
\end{defn}

The following example demonstrates the non-equivalence of rigidity and sequential rigidity for countable graphs.
The surprising fact that these properties are in fact equivalent in two dimensions is established in Theorem \ref{InfiniteLaman} below. 

\begin{eg}
\label{BananaTower}
Figure \ref{Banana2} illustrates the first three graphs in a tower $\{G_n:n\in\mathbb{N}\}$ in which  
$G_n$  is constructed inductively from a double banana graph $G_1$ by flex cancelling additions of copies of $K_5\backslash e$  (single banana graphs). 
The union $G$ of these graphs is a countable graph whose maximal rigid subgraphs are copies of  $K_5\backslash e$. Thus the generic placements of $G$ are not sequentially infinitesimally rigid.
However the tower is relatively rigid in $(\bR^3,\|\cdot\|_2)$ and so $G$ is rigid.
\end{eg}

\begin{figure}[h]
\centering
  \begin{tabular}{  c   }
  
   \begin{minipage}{.28\textwidth}
    \begin{tikzpicture}[scale=0.27]
 
  \clip (-6.7,-4.2) rectangle (6.8cm,4.2cm); 
  
  \coordinate (A1) at (-6.5,0);
  \coordinate (A2) at (-4.5,0.7);
  \coordinate (A3) at (-4.5,-0.7);

  \coordinate (A4) at (6.5,0);
  \coordinate (A5) at (4.5,0.7);
  \coordinate (A6) at (4.5,-0.7);

  \coordinate (A7) at (0,4);
  \coordinate (A8) at (0,-4);

  \draw[fill=gray!10] (A7) -- (A1) -- (A2) -- (A3) -- cycle;
  \draw[fill=gray!10] (A3) -- (A7) -- (A2) -- cycle;
 
  \draw[fill=gray!10] (A8) -- (A1) -- (A2) -- (A3) -- cycle;
  \draw[fill=gray!10] (A3) -- (A8) -- (A2) -- cycle;

  \draw (A7) -- (A1) -- (A2) -- (A3) -- cycle;
  \draw (A3) -- (A7) -- (A2) -- cycle;
 
  \draw (A8) -- (A1) -- (A2) -- (A3) -- cycle;
  \draw (A3) -- (A8) -- (A2) -- cycle;
  \draw (A1) -- (A2) -- (A3) -- cycle;

  \draw[fill=gray!10] (A7) -- (A4) -- (A5) -- (A6) -- cycle;
  \draw[fill=gray!10] (A6) -- (A7) -- (A5) -- cycle;
  \draw (A6) -- (A4);
  \draw[fill=gray!10] (A8) -- (A4) -- (A5) -- (A6) -- cycle;
  \draw[fill=gray!10] (A6) -- (A8) -- (A5) -- cycle;

  \draw (A7) -- (A4) -- (A5) -- (A6) -- cycle;
  \draw (A6) -- (A7) -- (A5) -- cycle;
  \draw (A6) -- (A4);
  \draw (A8) -- (A4) -- (A5) -- (A6) -- cycle;
  \draw (A6) -- (A8) -- (A5) -- cycle;

  \node[draw,circle,inner sep=1.2pt,fill] at (A1) {};
  \node[draw,circle,inner sep=1.2pt,fill] at (A2) {};
  \node[draw,circle,inner sep=1.2pt,fill] at (A3) {};
  \node[draw,circle,inner sep=1.2pt,fill] at (A4) {};
  \node[draw,circle,inner sep=1.2pt,fill] at (A5) {};
  \node[draw,circle,inner sep=1.2pt,fill] at (A6) {};
  \node[draw,circle,inner sep=1.2pt,fill] at (A7) {};
  \node[draw,circle,inner sep=1.2pt,fill] at (A8) {};

\end{tikzpicture}
\end{minipage}

\begin{minipage}{.28\textwidth}
\begin{tikzpicture}[scale=0.27]
 
  \clip (-7.5,-4.2) rectangle (7cm,4.2cm); 
  
  \coordinate (A1) at (-6.5,0);
  \coordinate (A2) at (-4.5,0.7);
  \coordinate (A3) at (-4.5,-0.7);

  \coordinate (A4) at (6.5,0);
  \coordinate (A5) at (4.5,0.7);
  \coordinate (A6) at (4.5,-0.7);

  \coordinate (A7) at (0,4);
  \coordinate (A8) at (0,-4);

  \coordinate (A9) at (0,2.4);
  \coordinate (A10) at (-0.9,1.2);
  \coordinate (A11) at (0.9,1.2);

  \draw[fill=gray!10] (A7) -- (A1) -- (A2) -- (A3) -- cycle;
  \draw[fill=gray!10] (A3) -- (A7) -- (A2) -- cycle;
 
  \draw[fill=gray!10] (A8) -- (A1) -- (A2) -- (A3) -- cycle;
  \draw[fill=gray!10] (A3) -- (A8) -- (A2) -- cycle;

  \draw (A7) -- (A1) -- (A2) -- (A3) -- cycle;
  \draw (A3) -- (A7) -- (A2) -- cycle;
 
  \draw (A8) -- (A1) -- (A2) -- (A3) -- cycle;
  \draw (A3) -- (A8) -- (A2) -- cycle;
  \draw (A1) -- (A2) -- (A3) -- cycle;

  \draw[fill=gray!10] (A7) -- (A4) -- (A5) -- (A6) -- cycle;
  \draw[fill=gray!10] (A6) -- (A7) -- (A5) -- cycle;
  \draw (A6) -- (A4);
  \draw[fill=gray!10] (A8) -- (A4) -- (A5) -- (A6) -- cycle;
  \draw[fill=gray!10] (A6) -- (A8) -- (A5) -- cycle;

  \draw (A7) -- (A4) -- (A5) -- (A6) -- cycle;
  \draw (A6) -- (A7) -- (A5) -- cycle;
  \draw (A6) -- (A4);
  \draw (A8) -- (A4) -- (A5) -- (A6) -- cycle;
  \draw (A6) -- (A8) -- (A5) -- cycle;

  \draw (A3) -- (A9) -- (A10) -- (A11) -- (A3) -- (A10);
  \draw (A9) -- (A11);
  \draw (A6) -- (A9) -- (A10) -- (A11) -- (A6) -- (A10);

  \node[draw,circle,inner sep=1.2pt,fill] at (A1) {};
  \node[draw,circle,inner sep=1.2pt,fill] at (A2) {};
  \node[draw,circle,inner sep=1.2pt,fill] at (A3) {};
  \node[draw,circle,inner sep=1.2pt,fill] at (A4) {};
  \node[draw,circle,inner sep=1.2pt,fill] at (A5) {};
  \node[draw,circle,inner sep=1.2pt,fill] at (A6) {};
  \node[draw,circle,inner sep=1.2pt,fill] at (A7) {};
  \node[draw,circle,inner sep=1.2pt,fill] at (A8) {};
  \node[draw,circle,inner sep=1.2pt,fill] at (A9) {};
  \node[draw,circle,inner sep=1.2pt,fill] at (A10) {};
  \node[draw,circle,inner sep=1.2pt,fill] at (A11) {};

\end{tikzpicture}
\end{minipage}

\begin{minipage}{.28\textwidth}
\begin{tikzpicture}[scale=0.27]
 
  \clip (-8.4,-4.2) rectangle (7.5cm,4.2cm); 
  
  \coordinate (A1) at (-6.5,0);
  \coordinate (A2) at (-4.5,0.7);
  \coordinate (A3) at (-4.5,-0.7);

  \coordinate (A4) at (6.5,0);
  \coordinate (A5) at (4.5,0.7);
  \coordinate (A6) at (4.5,-0.7);

  \coordinate (A7) at (0,4);
  \coordinate (A8) at (0,-4);

  \coordinate (A9) at (0,2.4);
  \coordinate (A10) at (-0.9,1.2);
  \coordinate (A11) at (0.9,1.2);

  \coordinate (A12) at (-0.8,-1.4);
  \coordinate (A13) at (-1.4,-0.9);
  \coordinate (A14) at (-0.6,-0.6);

  \draw[fill=gray!10] (A7) -- (A1) -- (A2) -- (A3) -- cycle;
  \draw[fill=gray!10] (A3) -- (A7) -- (A2) -- cycle;
 
  \draw[fill=gray!10] (A8) -- (A1) -- (A2) -- (A3) -- cycle;
  \draw[fill=gray!10] (A3) -- (A8) -- (A2) -- cycle;

  \draw (A7) -- (A1) -- (A2) -- (A3) -- cycle;
  \draw (A3) -- (A7) -- (A2) -- cycle;
 
  \draw (A8) -- (A1) -- (A2) -- (A3) -- cycle;
  \draw (A3) -- (A8) -- (A2) -- cycle;
  \draw (A1) -- (A2) -- (A3) -- cycle;

  \draw[fill=gray!10] (A7) -- (A4) -- (A5) -- (A6) -- cycle;
  \draw[fill=gray!10] (A6) -- (A7) -- (A5) -- cycle;
  \draw (A6) -- (A4);
  \draw[fill=gray!10] (A8) -- (A4) -- (A5) -- (A6) -- cycle;
  \draw[fill=gray!10] (A6) -- (A8) -- (A5) -- cycle;

  \draw (A7) -- (A4) -- (A5) -- (A6) -- cycle;
  \draw (A6) -- (A7) -- (A5) -- cycle;
  \draw (A6) -- (A4);
  \draw (A8) -- (A4) -- (A5) -- (A6) -- cycle;
  \draw (A6) -- (A8) -- (A5) -- cycle;

  \draw[fill=gray!10] (A3) -- (A9) -- (A10) -- (A11) -- cycle;
  \draw[fill=gray!10] (A11) -- (A9) -- (A10) -- cycle;
  \draw (A9) -- (A11);
  \draw[fill=gray!10] (A6) -- (A9) -- (A10) -- (A11) -- cycle;
  \draw[fill=gray!10] (A11) -- (A6) -- (A10) -- cycle;

  \draw (A3) -- (A9) -- (A10) -- (A11) -- cycle;
  \draw (A11) -- (A9) -- (A10) -- cycle;
  \draw (A9) -- (A11);   \draw (A3) -- (A10);
  \draw (A6) -- (A9) -- (A10) -- (A11) -- cycle;
  \draw (A11) -- (A6) -- (A10) -- cycle;

  \draw (A8) -- (A12) -- (A13) -- (A14) -- cycle;
  \draw (A8) -- (A13) -- (A11) -- (A12) -- (A14);
  \draw (A14) -- (A11);

  \node[draw,circle,inner sep=1.2pt,fill] at (A1) {};
  \node[draw,circle,inner sep=1.2pt,fill] at (A2) {};
  \node[draw,circle,inner sep=1.2pt,fill] at (A3) {};
  \node[draw,circle,inner sep=1.2pt,fill] at (A4) {};
  \node[draw,circle,inner sep=1.2pt,fill] at (A5) {};
  \node[draw,circle,inner sep=1.2pt,fill] at (A6) {};
  \node[draw,circle,inner sep=1.2pt,fill] at (A7) {};
  \node[draw,circle,inner sep=1.2pt,fill] at (A8) {};
  \node[draw,circle,inner sep=1.2pt,fill] at (A9) {};
  \node[draw,circle,inner sep=1.2pt,fill] at (A10) {};
  \node[draw,circle,inner sep=1.2pt,fill] at (A11) {};
  \node[draw,circle,inner sep=1.2pt,fill] at (A12) {};
  \node[draw,circle,inner sep=1.2pt,fill] at (A13) {};
  \node[draw,circle,inner sep=1.2pt,fill] at (A14) {};

\end{tikzpicture}

  \end{minipage}
 \end{tabular}
  \caption{The graphs $G_1$, $G_2$ and $G_3$ in Example \ref{BananaTower}.}
\label{Banana2}  
\end{figure}
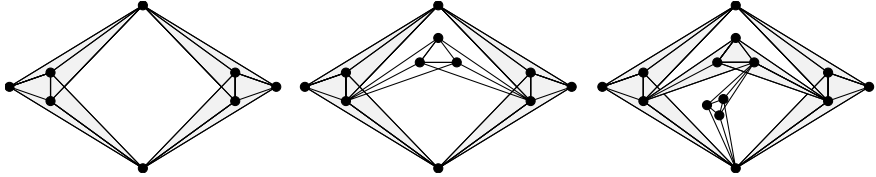

We now prove the equivalence of rigidity and sequential rigidity for countable graphs in $(\bR^2,\|\cdot\|_q)$ for $q\in (1,\infty)$.

\begin{thm}
\label{InfiniteLaman}
Let $G$ be a countable simple graph and let $q\in (1,\infty)$. 
Then the following statements are equivalent.
\begin{enumerate}[(i)]
\item
$G$ is  rigid in $(\mathbb{R}^2,\|\cdot\|_q)$.
\item
$G$ is  sequentially rigid in $(\mathbb{R}^2,\|\cdot\|_q)$.
\end{enumerate}
\end{thm}

\proof
$(i)\Rightarrow (ii)$ Suppose $G$ is rigid in $(\mathbb{R}^2,\|\cdot\|_q)$
and let $p:V(G)\to\mathbb{R}^2$ be a generic placement.
By Theorem \ref{t:IR}, $(G,p)$ has a vertex-complete framework tower $\{(G_k,p):k\in\mathbb{N}\}$ which is relatively infinitesimally rigid.
By Theorem \ref{RelRigidProp}, $G_k$ has a rigid container $H_k$ in $G_{k+1}$ for each $k\in\mathbb{N}$.
Thus $\{H_k:k\in\mathbb{N}\}$ is the required vertex-complete tower of rigid subgraphs in $G$.

$(iii)\Rightarrow (i)$ 
If $p:V(G)\to\mathbb{R}^2$ is a generic placement of $G$ in $(\mathbb{R}^2,\|\cdot\|_q)$ then by Corollary \ref{sequential} $(G,p)$ is infinitesimally rigid and so $G$ is rigid.
\endproof

The following two theorems  can be viewed as an extension of Laman's theorem and its analogue for the non-Euclidean $\ell^q$ norms to countable graphs.
We use the convention that if $P$ is a property of a graph then a $P$-tower is a tower for which each graph $G_k$ has property $P$. 
Thus a \emph{$(2,3)$-tight tower}  is a nested sequence of subgraphs $\{G_k:k\in \bN\}$ each of which is $(2,3)$-tight.

\begin{thm}
\label{InfiniteLaman2}
Let $G$ be a countable simple graph. 
\begin{enumerate}[\bf (A)]
\item
The following statements are equivalent.
\begin{enumerate}[(i)]
\item
$G$ is rigid in $(\mathbb{R}^2,\|\cdot\|_2)$.
\item
$G$ contains a $(2,3)$-tight vertex-complete tower. 
\end{enumerate}
\item
If $q\in (1,2)\cup(2,\infty)$ then the following statements are equivalent.
\begin{enumerate}[(i)]
\item
$G$ is  rigid in $(\mathbb{R}^2,\|\cdot\|_q)$.
\item
$G$ contains a $(2,2)$-tight vertex-complete tower. 
\end{enumerate}
\end{enumerate}
\end{thm}

\proof
$(i)\Rightarrow (ii)$ If $G$ is rigid then by Theorem \ref{InfiniteLaman} $G$ is sequentially rigid and so there exists a vertex-complete tower $\{G_k:k\in\bN\}$ of rigid subgraphs in $G$.
We will construct a tower $\{H_k:k\in\bN\}$ of $(2,l)$-tight subgraphs of $G$ satisfying $V(H_k)=V(G_k)$ for each $k\in\bN$.

Let $H_1=G_1\backslash E_1$ be a minimally rigid spanning subgraph of $G_1$ obtained by removing a set $E_1\subset E(G_1)$ of edges from $G_1$. It follows on considering the rigidity matrix for a generic placement of $G_k$ that $G_k\backslash E_1$ is rigid for each $k\in\bN$. Letting  $G_k'=G_k\backslash E_1$ for all $k\geq 2$ we obtain a vertex-complete tower 
of rigid subgraphs in $G$,
\[H_1\subset G_2'\subset  G_{3}'\subset \cdots\]
where $H_1$ is minimally rigid, $V(H_1)=V(G_1)$ and $V(G_k')=V(G_k)$ for all $k\geq2$.

Suppose we have constructed a vertex-complete tower of rigid subgraphs in $G$,
\[H_1\subset H_2\subset \cdots\subset H_n\subset G_{n+1}'\subset G_{n+2}'\subset \cdots\]
where $H_1,H_2,\ldots,H_n$ are minimally rigid, $V(H_k)=V(G_k)$ for each $k=1,2,\ldots,n$ and 
$V(G_k')=V(G_k)$ for all $k\geq n+1$.
Let $H_{n+1}=G_{n+1}\backslash E_{n+1}$ be a minimally rigid spanning subgraph of $G_{n+1}'$ obtained by removing a set 
$E_{n+1}\subset E(G_{n+1}')$ of edges from $G_{n+1}'$. We can arrange that $H_n$ is a subgraph of $H_{n+1}$.
It follows on considering the rigidity matrix for a generic placement of $G_k'$ that $G_k'\backslash E_{n+1}$ is rigid for each $k\geq n+1$. Replacing $G_k'$ with $G_k'\backslash E_{n+1}$ for each $k\geq n+2$ we obtain a vertex-complete tower in $G$,
\[H_1\subset H_2\subset \cdots\subset H_{n+1}\subset G_{n+2}'\subset G_{n+3}'\subset \cdots\]
consisting of rigid subgraphs with $H_1,H_2,\ldots,H_{n+1}$ minimally rigid, $V(H_k)=V(G_k)$ for each $k=1,2,\ldots,n+1$ and 
$V(G_k')=V(G_k)$ for all $k\geq n+2$.

By induction there exists a vertex-complete tower $\{H_k:k\in\bN\}$ of minimally rigid subgraphs in $G$.
In case (A), Theorem \ref{Laman} implies that each $H_k$ is $(2,3)$-tight and in case (B) Theorem \ref{qNormLaman} implies that each $H_k$ is $(2,2)$-tight.

$(ii)\Rightarrow (i)$ 
Let $\{G_k:k\in\mathbb{N}\}$ be a $(2,l)$-tight vertex-complete tower in $G$.
By Theorems \ref{Laman} and  \ref{qNormLaman},  each $G_k$ is a rigid graph in $(\mathbb{R}^2,\|\cdot\|_q)$ and so $G$ is sequentially rigid.
By Theorem \ref{InfiniteLaman}, $G$ is rigid. 
\endproof

With the convention that a Laman graph is a $(2,3)$-tight finite simple graph the above theorem states that a countable simple graph is minimally  $2$-rigid  if and only if it is the union of an increasing sequence of Laman graphs. 
It follows that some or even all of the vertices of a countable minimally $2$-rigid graph  may  have  infinite degree.

\begin{cor}
\label{MinInfiniteLaman}
Let $G$ be a countable simple graph. 
\begin{enumerate}[\bf (A)]
\item
The following statements are equivalent.
\begin{enumerate}[(i)]
\item
$G$ is minimally rigid in $(\mathbb{R}^2,\|\cdot\|_2)$.
\item
$G$ contains a $(2,3)$-tight edge-complete tower. 
\end{enumerate}
\item
If $q\in (1,2)\cup(2,\infty)$ then the following statements are equivalent.
\begin{enumerate}[(i)]
\item
$G$ is  minimally rigid in $(\mathbb{R}^2,\|\cdot\|_q)$.
\item
$G$ contains a $(2,2)$-tight edge-complete tower. 
\end{enumerate}
\end{enumerate}
\end{cor}

\proof
$(i)\Rightarrow (ii)$ If $G$ is minimally rigid in $(\mathbb{R}^2,\|\cdot\|_q)$ then
by Theorem \ref{InfiniteLaman2}, $G$ contains a $(2,l)$-tight vertex-complete tower $\{G_k:k\in\bN\}$ and this tower must be edge-complete. 

$(ii)\Rightarrow (i)$ 
If $G$ contains a $(2,l)$-tight edge-complete tower $\{G_k:k\in\bN\}$ then by Theorem \ref{InfiniteLaman2}, $G$ is rigid. 
Let $vw\in E(G)$ and suppose $G\backslash \{vw\}$ is rigid.
By Theorem \ref{InfiniteLaman} $G\backslash \{vw\}$ is sequentially rigid and so 
there exists a vertex-complete tower $\{H_k:k\in\bN\}$ in $G\backslash \{vw\}$ consisting of rigid subgraphs.
Choose a sufficiently large $k$ such that  $v,w\in V(H_k)$ and choose a sufficiently large $n$ such that
$vw\in E(G_n)$ and $H_k$ is a subgraph of $G_n$. Then $H_k\cup\{vw\}$ is a subgraph of $G_n$ which fails the sparsity count for $G_n$. We conclude that $G\backslash \{vw\}$ is not rigid in $(\mathbb{R}^2,\|\cdot\|_q)$ for all $vw\in E(G)$. 
\endproof

\subsection{Remarks}
The rigidity of general infinite graphs as bar-joint frameworks was considered first in Owen and Power \cite{owe-pow-hon}, \cite{owe-pow-crystal}.
Part (A) of Theorem \ref{InfiniteLaman2}  answers a question posed in Section 2.6 of Owen and Power \cite{owe-pow-crystal}.

In their analysis of globally linked pairs of vertices in rigid frameworks Jackson, Jordan and Szabadka  \cite{jac-jor-sza}  remark that it follows from the characterisation of independent sets for the rigidity matroid for the Euclidean plane that \emph{linked vertices} $\{v_1,v_2\}$ must lie in the same rigid component. (See also 
\cite{jac-jor}). This assertion is essentially equivalent to part (a) of our Theorem \ref{RelRigidProp}.
The terminology here is that a pair of vertices $\{v_1,v_2\}$ in a graph $G$ is \emph{linked} in $(G,p)$ if there exists an $\epsilon >0$ such that if $q \in P(G)$ is another placement of $G$ with $\|q_v-q_w\|_2=\|p_v-p_w\|_2$ for all $vw\in E(G)$ 
and $\|q_v-p_v\|_2 <\epsilon$ for all $v\in V(G)$ then $\|q_{v_1}-q_{v_2}\|_2=\|p_{v_1}-p_{v_2}\|_2$. 
It can be shown that this is a generic property and
that a subgraph $H\subseteq G$ is relatively rigid in $G$ if and only if for a generic placement $(G,p)$ each pair of vertices in $H$ is linked in $(G,p)$.  

The elementary but key Lemma \ref{FlexProp} is reminiscent of the compactness principle for locally finite structures to the effect that certain properties prevailing for all finite substructures  hold also for the infinite structure. 
For example the $k$-colourability of a graph is one such property.
See Nash-Williams \cite{nas-wil-infinite}.

One can also take a matroidal view for infinitesimally rigid frameworks and define the infinite matroid $\R_2$ (resp $\R_{2,q}$) on the set $S$ of edges of the countable complete graph $K_\infty$. The independent sets in this matroid are the subsets of edges of a sequential Laman graph (resp. sequentially $(2,2)$-tight graph). 
Such  matroids are finitary (see Oxley \cite{oxl}
and Bruhn et al \cite{bru}) and so closely related to their finite matroid counterparts.


\section{Inductive constructions of countable graphs}
\label{InductiveConstructions}
In this section we establish the existence of Henneberg-type construction chains  for the countable graphs that are minimally rigid in 
$(\bR^2,\|\cdot \|_q)$ for $1 < q < \infty$. We also use inductive methods for countably infinite graphs associated with certain  infinitely faceted polytopes in $\bR^3$ and obtain a counterpart of the generic Cauchy theorem. We revisit these frameworks in Section 
\ref{ContinuousRigidity} in connection with more analytic forms of rigidity.

\subsection{Existence of construction chains for  countable rigid graphs.}
\label{ConstructionChains}
Given two finite simple graphs $H$ and $G$ we will use the notation $H\overset{\mu}\longrightarrow G$ to indicate that $G$ is the result of a graph move $\mu$ applied to $H$. 
The following  two graph moves  are designed to preserve
the sparsity count $d|V(G)|-|E(G)|$ appropriate for bar-joint frameworks $(G,p)$ in $\bR^d$.

\begin{defn} Let $G$ be a simple graph.
\begin{enumerate}
\item
A graph $G'$ is said to be obtained from $G$ by applying a {\em Henneberg  vertex extension move of degree $d$} if it results from adjoining a vertex $v$ to $V(G)$ and $d$ edges $v_1v, \dots,v_dv$ to $E(G)$.
\item
A graph $G'$ is said to be obtained from $G$ by applying a {\em Henneberg edge move of degree $d$} if it results from removing an edge $v_1v_2$ from $E(G)$ and applying a Henneberg vertex extension move of degree $d+1$ to $G\backslash \{v_1v_2\}$ where the $d+1$ new edges include $v_1v$ and $v_2v$. 
\end{enumerate}
\end{defn}

An inverse Henneberg vertex extension move can be applied whenever $G$ contains  a vertex $v$ of degree $d$ by simply removing $v$ and all edges incident with $v$.
To obtain conditions under which  an inverse  Henneberg edge move  can be applied we require the following lemma.

\begin{lem}
\label{H2Lemma}
Let $G$ be a $(k,l)$-sparse simple graph  and let $v\in V(G)$ such that $\deg(v)\geq k+1$ and $vv_1,vv_2,\ldots,vv_{k+1}\in E(G)$.
Suppose that  one of the following conditions holds.
\begin{enumerate}[(a)]
\item $k=2$ and $l=3$, or, 
\item $k=l$ and $v_iv_j\notin E(G)$ for some distinct pair $v_i,v_j\in \{v_1,v_2,\ldots,v_{k+1}\}$.
\end{enumerate}
Then there exists  $v_i,v_j\in \{v_1,v_2,\ldots,v_{k+1}\}$ with $v_iv_j\notin E(G)$ such that no $(k,l)$-tight subgraph of $G\backslash\{v\}$  contains both $v_i$ and $v_j$. 
\end{lem} 

\proof
In case $(a)$, since $\deg(v)\geq 3$  there must exist a pair $v_i,v_j\in \{v_1,v_2,v_3\}$ with $v_iv_j\notin E(G)$ since otherwise  the subgraph of $G$ induced by the vertices 
$\{v,v_1,v_2,v_3\}$ would contradict the $(2,3)$-sparsity count for $G$.
Suppose that every pair of distinct vertices $v_i,v_j\in\{v_1,v_2,v_3\}$ with $v_iv_j\notin E(G)$ is contained in a $(k,l)$-tight subgraph $H_{i,j}$ of $G\backslash\{v\}$.
Let $H$ be the union of the subgraphs $H_{i,j}$ together with the vertex $v$, the edges $vv_1,vv_2,vv_3$ and the edge $v_iv_j$ whenever $v_iv_j\in E(G)$.
By Proposition \ref{SparseLemma1}  $H$ is a subgraph of $G$ which contradicts the $(2,3)$-sparsity count.

Similarly in case $(b)$, suppose that every pair  $v_i,v_j\in\{v_1,v_2,\ldots,v_{k+1}\}$ with $v_iv_j\notin E(G)$ is contained in a $(k,k)$-tight subgraph $H_{i,j}$ of $G\backslash\{v\}$.
Let $H$ be the union of the subgraphs $H_{i,j}$ together with the vertex $v$, the edges $vv_1,vv_2,\ldots, vv_{k+1}$ and the edge $v_iv_j$ whenever $v_iv_j\in E(G)$.
By Proposition \ref{SparseLemma1}  $H$ is a subgraph of $G$ which contradicts the $(k,k)$-sparsity count.

\begin{prop}
\label{ReverseHennebergLemma}
Let $G$ be a $(k,l)$-sparse simple graph  and let $v\in V(G)$
such that  $\deg(v)= k+1$ and $vv_1,vv_2,\ldots,vv_{k+1}\in E(G)$.
Suppose that  one of the following conditions holds.
\begin{enumerate}[(a)]
\item $k=2$ and $l=3$, or, 
\item $k=l$ and $v_iv_j\notin E(G)$ for some distinct pair $v_i,v_j\in \{v_1,v_2,\ldots,v_{k+1}\}$.
\end{enumerate}
Then there exists a Henneberg edge move  $\mu:H\to G$ of degree $k$ where 
$H$ is a $(k,l)$-sparse graph obtained by removing the vertex $v$ and adjoining an edge of the form $v_iv_j$.

Moreover,  if $G$ is $(k,l)$-tight then $H$ is also $(k,l)$-tight.
\end{prop}

\proof
By Lemma \ref{H2Lemma} there exists $v_i,v_j\in \in \{v_1,v_2,\ldots,v_{k+1}\}$ with $v_iv_j\notin E(G)$ such that no $(k,l)$-tight subgraph of $G$ contains both $v_i$ and $v_j$.
Let $H$ be the graph with vertex set $V(H)=V(G)\backslash\{v\}$ and edge set $E(H)=(E(G)\backslash\{vv_1,\ldots,vv_{k+1}\})\cup\{v_iv_j\}$.
Then
\[|E(H)|=|E(G)|-k\leq k|V(G)|-l-k=k(|V(H)|+1)-l-k=k|V(H)|-l\]
If  $H'$ is a subgraph of $H$ and $v_iv_j\notin E(H')$ then $H'$ is a subgraph of $G$ and so the $(k,l)$-sparsity count holds.
If $v_iv_j\in E(H')$ then $H'\backslash\{v_iv_j\}$ is a subgraph of $G$ which contains the vertices $v_i$ and $v_j$ and so $H'\backslash\{v_iv_j\}$  is $(k,l)$-sparse but not $(k,l)$-tight.
Hence
\[|E(H')|=|E(H'\backslash\{v_iv_j\})|+1\leq (k|V(H'\backslash\{v_iv_j\})|-l-1)+1=k|V(H')|-l\]
Thus $H$ is $(k,l)$-sparse and there exists a Henneberg edge move $\mu:H\to G$ of degree $k$.
The final statement is clear.
\endproof

\begin{defn}  
A {\em  Henneberg construction chain}  is a finite or countable sequence of graphs $G_1,G_2,G_3,\ldots$  together with graph moves,
\[G_1\overset{\mu_1}\longrightarrow  G_2\overset{\mu_2}\longrightarrow G_3\overset{\mu_{3}}\longrightarrow \cdots\]
such that each  $\mu_{k}$ is either a Henneberg vertex addition move of degree $2$ or a Henneberg edge move of degree $2$. 
\end{defn}

If $G$ is a $(2,3)$-tight simple graph then there exists a Henneberg construction chain of finite length from $K_2$ to $G$ (see \cite{gra-ser-ser}, \cite{Hen}),
\[K_2=G_1\overset{\mu_1}\longrightarrow  G_2\overset{\mu_2}\longrightarrow  \cdots \overset{\mu_{n-1}}\longrightarrow G_n=G\]
We prove below that such a construction chain exists between any nested pair of $(2,3)$-tight graphs.

\begin{lem}
\label{DegreeLemma}
Let $G$ be a $(k,l)$-sparse graph with $k,l\geq 1$.
\begin{enumerate}[(i)]
\item $\deg(v)\leq 2k-1$ for some $v\in V(G)$.
\item If  $G$ is $(k,l)$-tight then $\min \{\deg (v):v\in V(G)\}\in[k, 2k-1]$.
\item If $G$ is $(k,l)$-tight and $H$ is a subgraph of $G$ which is not a spanning subgraph then there exists a vertex $v\in V(G)\backslash V(H)$ such that $\deg(v)\in [k,2k-1]$. 
\end{enumerate}
\end{lem}

\proof
$(i)$
If $\deg(v)\geq 2k$ for all $v\in V(G)$ then 
\[|E(G)|=\frac{1}{2}\sum_{v\in V(G)}\deg(v) \geq k|V(G)|> k|V(G)|-l\]
and this contradicts the sparsity count for $G$.

$(ii)$
Suppose $\deg(v) <k$ for some $v\in V(G)$ and let $H$ be the vertex-induced subgraph of $G$ obtained from $V(G)\backslash\{v\}$.
Then 
\[|E(H)| = |E(G)|- \deg(v) 
= k(|V(H)|+1) -l-\deg(v) 
> k|V(H)|-l\]
This contradicts the sparsity count for the subgraph $H$ and so $\deg(v)\geq k$ for all $v\in V(G)$.
The result now follows from $(i)$.

$(iii)$ The vertex-induced subgraph of $G$ determined by the vertices in $V(G)\backslash V(H)$ is $(k,l)$-sparse and so it follows from $(i)$ that there exists $v\in V(G)\backslash V(H)$ with $\deg (v)\leq 2k-1$. The result now follows from $(ii)$. 
\endproof

\begin{prop}
\label{ConstructionProp}
Let $G \subseteq G'$ be an inclusion of  $(2,3)$-tight graphs.
Then there exists a Henneberg  construction chain from $G$ to $G'$,
\[G=G_1\overset{\mu_1}\longrightarrow  G_2\overset{\mu_2}\longrightarrow  \cdots \overset{\mu_{n-1}}\longrightarrow G_n=G'\]

\end{prop}

\begin{proof}
Suppose there is no Henneberg construction chain  from $G$ to $G'$.
Let $\S(G')$ be the set of all pairs $(H_1,H_2)$ of $(2,3)$-tight subgraphs of $K_{V(G')}$ such that $H_1$ is a proper subgraph of $H_2$. Then $\S(G')$  can be endowed with a partial order by the rule $(H_1,H_2)\leq (\tilde{H}_1,\tilde{H}_2)$ if and only if either $|V(H_1)|<|V(\tilde{H}_1)|$, or, $|V(H_1)|=|V(\tilde{H}_1)|$ and $|V(H_2)|\leq |V(\tilde{H}_2)|$.
Let $\tilde{\S}(G')\subset \S(G')$ be the subset of all pairs $(H_1,H_2)$ with the property that there does not exist  a Henneburg construction chain from $H_1$ to $H_2$.
Then $\tilde{\S}(G')\not=\emptyset$ and so there exists a minimal element $(H_1,H_2)\in\tilde{\S}(G')$.
By Lemma \ref{DegreeLemma} there exists $v\in V(H_2)\backslash V(H_1)$ such that 
$v$ has degree $2$ or degree $3$ in $H_2$.

Suppose $\deg(v)=2$ with $vv_1,vv_2\in E(H_2)$. 
Let $\tilde{H}_2$ be the vertex-induced subgraph of $H_2$ with $V(\tilde{H}_2)=V(H_2)\backslash\{v\}$.
Then $H_1$ is a subgraph of $\tilde{H}_2$ and $\tilde{H}_2$ is $(2,3)$-tight. Thus $(H_1,\tilde{H}_2)\in \S(G')$ and $(H_1,\tilde{H}_2)<(H_1,H_2)$.
By the minimality of $(H_1,H_2)$ in $\tilde{\S}(G')$ there must exist a Henneberg construction chain of finite length from $H_1$ to $\tilde{H}_2$,
\[H_1=H_{1,1}\overset{\mu_1}\longrightarrow  H_{1,2}\overset{\mu_2}\longrightarrow  \cdots \overset{\mu_{n-1}}\longrightarrow H_{1,n}=\tilde{H}_2\]
 Applying a Henneburg vertex extension move of degree $2$ to $\tilde{H}_2$ based on the vertices $v_1$ and $v_2$ we obtain $H_2$. This is a contradiction as there does not exist  a Henneberg construction chain from $H_1$ to $H_2$. 

Suppose $\deg(v)=3$ with $vv_1,vv_2,vv_3\in E(H_2)$. 
By Proposition \ref{ReverseHennebergLemma} there exists a $(2,3)$-tight graph $\tilde{H}_2$ obtained by removing the vertex $v$ from $H_2$ and adjoining an edge of the form $v_iv_j$ together with a Henneberg edge move $\mu:\tilde{H}_2\to H_2$ of degree $2$.
Now $H_1$ is a subgraph of $\tilde{H}_2$ and so $(H_1,\tilde{H}_2)\in \S(G')$ with $(H_1,\tilde{H}_2)<(H_1,H_2)$.
By the minimality of $(H_1,H_2)$ in $\tilde{\S}(G')$ there must exist a Henneberg  construction chain from $H_1$ to $\tilde{H}_2$,
\[H_1=H_{1,1}\overset{\mu_1}\longrightarrow  H_{1,2}\overset{\mu_2}\longrightarrow  \cdots \overset{\mu_{n-1}}\longrightarrow H_{1,n}=\tilde{H}_2\]
Applying a Henneburg edge move of degree $2$ to $\tilde{H}_2$ based on the vertices $v_1,v_2,v_3$ and the edge $v_iv_j$ we obtain $H_2$.
This is again a contradiction and so we conclude that there must exist a Henneburg  construction chain from $G$ to $G'$.
\end{proof}

To establish the existence of construction chains in the class of $(2,2)$-tight graphs we require the following two additional graph moves.

\begin{defn} Let $G$ be a simple graph.
\begin{enumerate}
\item A graph $G'$ is obtained from $G$ by a {\em vertex-to-$K_4$ move} if a vertex $w_0$ is chosen in $G$ and three vertices $w_1,w_2, w_3$ are adjoined to $V(G)$ together with all interconnecting edges $w_iw_j$ and every edge of the form $vw_0\in E(G)$ is either left unchanged or is replaced with an edge of the form $vw_j$ for some $j\in\{1,2,3\}$. 
\item A graph $G'$ is obtained from $G$ by a {\em vertex to $4$-cycle move} if a vertex $v$ of degree $2$ is chosen in $G$ with $vv_1,vv_2\in E(G)$ and a vertex $v_0$ is adjoined to $V(G)$ together with the edges $v_0v_1,v_0v_2$ such that every edge of the form $vw\in E(G)$ is either left unchanged or is replaced with the edge $v_0w$.
\end{enumerate}
\end{defn}

We  prove below an analogue of Proposition \ref{ConstructionProp} for the class of $(2,2)$-tight graphs.
The method of proof is similar although more involved due to the additional graph moves required to construct $(2,2)$-tight graphs.

\begin{defn}
A {\em non-Euclidean Henneberg  construction chain} is  a finite or countable sequence of graphs $G_1,G_2,G_3,\ldots$  together with graph moves,
\[G_1\overset{\mu_1}\longrightarrow  G_2\overset{\mu_2}\longrightarrow G_3\overset{\mu_{3}}\longrightarrow \cdots\]
such that each  $\mu_{k}$ is either a Henneberg vertex addition move of degree $2$, a Henneberg edge move of degree $2$, a vertex-to-$K_4$ move or a vertex to $4$-cycle move. 
\end{defn}

If $G$ is  a $(2,2)$-tight  simple graph then there exists a non-Euclidean Henneberg  construction chain  from $K_1$ to $G$ (see \cite{NOP2}),
\[K_1=G_1\overset{\mu_1}\longrightarrow  G_2\overset{\mu_2}\longrightarrow  \cdots \overset{\mu_{n-1}}\longrightarrow G_n=G\]
More generally, we have the following proposition.

\begin{prop}
\label{NonEuclideanConstructionProp}
Let $G \subseteq G'$ be an inclusion of  $(2,2)$-tight graphs.
Then there exists a non-Euclidean Henneberg construction chain from $G$ to $G'$,
\[G=G_1\overset{\mu_1}\longrightarrow  G_2\overset{\mu_2}\longrightarrow  \cdots \overset{\mu_{n-1}}\longrightarrow G_n=G'\]
\end{prop}

\begin{proof}
Suppose there does not exist a non-Euclidean Henneberg  construction chain  from $G$ to $G'$. Let $\S(G')$ be the set of all pairs $(H_1,H_2)$ of $(2,2)$-tight subgraphs of $K_{V(G')}$ such that $H_1$ is a proper subgraph of $H_2$. 
Let $\tilde{\S}(G')\subset \S(G')$ be the subset of all pairs $(H_1,H_2)$ with the property that there does not exist  a non-Euclidean Henneburg  construction chain from $H_1$ to $H_2$.
Choose a minimal pair $(H_1,H_2)\in \tilde{\S}(G')$, as in the proof of Proposition \ref{ConstructionProp}. 
By Lemma \ref{DegreeLemma} there exists $v_0\in V(H_2)\backslash V(H_1)$ such that 
$v_0$ has degree $2$ or degree $3$ in $H_2$.

If  $\deg(v_0)=2$ then we can apply an inverse Henneberg vertex addition move as in Proposition \ref{ConstructionProp} to obtain a contradiction.

Suppose $\deg(v_0)=3$ with $v_0v_1,v_0v_2,v_0v_3\in E(H_2)$.
If  $v_iv_j\notin E(H_2)$ for some distinct pair $v_i,v_j\in\{v_1,v_2,v_3\}$ then we can apply an inverse Henneberg edge move as in Proposition \ref{ConstructionProp} to obtain a contradiction.

If the complete graph $K$ on the vertices $\{v_0,v_1,v_2,v_3\}$ is a subgraph of $H_2$ then
every vertex  $v\in V(H_2)\backslash V(H_1)$ is incident with at most two vertices in $\{v_1,v_2,v_3\}$.
Otherwise, the vertex-induced subgraph on $\{v,v_0,v_1,v_2,v_3\}$ would contradict the sparsity count for $H_2$.

We consider two possible cases.
Firstly, the case when there does not exist a vertex in $V(H_2)\backslash V(H_1)$ which is incident with two vertices in $\{v_1,v_2,v_3\}$. And secondly, the case when there does exist a vertex in $V(H_2)\backslash V(H_1)$ which is incident with two vertices in $\{v_1,v_2,v_3\}$.

In the first case, if $K\cap H_1 = \emptyset$ then let $\tilde{H}_2$ be the $(2,2)$-tight graph obtained from $H_2$ by contracting the complete graph on $\{v_0,v_1,v_2,v_3\}$ to any one of the vertices $v_0,v_1,v_2,v_3$.
If $K\cap H_1\not=\emptyset$ then, by Lemma \ref{SparseLemma3},  $K\cap H_1$ is $(2,2)$-tight and 
hence must consist of a single vertex $v_i\in\{v_1,v_2,v_3\}$.
Let $\tilde{H}_2$ be the $(2,2)$-tight graph obtained from $H_2$ by contracting the complete graph on $\{v_0,v_1,v_2,v_3\}$ to $v_i$.
Then $(H_1,\tilde{H}_2)\in \S(G')$ and $(H_1,\tilde{H}_2)<(H_1,H_2)$.
By the minimality of $(H_1,H_2)$ in $\tilde{\S}(G')$ there must exist a non-Euclidean Henneberg  construction chain from $H_1$ to $\tilde{H}_2$,
\[H_1=H_{1,1}\overset{\mu_1}\longrightarrow  H_{1,2}\overset{\mu_2}\longrightarrow  \cdots \overset{\mu_{n-1}}\longrightarrow H_{1,n}=\tilde{H}_2\]
Applying a vertex-to-$K_4$ move to $\tilde{H}_2$  we obtain $H_2$.
This is  a contradiction as there does not exist  a non-Euclidean Henneberg  construction chain from $H_1$ to $H_2$.

In the second case, suppose $w_0\in V(H_2)\backslash V(H_1)$ and $w_0$ is incident with the vertices $v_i,v_j\in \{v_1,v_2,v_3\}$.
Let $\tilde{H}_2$ be the graph obtained from $H_2$ by contracting $w_0$ to $v_0$.
Thus the edges $w_0v_i$ and $w_0v_j$ are removed, every remaining edge of the form $w_0v$ is replaced with $v_0v$ and the vertex $w_0$ is removed.
Since $w_0\notin V(H_1)$, $H_1$ is a subgraph of $\tilde{H}_2$. 
We have $(H_1,\tilde{H}_2)\in \S(G')$ and $(H_1,\tilde{H}_2)<(H_1,H_2)$.
By the minimality of $(H_1,H_2)$ in $\tilde{\S}(G')$ there must exist a non-Euclidean Henneberg  construction chain from $H_1$ to $\tilde{H}_2$,
\[H_1=H_{1,1}\overset{\mu_1}\longrightarrow  H_{1,2}\overset{\mu_2}\longrightarrow  \cdots \overset{\mu_{n-1}}\longrightarrow H_{1,n}=\tilde{H}_2\]
Now $H_2$ is obtained by applying a vertex to $4$-cycle move to $\tilde{H}_2$ which is based on the edges $v_0v_i$ and $v_0v_j$.
This is  a contradiction as there does not exist  a non-Euclidean Henneberg  construction chain from $H_1$ to $H_2$. 

We conclude that there must exist a non-Euclidean Henneberg  construction chain from $G$ to $G'$.
\end{proof}

We now prove that every minimally rigid graph in $(\bR^2,\|\cdot\|_q)$ arises as the graph limit determined by a countable construction chain.

\begin{defn}
Let $\{G_k:k\in\bN\}$ be a sequence of finite simple graphs such that $V(G_k)\subseteq V(G_{k+1})$ for all $k\in\bN$.
The {\em countable graph limit} $\varinjlim\, G_k$ is the countable graph with vertex set
\[V(\varinjlim\, G_k)=\bigcup_{k\in\mathbb{N}} V(G_k)\] 
and edge set \[E(\varinjlim\, G_k)=\{vw:\exists\,n\in\mathbb{N} \, |\, vw\in E(G_k), \,\,\forall\, k\geq n\}\]
\end{defn}

The edge set of a countable graph limit consists of edges which are contained in some $G_k$ and are not subject to removal further along the chain.
It may be that every edge in every finite graph in the sequence is subject to later removal, in which case the limit graph will have no edges. 

The graph moves that we have considered have the property that each move $G \overset{\mu}\longrightarrow G'$ is associated with a vertex set inclusion $V(G)\subset V(G')$. In particular, a countable Henneberg construction chain or a countable non-Euclidean Henneberg construction chain 
\[G_1\overset{\mu_1}\longrightarrow  G_2\overset{\mu_2}\longrightarrow G_3\overset{\mu_{3}}\longrightarrow \cdots\]
has an associated graph limit $\varinjlim\, G_k$.

\begin{thm}\label{t:inductive}
Let $G$ be a countable simple graph which is minimally rigid in $(\bR^2,\|\cdot\|_q)$ where $q\in (1,\infty)$. 

\begin{enumerate}[\bf (A)]
\item 
If $q=2$ then there exists a countable Henneberg construction chain 
\[K_2=G_1\overset{\mu_1}\longrightarrow  G_2\overset{\mu_2}\longrightarrow G_3\overset{\mu_{3}}\longrightarrow \cdots\] 
such that
 $G=\varinjlim\,G_k$.
\item 
If $q\not=2$ then there exists a countable non-Euclidean Henneberg construction chain 
\[K_1=G_1\overset{\mu_1}\longrightarrow  G_2\overset{\mu_2}\longrightarrow G_3\overset{\mu_{3}}\longrightarrow \cdots\] 
such that
 $G=\varinjlim\,G_k$.
\end{enumerate}
\end{thm}

\proof
(A) If $G$ is minimally rigid in $(\mathbb{R}^2,\|\cdot\|_2)$ then by Corollary \ref{MinInfiniteLaman}(A) there exists an edge-complete tower $\{G_k:k\in\mathbb{N}\}$ in $G$ such that
each $G_k$ is $(2,3)$-tight.
By Proposition \ref{ConstructionProp} there exists a Henneberg  construction chain of finite length, \[G_k=H_{k,1} \overset{\mu_{k,1}}\longrightarrow  H_{k,2}\overset{\mu_{k,2}}\longrightarrow \cdots\overset{\mu_{k,n_k-1}}\longrightarrow  H_{k,n_k}=G_{k+1}\] for each $k\in\mathbb{N}$.
We can concatenate these construction chains to obtain a  countable Henneberg  construction chain \[K_2=\tilde{G}_1\overset{\tilde{\mu}_1}\longrightarrow  \tilde{G}_2\overset{\tilde{\mu}_2}\longrightarrow \tilde{G}_3\overset{\tilde{\mu}_{3}}\longrightarrow \cdots\] which satisfies   $G= \varinjlim\,\tilde{G}_k$.

For the proof of (B) apply a similar argument using Corollary \ref{MinInfiniteLaman}(B) and Proposition \ref{NonEuclideanConstructionProp}.
\endproof


\subsection{Simplicial graphs and Cauchy's rigidity theorem}\label{FiniteSimplicial}
First we show how one may compute the flexibility dimension $\dim_{3,q}(G)$
of the edge graph of any polytope.  The following terminology  will be convenient. 

\begin{defn}\label{d:SimplicialGraph} A finite simple graph is said to be {\em simplicial with 
topological connectivity $\kappa$} if it has at least three edges and has a planar embedding in which exactly $\kappa$ faces 
are not triangular. Furthermore 
these faces have no edges in common.
\end{defn}

If $\kappa=0$ then such a graph is the edge graph of a convex bounded polyhedron with triangular 
faces. Such polyhedra are the simplicial polytopes in $\bR^3$.

The following vertex splitting construction move preserves the class of $(3,6)$-tight graphs
and the class of $(3,3)$-tight graphs, and  will be particularly useful in connection with the  rigidity and flexibility of placements in three dimensional normed spaces. 
Note that a Henneberg vertex addition move which adds a degree $3$ vertex to the face of a planar graph is in fact a vertex splitting move.

\begin{defn} Let $G$ and $G'$ be  finite simple graphs.
Then $G'$ is said to be obtained from $G$ by applying a {\em  vertex 
splitting move for three dimensions} if it results from
\begin{enumerate}[(i)]
\item adding a vertex $v_0$ to $V(G)$,
\item adding an edge $v_0v_1$ for a vertex $v_1$ in $V(G)$ which has at 
least two incident edges $v_1v_2$ and $v_1v_3$,
\item adding the edges $v_0v_2, v_0v_3$,
\item replacing any number of the edges $wv_1, w\neq v_2,v_3$ by edges 
$wv_0$.
\end{enumerate}
\end{defn}

The preservation of $3$-rigidity under the Henneberg vertex addition move on vertices $v_1,v_2,v_3\in V(G)$ is clear, since for any generic placement $(G',p')$ with $p_v,p_{v_1},p_{v_2},p_{v_3}$ in general position we have \[\rank R_q(G',p') = \rank R_q(G,p')+3\]
The proof for the Henneberg edge move in three dimensions is also straightforward, being a  variation of the proof of \cite[Lemma 3.4]{kit-pow}.
The  preservation of $3$-rigidity under the vertex splitting move is well-known for $q=2$. 
See Whiteley \cite[Theorem 1.4.6]{whi-1}. This argument also has a
straightforward $q$-norm variant.

\begin{thm}\label{t:SimpGraphDim}
Let $G=(V(G),E(G))$ be a finite simplicial graph of connectivity $\kappa$  whose non-triangular faces are bordered by cycles of length $\gamma_1,\dots ,\gamma_\kappa$.
Then 
\begin{enumerate}[(i)]
\item
\[
\dim_{3,2}(G)= \gamma_1+\dots +\gamma_\kappa-3\kappa
\]
\item For $|V(G)|\geq 6$ and $1<q<\infty$, $q\neq 2$,
\[
\dim_{3,q}(G)= \dim_{3,2}(G)+3
\]
\end{enumerate}
\end{thm}
\begin{proof}Consider first the case $\kappa=0$. Observe that any simplicial  graph $G$ of connectivity $0$ can be constructed
from $K_3$ by vertex splitting moves.  To see this note that the inverse of a vertex splitting move is an edge contraction move. 
Also, if $G$ is not equal to $K_3$ then there are adjacent triangles, sharing an edge, and edge contraction of this edge results in a simplicial graph of connectivity $0$ which is an antecedent  of $G$ for vertex splitting.
In view of this constructibility of $G$ and the fact that  $K_3$ is  $3$-rigid it follows that $G$ is $3$-rigid and 
$\dim_{3,2}(G)=0$. 
\begin{center}
\begin{figure}[ht]
\centering
\includegraphics[width=8cm]{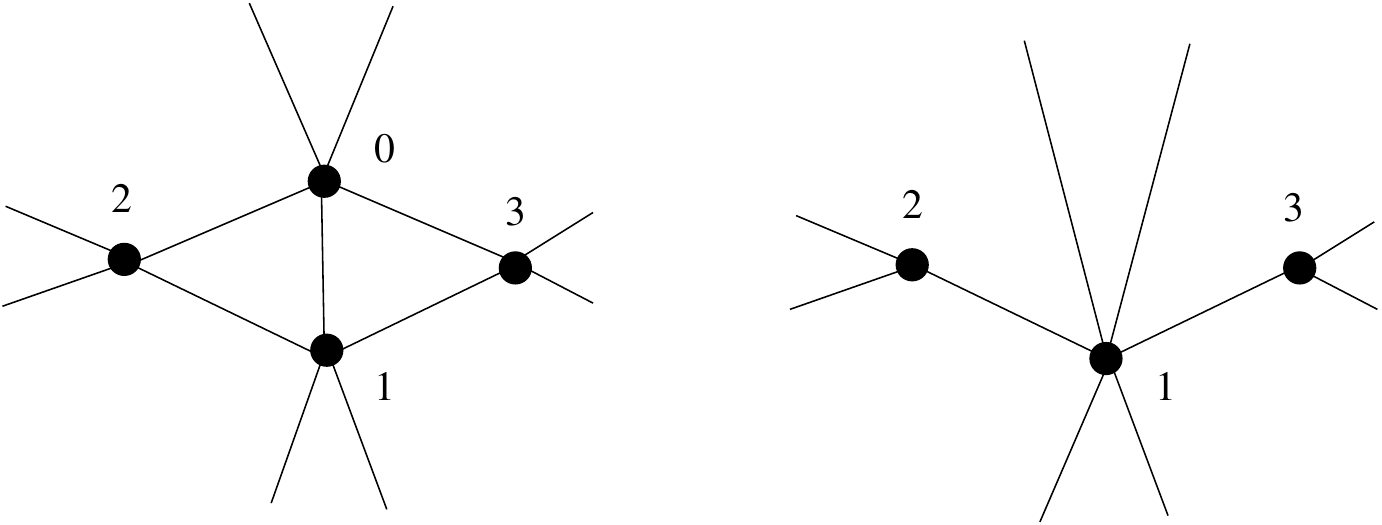}
\caption{Edge contraction.}
\label{vertexmerge}
\end{figure}
\end{center}
Since
$|E|=3|V|-6$ the graph $G$ is minimally $3$-rigid and in particular is $3$-independent. The formula for $\dim_{3,2}(G)$ in the case $\kappa >0$ now follows since in this case we may add $\kappa$ vertices and $\gamma_1+\dots +\gamma_\kappa-3\kappa$ edges 
to $G$ to create a $3$-independent simplicial graph of connectivity $0$.
For $q\neq 2$  note first  that the $(3,3)$-tight graph $K_6$ may be viewed as the graph of the regular octahedron together with three added internal edges (or "shafts"). An argument paralleling that above shows that
any simplicial graph $G$ of connectivity $0$ to which three nonincident internal edges have been added, gives a graph $G^+$ which can be constructed from $K_6$ by vertex splitting moves (which carry the shaft edges). Since $K_6$ is minimally rigid it follows that $G^+$ is minimally rigid and the proof is completed as for the case $q=2$.
To see this form of constructibility suppose that $G^+$ is not so constructed and that $|V(G^+)|$ is minimal amongst such graphs. We show that $G^+=K_6$. If this is not the case then there are more than $6$ vertices and thus a vertex $v$ in $V(G)=V(G^+)$ with the same degree in $G$ and 
$G^+$. But then there is a vertex splitting move
$G_-\to G$ for which the antecedent is simplicial of connectivity $0$, with an implied (shaft carrying) vertex splitting move $G_-^+\to G^+$. This is a contradiction since $G_-^+$
cannot be constructible and yet it has fewer vertices than $G^+$.
Since $K_6$ is minimally rigid for the non-Euclidean norms the proof is completed as before.
\end{proof}

The case $\kappa =0$ is the generic Cauchy rigidity theorem. For the non-Euclidean norms we have the following counterpart.
Borrowing terminology from Whitelely \cite{whi-poly2} we shall refer to an  internal bar that is added to a polytope framework as a shaft.

\begin{cor}\label{t:q-polytope}
A generic simplicial polytope framework with  three non-incident shafts is
isostatic in $(\bR^3,\|\cdot \|_q)$ for $1<q<\infty$, $q\neq 2$.
\end{cor}

We see that as a corollary of infinitesimal rigidity preservation the vertex splitting move applied to a simplicial graph of connectivity $\kappa$ preserves independence, and preserves the joint cycle index $(\gamma_1, \dots ,\gamma_\kappa)$.

\subsection{Countable simplicial graphs.}\label{sub:SimplicialGraphs}
We now consider the countable graphs that can be obtained
from finite simplicial graphs of connectivity type $\kappa$ by certain directed construction chains. 
Viewing such graphs as inscribed on the surface of a sphere this "directedness" corresponds, roughly speaking,  to the $\kappa$ nontriangular faces (the topological "holes") converging towards a set $F$  on the sphere consisting of $\kappa$ points.

These graphs can be thought of as  simplicial triangulations of a finitely punctured sphere, $S^2\backslash F$, and we give the following direct definition in these terms, as a class of infinite planar graphs. The class is somewhat larger than that already alluded to above since it also allows for the infinite internal refinement of a finite number of triangular faces of the initial graph $G_1$.

Recall that a countable graph has a planar embedding or planar 
realisation, or is a planar graph,
if it may be realised by a set of distinct points in the plane and a 
family of non-crossing continuous paths. Let $S^2$ be the one 
point compactification of the plane.

\begin{defn}\label{simplicialgraph}Let $G$ be a countable graph and 
$\rho$ a positive integer.
Then $G$ is said to be
a {\em simplicial graph of refinement type $\rho $}
if there is a planar realisation of $G$ which is a triangulation of 
$S^2\backslash F$, where $|F|=\rho$, 
and
\begin{enumerate}[(i)]
\item the vertices of $G$ have finite incidence,
\item the set $F$ is the set of accumulation points of the vertex points,
\item $S^2\backslash F$ is the union of the triangular faces of 
the realisation of $G$.
\end{enumerate}
\end{defn}

\begin{defn}
A simplicial graph of refinement type $\rho\geq 1$ has {\em connectivity 
$0\leq\kappa\leq \rho$}
if in its representation in $S^2\backslash F$ there are 
$\kappa$ points $p$ of $F$ with the following property: there is a 
neighbourhood of $p$
which contains no $3$-cycles of represented edges around $p$.
\end{defn}

It follows that the simplicial graphs of finite refinement type and 
connectivity $\kappa$ have edge-complete construction chains 
\[
G_1 \overset{\mu_1}\longrightarrow G_2 \overset{\mu_2}\longrightarrow G_3 \overset{\mu_3}\longrightarrow \cdots
\]
In this chain successive graphs
are obtained by partially paving in the $\kappa$ nontriangular faces 
and $\rho -\kappa$ triangular faces. In particular a simplicial graph with $(\rho, \kappa)=(1,0)$ is obtained  by  
joining a sequence of simplicial polyhedral graphs over common faces so 
that each polyhedron (except the first) has two neighbours. Similarly, 
the simplicial graphs with $(\rho, \kappa)=(1,1)$ are obtained by 
joining together a sequence of simplicial tubes (except for the first) 
over identified end cycles.

\begin{thm}\label{genericCauchy}Let $G$ be a countable simplicial  graph of 
refinement type $\rho$ and connectivity $\kappa$. Then the following conditions are equivalent.
\begin{enumerate}[(i)]
\item $G$ is $3$-rigid.
\item $G$ is sequentially $3$-rigid.
\item $\kappa=0.$
\end{enumerate}

Furthermore, in this case
$G=\varinjlim\,G_k$ where 
\[K_4=G_1\overset{\mu_1}\longrightarrow  G_2\overset{\mu_2}\longrightarrow G_3\overset{\mu_{3}}\longrightarrow \cdots\] 
is a construction chain of 
convex simplicial polyhedral graphs and the construction moves are 
vertex splitting moves.
\end{thm}

\begin{proof}Condition $(iii)$ implies that there is an edge-complete inclusion chain of finite graphs each of which is the edge graph of a simplicial polytope. Thus $(ii)$ holds, by the generic Cauchy Theorem \ref{t:GenericCauchy}, and $(i)$ follows.

To see that $(i)$ implies $(iii)$ we make use of the following theorem of Whiteley \cite{whi-poly2}: Let $G$ be a $4$-connected simplicial polytope graph. Then any graph formed  by the removal of an edge $v_1v_2$ and the addition of an edge $v_iv_j$ with $i,j \neq 1,2$ to create a subgraph isomorphic to $K_4$ is $3$-rigid. This may also be proven using a similar argument to that of the proof of Theorem \ref{t:SimpGraphDim} and the fact that small graphs of this type are $3$-rigid. 

Suppose then that $(iii)$ does not hold and consider first the case $0<\kappa=\rho=1$. Thus there is an edge-complete
inclusion chain $G_1\subseteq G_2 \subseteq \dots $ of graphs for polytopes
each of which has one nontriangular face which is a $\gamma$-cycle, with $\gamma \geq 4$. Let $(G,p)$ be a generic placement. For every $k>1$ the restriction map $\F(G_{k},p)\to \F(G_1,p)$ is nonzero. Indeed consider a nonzero infinitesimal flex of the augmentated framework $(G_k^+,p)$ of $(G_k,p)$ obtained by adding $\gamma - 4$ edges
across the nonsimplicial face. By the generic Cauchy theorem such a flex exists and by Whiteley's theorem its  restriction to $(G_1,p)$ is nonzero and nontrivial. It follows from the finite-dimensionality of the flex space of $(G_1,p)$ that there is an infinitesimal flex of $(G_1,p)$ that is an enduring flex for the tower. Thus there exists a nontrivial infinitesimal flex of $(G,p)$.

The general case follows from the same argument after taking an augmentation by
edges to create a simplicial  polytope with one edge missing.
\end{proof}

We note below some examples of simplicial graphs of finite 
refinement type. In Section \ref{ContinuousRigidity} we shall see that there exists 
continuously rigid strictly convex placements for all of these graphs.

As we have seen, for $1<q<\infty$, $q\neq 2,$ a convex simplicial 
polytope in three dimensions
needs the addition of three internal shafts to become a rigid graph for
$(\bR^3,\|\cdot \|_q)$. The argument above adjusts readily to show that the same is true
for the simplicial graphs of finite refinement type with $\kappa =0$.

\subsection{Infinitely-faceted polytopes.}
Let $G=(V(G),E(G))$ be a countable simplicial graph
of finite refinement type $\rho\geq 1$.
In particular $G$ has a planar  representation with $\rho$  points of 
accumulation for the vertex points.
It follows that $G$ admits natural bounded placements   $p: V(G) \to\bR^3$ 
with  $\rho$ points of
accumulation. Formally we define a \emph{continuous placement} of 
$G$ to be a placement for which the mapping $v \mapsto p_v$ is continuous for the 
relative topology on $V(G)$ coming from the one-point compactification of 
the plane.
We now define strictly convex and semi-convex continuous placements
of simplicial graphs of finite refinement type as well as 
infinitely-faceted generalisations of finite simplicial polytopes.

Recall that a $d$-dimensional finite polytope in $\bR^3$ is a compact 
subset $\P$ which is the convex hull of a non-coplanar finite set.
In particular polytopes are assumed to be convex, with nonzero volume,
and have well-defined closed subsets called faces, edges and vertices.
A finite polytope is {\em simplicial} if its faces are triangles.

In the next definition we refer to the faces of a countable simplicial graph
of finite refinement type. These are defined as the particular 
$3$-cycles of vertices corresponding to the faces of any planar 
representation.

\begin{defn}
Let $G$ be a countable simplicial graph of refinement type $\rho\geq 1$ and let 
$p:V(G) \to \bR^3$ be a placement. Then $(G,p)$ is a {\em strictly convex 
simplicial polytope framework} if the following two conditions hold.
\begin{enumerate}[(i)]
\item $p:V(G) \to \bR^3$ is a continuous placement.
\item For  every face  $\{a,b,c\}$ one of the open half-spaces for the 
plane through $\{p_a,p_b,p_c\}$ contains the  set $\{p_v: v\neq a,b,c\}$.
\end{enumerate}
Furthermore, the {\em body} $\P(G,p)$ of $(G,p)$ is defined to be the closed 
convex hull  of the set $\{p_v:v\in V\}$.
\end{defn}

A generalised polytope $\P$, with infinitely many faces, could be defined
as a strictly convex body with infinitely many polygonal faces whose 
union and closure form the topological boundary. The following subclass 
is a natural one from a rigidity perspective.

\begin{defn}\label{infinitepolytope}
A  {\em simplicial polytope of refinement type $\rho\geq 1$} is the body of a strictly convex continuous placement of a simplicial graph of refinement type  
$\rho\geq 1$.
\end{defn}

Note that 
the topological boundary  of the body $\P(G,p)$ is the union of the $\rho$ accumulation points and the 
closed convex sets determined by the placement of the faces of $G$.

One can similarly define a convex simplicial polytope framework in terms 
of closed half spaces, rather than open half-spaces. However, this  
class is not quite appropriate for us since there may be a proliferation of infinitesimal flexes arising from vertices 
whose incident edges lie in a common plane. There is nevertheless an 
appropriate intermediate
class which we refer to as \emph{semi-convex}. This  requires the convexity
of the placement together with the property that if $p_v$ lies in a 
convex hull of a finite number of vertices then it lies in the convex 
hull of two of them.  The body of such a framework may be a classical 
simplicial polytope
with countably many ``extra" vertices distributed on the edges of the 
polytope.
The semi-convex placements of finite simplicial polytopes are known to 
be infinitesimally rigid, so this is a natural class.

\begin{eg}\label{polytopeframeworks}
{\bf Infinite frameworks from finite polytopes.}
Figure \ref{trianglepyramid} indicates a bar-joint framework 
$\D_3=(G,p)$ in $3$-dimensional Euclidean space which is the placement of 
a simplicial graph $G$ with $(\rho(G),\kappa(G))=(1,0)$. The vertex placements 
occur on three edges of a triangle based pyramid and have a single point 
of accumulation. The simplicial graph is minimally $3$-rigid and 
sequentially rigid. That the semi-convex placement $(G,p)$ itself is sequentially infinitesimally 
rigid follows from the rigidity of finite semi-convex simplicial polytope frameworks (see \cite{alex}).

One can  define similar simplicial graphs and semiconvex placements 
by a similar triangular refinement on any a finite polytope,
with accumulation points at $\kappa$ vertices of the polytope.
Consider, for example, the graphs  that arise from the 
double cones over  regular $n$-gons, for $n=4,5,\dots $, by infinite 
triangulation towards one of the high degree vertices. This is illustrated in
Figure \ref{octagon} for $n=4$.
Here we  take the north and south pole vertices
to be located  at  $p_1=(0,0, 1)$ and $p_2=(0,0,-1)$, and the equatorial 
vertices $p_3, \dots ,p_{n+2}$ to be located symmetrically on the unit circle 
about the origin in the
plane $z=0$. The points on the $k^{th}$ latitude occupy the positions
$(1-t_k)p_1+t_kp_j$ where $3\leq j \leq n+2$ and where $(t_k)$ is a decreasing 
sequence tending to zero.
\end{eg}

\begin{eg}\label{eg:diamond}{\bf Diamond polytopes.}
Figure  \ref{delta3} indicates a  bar-joint framework  which is 
associated with a strictly convex compact  polytope $\P_{\rm dia}$, in 
the sense given above. (Its appearance has a passing resemblance to a 
cut diamond). The polytope vertices $p_1, p_2, \dots $ lie on the unit 
sphere and  $\P_{\rm dia}$ is the convex hull of these points together with
the north polar point.
The associated bar-joint framework $(G,p)$ is determined by the edges of
$\P_{\rm dia}$ and the underlying structure graph has refinement-connectivity type $(\rho(G),\kappa(G))=(1,1)$.
Once again, different placements of $G$ are 
possible for different choices of a sequence $(t_k)$, where $t_k$ is 
decreasing and tends to zero, representing the positioning of the 
latitudes below the north pole. In this example
the number of vertices of $G$ for successive latitudes doubles on passing to 
the next highest latitude and it follows that $\dim_{\rm fl}(G)$ is infinite. For further examples of interest one can vary 
this  multiplicity of increase and the spherical surface can be replaced 
by other convex rotationally symmetric surface.
\end{eg}

\begin{center}
\begin{figure}[ht]
\centering
\includegraphics[width=6cm]{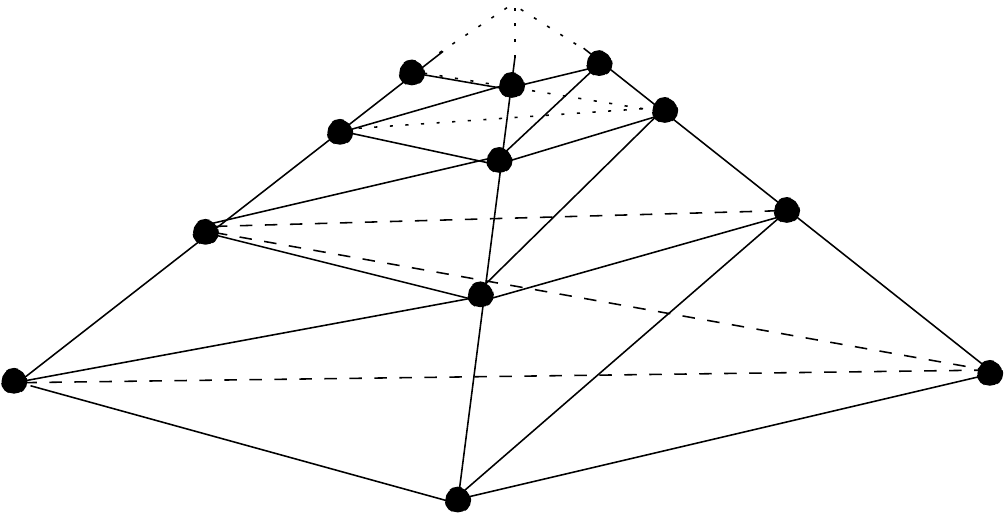}
\caption{Tetrahedral placement $\D_3=(G,p)$ of a simplicial graph $G$ 
with $(\kappa(G),\rho(G)) =(0,1)$.}
\label{trianglepyramid}
\end{figure}
\end{center}

\begin{center}
\begin{figure}[ht]
\centering
\includegraphics[width=5cm]{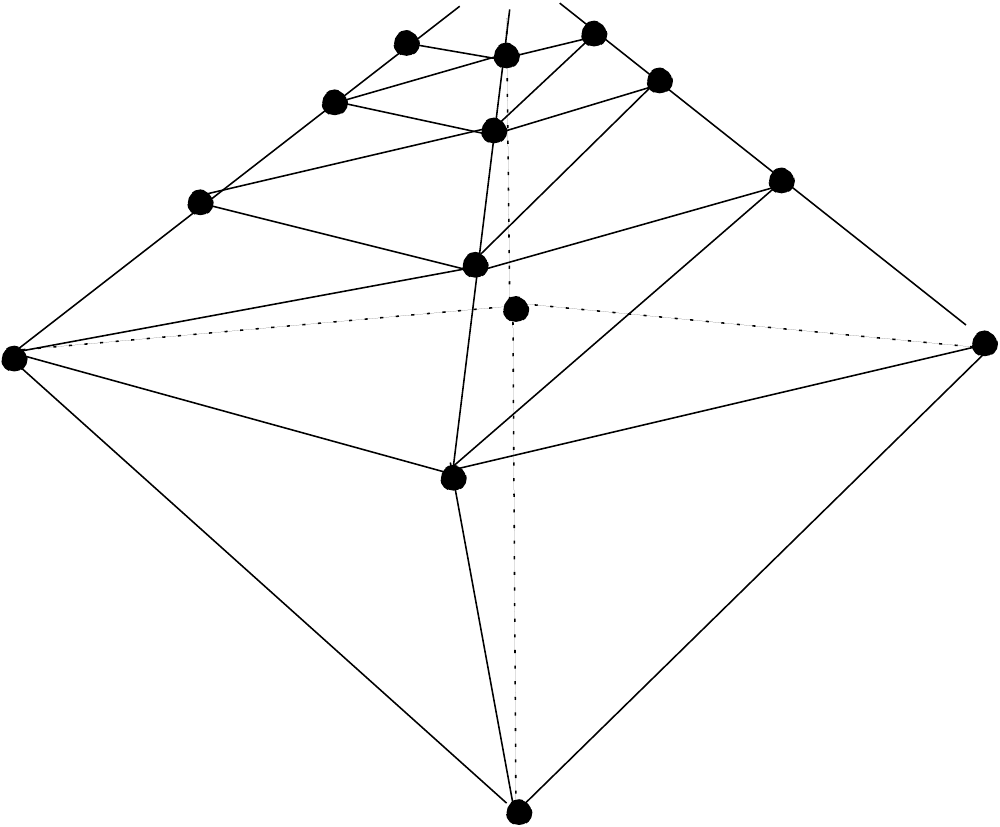}
\caption{A semi-convex placement $\D_4=(G,p)$ whose polytope body $\P(G,p)$ is 
an octahedron.
}
\label{octagon}
\end{figure}
\end{center}

\begin{center}
\begin{figure}[ht]
\centering
\includegraphics[width=5.5cm]{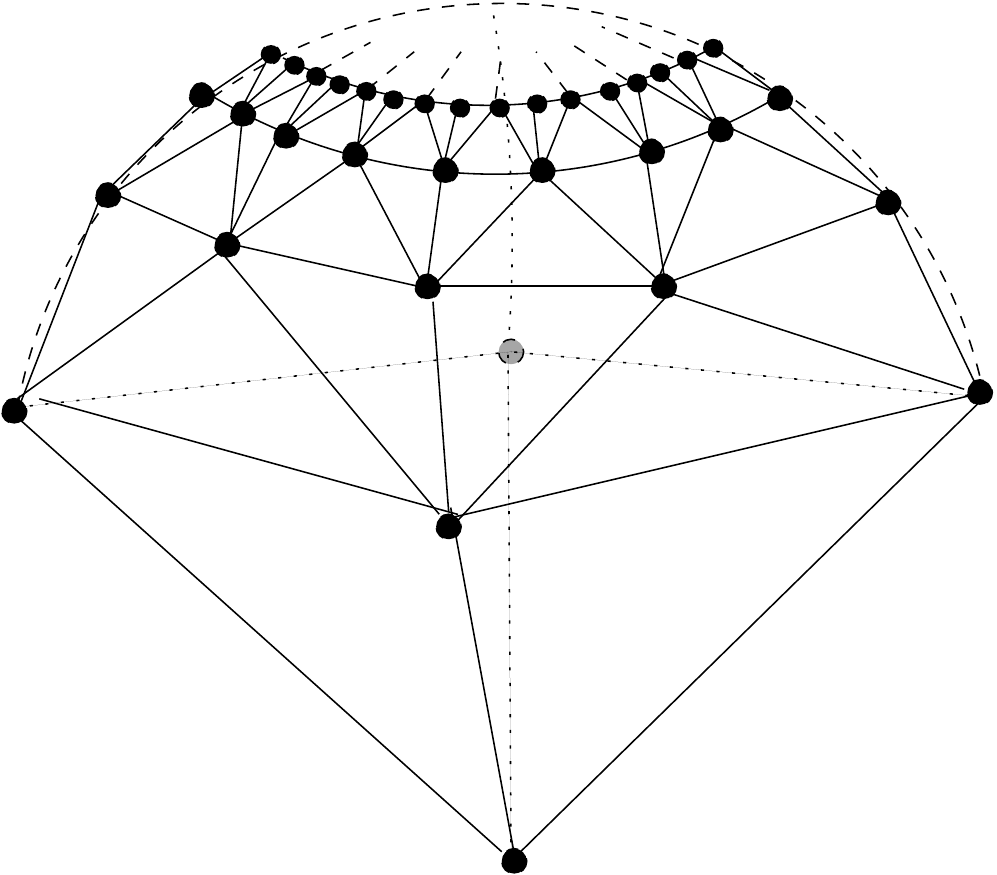}
\caption{A diamond polytope.}
\label{delta3}
\end{figure}
\end{center}

\subsection{Remarks.}It seems possible that Proposition \ref{ConstructionProp}, asserting that the sequential construction of a Laman graph is possible from \emph{any} Laman subgraph, is a known fact but we are unaware of a reference. However our proof provides a model for the $(2,2)$-tight counterpart given in Proposition \ref{NonEuclideanConstructionProp}.

In Nixon, Owen and Power \cite{NOP} (see also \cite{NOP2}) the constructive characterisation of finite $(2,2)$-tight graphs was obtained in order to characterise the generically rigid finite graphs with respect to placements
on a circular cylinder in $\bR^3$. In this setting the placements are viewed as movably attached to the cylinder so that the admissible flexes 
(and velocity fields, in the sense of Section \ref{ContinuousRigidityVelocityField}) are tangential to the cylinder. It follows from the arguments above that the countable simple graphs that are minimally rigid for the cylinder in this sense are the sequential $(2,2)$-tight graphs.

A recent survey of inductive methods for bar-joint frameworks
is given in Nixon and Ross \cite{nix-ros}. 

\section{Rigidity of multi-body graphs}
\label{Body-bar}
Tay's theorem \cite{tay} provides a combinatorial characterisation of the finite multi-graphs without reflexive edges  which have infinitesimally rigid generic realisations as  body-bar frameworks in Euclidean space. In this section our main goals are to extend Tay's characterisation to countable multi-graphs and to obtain analogues of both characterisations for the non-Euclidean $q$-norms for all dimensions $d\geq 2$.

\subsection{Tay's theorem and non-Euclidean rigidity}
We now consider bar-joint frameworks in $(\bR^d,\|\cdot\|_q)$, where $q\in (1,\infty)$, which arise from the following class of simple graphs.

\begin{defn}
A {\em multi-body graph} for $(\bR^d,\|\cdot\|_q)$ is a finite or countable simple graph $G$ for which there exists a vertex partition \[V(G)=\bigcup_k \,V_k\]  
consisting of a finite or countable collection of subsets $V_k$ such that  for each $k$,
\begin{enumerate}
\item the vertex-induced subgraph determined by $V_k$ is a rigid graph in $(\mathbb{R}^d,\|\cdot\|_q)$, and,
\item every vertex $v\in V_k$ is adjacent to at most one vertex in $V(G)\backslash V_k$.
\end{enumerate} 
\end{defn}
 
 The rigid vertex-induced subgraph determined by $V_k$ is denoted $B_k$ and is called a {\em body} of $G$.
An edge $vw\in E(G)$ which is incident with vertices from two distinct bodies $B_i$ and $B_j$ is called an {\em inter-body edge}.
Thus a multi-body graph is composed of pairwise vertex-disjoint bodies together with inter-body edges such that no pair of inter-body edges of $G$ share a vertex.
 
Each multi-body graph $G$ has an associated finite or countable {\em body-bar graph} $G_b=(V(G_b),E(G_b))$ which is the multi-graph  with vertex set labelled by the  bodies of $G$ 
and with  edge set derived from the inter-body edges of $G$. 

Tay's theorem may be restated as follows.

\begin{thm}[Tay, 1984]
\label{t:Tay}
Let $G$ be a finite multi-body graph for $(\bR^d,\|\cdot\|_2)$ and suppose that $G$ contains at least two bodies. Then the following statements are equivalent.
\begin{enumerate}[(i)]
\item $G$ is rigid in Euclidean space $(\bR^d,\|\cdot\|_2)$.
\item $G_b$ contains a $\left(\frac{d(d+1)}{2},\frac{d(d+1)}{2}\right)$-tight spanning subgraph.
\end{enumerate} 
\end{thm}

The following lemma shows that the bodies $B_1,B_2,\ldots $ of a multi-body graph $G$ may be modeled in a number of different ways without altering the rigidity properties of $G$. 

\begin{lem}
\label{BodyLemma}
Let $G$ and $G'$ be two finite multi-body graphs  for $(\bR^d,\|\cdot\|_q)$ with isomorphic body-bar graphs and $q\in (1,\infty)$.
Then $\dim_{d,q}(G)=\dim_{d,q}(G')$.
\end{lem}

\proof
Choose a multi-body graph $H$ with body-bar graph $H_b$ isomorphic to $G_b$ and $G'_b$ such that each body  of $H$ is a complete graph with more vertices than the corresponding bodies of $G$ and $G'$.
Then there exist natural graph homomorphisms $\phi:G\to H$ and $\phi':G'\to H$.
If $p_H:V(H)\to \bR^d$ is a generic placement of $H$ then $p:V(G)\to \bR^d$  defined by $p_v=(p_H)_{\phi(v)}$ is a generic placement of $G$.
Now the linear mapping $A:\F_q(H,p_H)\to\F_q(G,p)$, $A(u)_v=u_{\phi(v)}$ is an isomorphism.
Applying the same argument to $G'$ we obtain a generic placement $p:V(G')\to \bR^d$ and a linear isomorphism $A':\F_q(H,p_H)\to\F_q(G',p')$. The result follows.
\endproof

\begin{eg}
The complete graph $K_{d+1}$ is $(d,\frac{d(d+1)}{2})$-tight and is minimally rigid for $(\bR^d,\|\cdot\|_2)$.
The complete graph $K_{2d}$ is $(d,d)$-tight and  is a minimally rigid graph for  $(\bR^d,\|\cdot\|_q)$ for each of the non-Euclidean $\ell^q$-norms.
These sparsity and rigidity properties persist for graphs obtained from these complete graphs by a finite sequence of Henneberg vertex extension moves of degree $d$. Thus we may assume without loss of generality that the  bodies of a finite multi-body graph for  $(\bR^d,\|\cdot\|_q)$ are $(d,\frac{d(d+1)}{2})$-tight in the Euclidean case and $(d,d)$-tight 
in the non-Euclidean case. The convenience of modeling multi-body graphs in this way is that the combinatorial and $\ell^q$-norm analysis of earlier sections is ready-to-hand.
\end{eg}

There is a natural vertex-induced surjective graph homomorphism $\pi:G \to \bar{G}_b$ where $\bar{G}_b$ is the multi-graph obtained by contracting the bodies of $G$. The body-bar graph $G_b$ is a subgraph of $\bar{G}_b$
obtained by removing reflexive edges and  $\pi$ gives a bijection between the inter-body edges of $G$ and the edges of $G_b$.

\begin{thm}\label{t:q-Tay} Let $G$ be a finite multi-body graph for $(\bR^d,\|\cdot\|_q)$ where $q\in (1,2)\cup(2,\infty)$. Then the following statements are equivalent.
\begin{enumerate}[(i)]
\item $G$ is  rigid in  $(\bR^d,\|\cdot \|_q)$.
\item $G_b$ has a  $(d,d)$-tight spanning subgraph.
\end{enumerate} 
\end{thm}

\proof
$(i)\Rightarrow (ii)$
We can assume without loss of generality that each body of $G$ is $(d,d)$-tight.
Suppose that $G$ is minimally rigid in  $(\bR^d,\|\cdot \|_q)$ with bodies $B_1,B_2,\ldots,B_n$.
If $G_b$ is the body-bar graph for $G$ then $|V(G_b)|=n $ and we have
\begin{eqnarray*}
|E(G_b)| &=& |E(G)|-\sum_{i=1}^{n}|E(B_i)| \\
&=&  (d|V(G)|-d)-\sum_{i=1}^{n}(d|V(B_i)|-d) \\ 
&=& d|V(G_b)|-d
\end{eqnarray*}
Let $H_b$ be a subgraph of $G_b$ and let $\pi:G\to  \bar{G}_b$ be the natural graph homomorphism.
Define $H$ to be the subgraph of $G$ with $V(H)=\pi^{-1}(V(H_b))$ such that $H$ contains the body $B_i$ whenever $\pi(V(B_i))\in V(H_b)$ and $H$ contains the inter-body edge $vw$ whenever $\pi(v)\pi(w)\in E(H_b)$.
Then $|V(H_b)|=|\I|$ where  $\I=\{i\in \{1,\ldots,n\}: B_i\subset H\}$ and 
\begin{eqnarray*}
|E(H_b)| &=& |E(H)|-\sum_{i\in \I}|E(B_i)| \\
&\leq& (d|V(H)|-d)-\sum_{i\in\I}(d|V(B_i)|-d) \\
&=& d|V(H_b)|-d
\end{eqnarray*}
Thus $G_b$ is $(d,d)$-tight.
For the general case note that by removing edges from $G$ we obtain a minimally rigid multi-body graph $\tilde{G}$. Thus by the above argument $\tilde{G}_b$ is a vertex-complete $(d,d)$-tight subgraph of $G_b$.

$(ii)\Rightarrow (i)$ If $G_b$ is $(d,d)$-tight then 
it admits a partition as an edge-disjoint union of $d$ spanning trees 
$T_1,T_2,\ldots,T_d$ (see \cite{N-W_1964}).
We will construct a placement of $G$ such that $p_v-p_w$ lies on the $i^{th}$ coordinate axis in $\bR^d$
whenever $vw$ is an inter-body edge with $\pi(vw)\in T_i$.  

By Lemma \ref{BodyLemma} we can assume that the bodies $B_1,B_2,\ldots,B_n$ of $G$ are copies of the complete graph $K_m$  for some sufficiently large $m$.
Let $p_1:V(B_1)\to\bR^d$ be a generic placement of the body $B_1$ and 
define inductively the placements $p_k:V(B_k)\to \bR^d$ for $k=2,\ldots,n$ so that 
\begin{itemize}
\item $p_k(V(B_k))=p_1(V(B_1))$, and,
\item $p_j(v)=p_k(w)$ whenever $j<k$ and $vw\in E(G)$ is an inter-body edge
with $v\in V(B_j)$ and $w\in V(B_k)$.
\end{itemize}
Then $(B_k,p_k)$ is a generic, and hence infinitesimally rigid, bar-joint framework for each $k=1,2,\ldots,n$.

Define $p:V(G)\to \bR^d$ by setting $p(v)=p_i(v)$ whenever $v\in V(B_i)$.
Note that $p$ is not a placement of $G$ since $p_v=p_w$ for each inter-body edge $vw\in E(G)$. 
However, by perturbing $p$ by a small amount we can obtain a placement $p'$.
Let $\epsilon>0$ and let $e_1,e_2,\ldots,e_d$ be the usual basis in $\bR^d$.
If $v\in V(G)$ is not incident with an inter-body edge then set $p'_v=p_v$.
If $vw\in E(G)$ is an  inter-body edge  and $\pi(vw)\in T_i$ then let 
$p'_v=p_v+\epsilon e_i$ and $p'_w=p_w$. 
The rigidity matrix for  $(G,p')$ has the form,
\[R_q(G,p')=\left[ \begin{array}{ccc}
R_q(B_1,p') & & \\
&\ddots&\\
&&R_q(B_n,p') \\\hline
&Z &
\end{array} \right]\]
where the rows of  the submatrix $Z$ correspond to the inter-body edges in $G$. 

Suppose $u=(u_1,\ldots,u_n)\in\ker R_q(G,p')$. For a sufficiently small $\epsilon$ each subframework $(B_i,p')$ is infinitesimally rigid,
and so $u_i=(a_i,\ldots,a_i)$ for some $a_i\in \bR^d$.
If $vw$ is an inter-body edge with $\pi(vw)\in T_i$ then the corresponding row entries in $R_q(G,p')$  are non-zero in the 
$p_{v,i}$ and $p_{w,i}$ columns only.
The spanning tree property now ensures that $a_1=\cdots=a_n$ and so 
the kernel of $R_q(G,p')$ has dimension $d$.
Thus $p'$ is an infinitesimally  rigid placement of $G$  in  $(\bR^d,\|\cdot \|_q)$. More generally if $G_b$ contains a vertex-complete $(d,d)$-tight subgraph then by the above argument $G$ is rigid  in  $(\bR^d,\|\cdot \|_q)$.
\endproof

We will require the following definition and corollary to characterise the countable rigid multi-body graphs for
$(\bR^d,\|\cdot\|_q)$. 

\begin{defn}
A multi-body graph for $(\bR^d,\|\cdot\|_q)$ is {\em essentially minimally rigid} if it is rigid and removing any inter-body edge 
results in a multi-body graph which is not rigid.
\end{defn}

\begin{cor}\label{MinTay}
Let $G$ be a finite multi-body graph for $(\bR^d,\|\cdot\|_q)$ and suppose that one of the following conditions holds.
\begin{enumerate}[(a)]
\item $q=2$ and $k=\frac{d(d+1)}{2}$, or,
\item $q\in(1,2)\cup(2,\infty)$ and $k=d$.
\end{enumerate} 
Then the following statements are equivalent.
\begin{enumerate}[(i)]
\item $G$ is essentially minimally rigid in  $(\bR^d,\|\cdot \|_q)$.
\item $G_b$ is a $(k,k)$-tight multi-graph.
\end{enumerate} 
\end{cor}

\proof
The proof  follows  immediately from Theorem \ref{t:Tay} in case (a) and from Theorem \ref{t:q-Tay} in case (b).
\endproof

\subsection{Rigidity of countable multi-body graphs}
We are now able to characterise the countable rigid multi-body graphs in $(\bR^d,\|\cdot\|_q)$ for all dimensions $d\geq 2$ and all $q\in (1,\infty)$. 

Given a finite bar-joint framework $(G,p)$ in $(\bR^d,\|\cdot\|_q)$
we denote by $X_{row}(G,p)$ the row space of the rigidity matrix $R_q(G,p)$.

\begin{defn}
A finite multi-body graph for $(\bR^d,\|\cdot\|_q)$  is {\em essentially independent} 
if given any generic placement $p\in\P(G)$ the row space of the rigidity matrix $R_q(G,p)$ 
may be expressed as a direct sum
\[X_{row}(G,p) = X_{B_1}\oplus\cdots\oplus X_{B_n}\oplus X_{IB}\]
where $X_{B_i}$ is the subspace of $X_{row}(G,p)$ spanned by the rows of $R_q(G,p)$ which correspond to the edges of the body $B_i$ and $X_{IB}$ is the subspace  spanned by the rows  which correspond to the inter-body edges of $G$.
\end{defn}

The following result is an analogue of Proposition \ref{IndepLemma}.

\begin{prop}
\label{IndepBodyLemma}
Let $G$ be a finite multi-body graph for $(\bR^d,\|\cdot\|_q)$  and suppose that one of the following conditions holds.
\begin{enumerate}[(a)]
\item $q=2$, $k=\frac{d(d+1)}{2}$ and $G$ contains at least $d(d+1)$ vertices, or,
\item $q\in(1,2)\cup(2,\infty)$, $k=d$ and $G$ contains at least $2d$ vertices.
\end{enumerate}
Then the following statements are equivalent.
\begin{enumerate}[(i)]
\item $G$ is essentially independent with respect to $(\mathbb{R}^d,\|\cdot\|_q)$.
\item $G_b$ is $(k,k)$-sparse.
\end{enumerate}
\end{prop}

\proof
Suppose $G$ is essentially independent and  let $p:V(G)\to \mathbb{R}^d$ be a generic placement of $G$.
If  $H_b$ is a subgraph of $G_b$  and $B_1,B_2,\ldots,B_n$ are the bodies of $G$ then let $H$ be the  subgraph of $G$ with $V(H)=\pi^{-1}(V(H_b))$ such that $H$ contains the body $B_i$ whenever $\pi(V(B_i))\in V(H_b)$ and $H$ contains the inter-body edge $vw$ whenever $\pi(vw)\in E(H_b)$.
If  $\I=\{i\in \{1,\ldots,n\}: B_i\subset H\}$
then
\begin{eqnarray*}
|E(H_b)| &=& \rank R_q(H,p)-\sum_{i\in \I}\rank R_q(B_i,p)\\ &\leq& (d|V(H)|-k)-\sum_{i\in\I}(d|V(B_i)|-k) \\
&=& k|V(H_b)|-k
\end{eqnarray*}
Thus $G_b$ is $(k,k)$-sparse.

Conversely, if $G_b$ is $(k,k)$-sparse then
by Proposition \ref{SparseLemma1} $G_b$ is a vertex-complete subgraph of a $(k,k)$-tight multi-graph $G'_b$ which has no reflexive edges.  
Let $G'$ be a multi-body graph with body-bar graph isomorphic to $G'_b$ and which contains $G$ as a subgraph.
By Corollary \ref{MinTay}, 
$G'$ is essentially minimally rigid
and it follows that $G$ is essentially independent. 
\endproof

We now prove an analogue of Theorem \ref{RelRigidProp} which shows that in the category of multi-body graphs relative rigidity is equivalent to the existence of a rigid container for all dimensions $d$ and for all $\ell^q$ norms.

\begin{thm}
\label{RelRigidBody}
Let $G$ be a finite multi-body graph for $(\bR^d,\|\cdot\|_q)$ and let $H$ be a subgraph of $G$ which is a multi-body graph whose body subgraphs are bodies of $G$. 
Suppose that  one of the following conditions holds.
\begin{enumerate}[(a)]
\item $q=2$ and $H$ contains at least $d(d+1)$ vertices, or,
\item $q\in(1,2)\cup(2,\infty)$  and $H$ contains at least $2d$ vertices.
\end{enumerate}
Then the following statements are equivalent.
\begin{enumerate}[(i)]
\item $H$ is relatively  rigid in $G$ with respect to $(\mathbb{R}^d,\|\cdot\|_q)$.
\item $H$ has a rigid container in $G$ with respect to $(\mathbb{R}^d,\|\cdot\|_q)$ which is a multi-body graph.
\end{enumerate}

\end{thm}

\proof
$(i)\Rightarrow(ii)$
Consider first the case when $G$ is essentially  independent with respect to $(\mathbb{R}^d,\|\cdot\|_q)$.  
By Proposition \ref{IndepBodyLemma}, the body-bar graph $G_b$ is $(k,k)$-sparse.
Let $\pi(v),\pi(w)\in V(H_b)$ be distinct vertices of $H_b$ with $\pi(vw)\notin E(H_b)$.
By enlarging the bodies of $G$ and $H$ we can assume without loss of generality that there exist representative vertices 
$v,w\in V(H)$ such that $v$ and $w$ are not incident with any inter-body edges of $G$.
Since $K_{V(H)}$ is rigid in $(\mathbb{R}^d,\|\cdot\|_q)$ the relative rigidity property implies that 
\[\F_q(G,p)=\F_q(G\cup K_{V(H)},p)\] for every generic placement $p$.
It follows that $G'=G\cup \{vw\}$ is a multi-body graph which is not essentially independent. 
Note that $G'$ has the same bodies as $G$ and
so by Proposition \ref{IndepBodyLemma}, the associated body-bar graph $(G')_b=G_b\cup\{\pi(vw)\}$ is not $(k,k)$-sparse where $\pi:G\to \bar{G}_b$ is the natural graph homomorphism.
Thus by Lemma \ref{SparseLemma2} there exists a $(k,k)$-tight subgraph $(H_{v,w})_b$ of $G_b$ with $\pi(v),\pi(w)\in V((H_{v,w})_b)$.

Define $H'_b$ to be the union of $H_b$ together with the subgraphs $(H_{v,w})_b$ for all such pairs $\pi(v),\pi(w)$.
By Proposition \ref{SparseLemma1}, $H'_b$ is $(k,k)$-tight.
Let $H'$ be the induced multi-body subgraph of $G$ with body-bar graph isomorphic to $H'_b$.
By Corollary \ref{MinTay}   $H'$ is rigid in $(\mathbb{R}^d,\|\cdot\|_q)$ and so $H'$ is a rigid container for $H$ in $G$.

If $G$ is not essentially independent then let $p:V(G)\to\mathbb{R}^d$ be a generic placement of $G$.
There exists an inter-body edge $vw\in E(G)$ such that 
 \[\ker R_q(G,p)=\ker R_q(G\backslash vw,p)\]
Let $G_{1}=G\backslash vw$ and $H_{1}=H\backslash vw$ and note that $H_{1}$ is relatively  rigid in $G_{1}$. Continuing to remove edges in this way we arrive after finitely many iterations at  subgraphs $H_{n}$ and $G_{n}$ such that $V(H_n)=V(H)$, $H_{n}$ is relatively  rigid in $G_{n}$ and $G_{n}$ is essentially independent.
By the above argument there exists a rigid container $H'_{n}$ for $H_{n}$ in  $G_{n}$.
Now $H' = H'_{n} \cup H$ is a rigid container for $H$ in $G$. 

$(ii)\Rightarrow (i)$ If $H$ has a rigid container $H'$ in $G$ and $p:V(G)\to\bR^d$ is a generic placement of $G$ 
then no non-trivial  infinitesimal flex of $(H,p)$ extends to $(H',p)$. The result follows.
\endproof

We now  prove the equivalence of rigidity and sequential rigidity for multi-body graphs with respect to all $\ell^q$-norms and in all dimensions $d\geq 2$.

\begin{thm}\label{t:q-Tayinfinite}
Let $G$ be a  countable multi-body graph for $(\bR^d,\|\cdot \|_q)$ where $q\in (1,\infty)$. 
The following statements are equivalent.
\begin{enumerate}[(i)]
\item $G$ is rigid in $(\bR^d,\|\cdot \|_q)$.
\item $G$ is sequentially rigid in $(\bR^d,\|\cdot \|_q)$.
\end{enumerate}
\end{thm}

\proof
 Suppose $G$ is rigid in $(\mathbb{R}^d,\|\cdot\|_q)$
and let $p:V(G)\to\mathbb{R}^d$  be a generic placement.
By Theorem \ref{t:IR}, there exists a vertex-complete tower $\{(G_k,p_k):k\in\mathbb{N}\}$ in $(G,p)$ which is relatively infinitesimally rigid. Moreover, we can assume that each $G_k$ is a multi-body graph.
By Proposition \ref{RelRigidBody}, $G_k$ has a rigid container $H_k$ in $G_{k+1}$ for each $k\in\mathbb{N}$.
Thus the sequence $\{H_k:k\in\mathbb{N}\}$ is a vertex-complete tower of rigid graphs in $G$.
For the converse apply Corollary \ref{sequential}.
\endproof

\begin{thm}\label{t:q-Tayinfinite2}
Let $G$ be a  countable multi-body graph for $(\bR^d,\|\cdot \|_q)$ where $q\in (1,\infty)$. 
\begin{enumerate}[{\bf (A)}]
\item The following statements are equivalent.
\begin{enumerate}[(i)]
\item $G$ is rigid in $(\bR^d,\|\cdot \|_2)$.
\item $G_b$ has a $\left(\frac{d(d+1)}{2},\frac{d(d+1)}{2}\right)$-tight vertex-complete  tower.
\end{enumerate}
\item If $q\in (1,2)\cup(2,\infty)$ then the following statements are equivalent.
\begin{enumerate}[(i)]
\item $G$ is rigid in $(\bR^d,\|\cdot \|_q)$.
\item $G_b$ has a $(d,d)$-tight vertex-complete  tower.
\end{enumerate}
\end{enumerate}
\end{thm}

\proof
$(i)\Rightarrow (ii)$
If $G$ is rigid then by Theorem \ref{t:q-Tayinfinite}, $G$ is sequentially rigid.
Let $\{G_k:k\in\mathbb{N}\}$ be a vertex-complete tower of rigid subgraphs in $G$ and let $B_1,B_2,\ldots$ be the bodies of $G$.
We may assume that each $G_k$ is a multi-body graph.
Applying the induction argument used in Theorem \ref{InfiniteLaman2} we construct a vertex-complete tower of essentially minimally rigid multi-body subgraphs in $G$.
To do this let $H_1$ be the multi-body graph obtained by taking all bodies which lie in $\tilde{G}_1$ and adjoining inter-body edges of $G_1$ until an essentially minimally rigid graph is reached.
The induced sequence of body-bar graphs $\{(H_{k})_b:k\in\mathbb{N}\}$ is a vertex-complete tower in $G_b$.
By Corollary \ref{MinTay}  each body-bar graph $(H_k)_b$ is $\left(\frac{d(d+1)}{2},\frac{d(d+1)}{2}\right)$-tight  in case (A) and $(d,d)$-tight in case (B).

$(ii)\Rightarrow (i)$ 
Let $\{G_{k,b}:k\in\mathbb{N}\}$ be a $\left(\frac{d(d+1)}{2},\frac{d(d+1)}{2}\right)$-tight vertex-complete tower in $G_b$ and let $\pi:G\to  \bar{G}_b$ be the natural graph homomorphism.
Define $G_k$ to be the subgraph of $G$ with $V(G_k)=\pi^{-1}(V(G_{k,b}))$ such that $G_k$ contains the body $B_i$ whenever $\pi(V(B_i))\in V(G_{k,b})$ and $G_k$ contains the inter-body edge $vw$ whenever $\pi(vw)\in E(G_{k,b})$.
Then $G_{k,b}$ is the body-bar graph for $G_k$ and so $G_k$ is rigid by Theorem \ref{t:Tay}.
Thus $\{G_k:k\in\mathbb{N}\}$ is a vertex-complete  tower of rigid subgraphs in $G$ and so $G$ is sequentially rigid.
By Theorem \ref{t:q-Tayinfinite}, $G$ is rigid.

To prove (B) we apply similar arguments to the above using  the non-Euclidean versions of the relevant propositions and substituting Theorem \ref{t:q-Tay} for Theorem \ref{t:Tay}.  
\endproof

\begin{cor}\label{MinInfiniteTay}
Let $G$ be a  countable multi-body graph for $(\bR^d,\|\cdot \|_q)$ where $q\in (1,\infty)$. 
\begin{enumerate}[{\bf (A)}]
\item The following statements are equivalent.
\begin{enumerate}[(i)]
\item $G$ is essentially minimally rigid in $(\bR^d,\|\cdot \|_2)$.
\item $G_b$ has a $\left(\frac{d(d+1)}{2},\frac{d(d+1)}{2}\right)$-tight edge-complete  tower.
\end{enumerate}
\item If $q\in (1,2)\cup(2,\infty)$ then the following statements are equivalent.
\begin{enumerate}[(i)]
\item $G$ is essentially minimally rigid in $(\bR^d,\|\cdot \|_q)$.
\item $G_b$ has a $(d,d)$-tight edge-complete  tower.
\end{enumerate}
\end{enumerate}
\end{cor}

\proof
$(i)\Rightarrow (ii)$ If $G$ is essentially minimally rigid in $(\mathbb{R}^d,\|\cdot\|_q)$ then
by Theorem \ref{t:q-Tayinfinite2}, $G_b$ contains a $(k,k)$-tight vertex-complete tower $\{G_k:k\in\bN\}$ and this tower must be edge-complete. 

$(ii)\Rightarrow (i)$ 
If $G_b$ contains a $(k,k)$-tight edge-complete tower $\{G_{k,b}:k\in\bN\}$ then by Theorem \ref{t:q-Tayinfinite2}, $G$ is rigid. 
Let $vw\in E(G)$ be an inter-body edge and suppose $G\backslash \{vw\}$ is rigid.
By Theorem \ref{t:q-Tayinfinite} $G\backslash \{vw\}$ is sequentially rigid and so 
there exists a vertex-complete tower $\{H_k:k\in\bN\}$ in $G\backslash \{vw\}$ consisting of rigid subgraphs.
Moreover, we can assume that each $H_k$ is a multi-body graph.
Choose a sufficiently large $k$ such that  $v,w\in V(H_k)$ and choose a sufficiently large $n$ such that
$vw\in E(G_n)$ and $H_k$ is a subgraph of $G_n$. Then the body-bar graph for $H_k\cup\{vw\}$ is a subgraph of $(G_n)_b$ which fails the sparsity count for $(G_n)_b$. 
We conclude that $G\backslash \{vw\}$ is not rigid in $(\mathbb{R}^d,\|\cdot\|_q)$ for all $vw\in E(G)$. 
\endproof

\subsection{Remarks}
A key feature of body-bar frameworks is the nonincidence condition for the bars. This makes available special realisations which are rigid, as we have seen in the proof of the analogue of Tay's theorem, Theorem \ref{t:q-Tay}. 
Other instances of this can be seen in the 
matroid analysis of Whiteley \cite{whi-union} and in the analysis of
Borcea and Streinu \cite{bor-str-2} and Ross \cite{ros-kavli} of locally finite graphs with periodically rigid periodic bar-joint frameworks.


\section{Continuous rigidity for countable graphs}
\label{ContinuousRigidityVelocityField}
In this section we extend the Asimow-Roth theorem (Theorem \ref{AsimowRoth}) on the equivalence of infinitesimal rigidity and continuous rigidity for  finite bar-joint frameworks  to the non-Euclidean  $\ell^q$-norms. 
Both directions of this equivalence
are shown to fail for placements of countable graphs in the Euclidean plane.
We also discuss other forms of rigidity beyond absolute infinitesimal rigidity, including  infinitesimal rigidity relative to continuous  velocity fields and continuous rigidity relative to bounded  motions. These forms  are appropriate for the analysis of infinite bar-joint frameworks that sit in a bounded region of space and in particular we show that the infinitely-faceted simplicial polytope graphs of Section \ref{InductiveConstructions} admit various continuously rigid placements.

\subsection{Asimow-Roth theorem and  $\ell^q$ norms.}
\label{AREquivalence}
Let $(G,p)$ be a  bar-joint framework in  a normed vector space $(\bR^d,\|\cdot\|)$.
The \textit{edge-function} $f_G(x)$ for $G$ is the mapping
\[
f_G : \bR^{d|V(G)|} \to \bR^{|E(G)|}, \quad x=(x_v)_{v\in V(G)} \mapsto (\|x_v-x_w\|)_{vw\in E(G)}
\]
and the {\em configuration space} $V(G,p)$ for $(G,p)$ is the pre-image 
\[V(G,p) = \{x\in \bR^{d|V(G)|}:f_G(x)=f_G(p)\}\]
There is a one-to-one correspondence between the continuous flexes $\{\alpha_v:v\in V(G)\}$ of $(G,p)$ and 
continuous paths in the configuration space $\alpha:[-1,1]\to V(G,p)$ which are based at $p$. 
The following lemma establishes the connection between the configuration space and the space of infinitesimal flexes $\F(G,p)$. 

\begin{lem}
\label{t:rigiditythm1}
Let $(G,p)$ be a finite bar-joint framework in  $(\bR^d,\|\cdot\|)$.
Suppose there exists  a neighbourhood $U$ of $p$ on which $f_G(x)$ is a $C^1$ mapping and $df_G(x)$ has constant rank $k$.
Then, 
\begin{enumerate}[(i)]
\item $V(G,p)\cap U$ is a $C^1$-manifold of codimension $k$ in $\bR^{d|V(G)|}$.
\item $\F(G,p)$ is linearly isomorphic to the tangent space to $V(G,p)\cap U$ at $p$. 
\end{enumerate}
\end{lem}

\proof
$(i)$ Apply the rank theorem \cite[Theorem 1.3.14]{Narasimhan}.

$(ii)$
A mapping $u:V(G)\to \bR^d$ is an infinitesimal flex of $(G,p)$ if and only if $D_u f_G(p)=0$
where  $D_uf_G$ denotes the directional derivative of $f_G(x)$ in the direction of $u=(u_{v_1},\ldots,u_{v_n})\in \bR^{d|V(G)|}$.
Since $f_G(x)$ is differentiable at $p$ we have $D_uf_G(p)=df_G(p)u$ and so the result follows.
\endproof

The following lemma provides a characterisation for the trivial continuous flexes of a bar-joint framework with respect to the non-Euclidean $\ell^q$ norms.

\begin{lem}
\label{ContinuousRigidMotions}
Let $\{\alpha_v:v\in V(G)\}$ be a continuous flex of a bar-joint framework $(G,p)$ in $(\bR^d,\|\cdot\|)$.
If $(\bR^d,\|\cdot\|)$ admits only  finitely many linear isometries $T:\bR^d\to \bR^d$
then the following statements are equivalent.
\begin{enumerate}[(i)]
\item
$\{\alpha_v:v\in V(G)\}$ is a trivial continuous flex of $(G,p)$.
\item There exists $\delta>0$ such that $\alpha_v(t)=p_v+c(t)$ for all $t\in(-\delta,\delta)$ and all $v\in V(G)$
where $c:[-1,1]\to \bR^d$ is a continuous path in $(\bR^d,\|\cdot\|)$.
\end{enumerate}
\end{lem}

\proof
If $\{\alpha_v:v\in V(G)\}$ is a trivial continuous flex of $(G,p)$ then there exists a continuous rigid motion $\Gamma(x,t):\bR^d\times [-1,1]\to \bR^d$ with
$\Gamma(p_v,t)=\alpha_v(t)$ for all $v\in V(G)$ and all $t\in[-1,1]$.
By the Mazur-Ulam theorem, $\Gamma_t:\bR^d\to \bR^d$, $x\mapsto \Gamma(x,t)$ is an affine transformation and so there exists a linear map $A_t$ and a vector $c(t)\in \bR^d$ with $\Gamma_t(x)=A_t(x)+c(t)$ for all $x\in \bR^d$.
Since $(\bR^d,\|\cdot\|)$ admits only finitely many linear isometries it follows 
 that $A_t=I$ for all sufficiently small values of $t$.
Conversely, if $(ii)$ holds then
$\Gamma(x,t):\bR^d\times [-1,1]\to \bR^d$, $(x,t)\mapsto x+c(t)$ defines a continuous rigid motion of $\bR^d$ which induces  the continuous flex $\{\alpha_v:v\in V(G)\}$. 

\endproof

We now prove the equivalence of continuous and infinitesimal rigidity with respect to the non-Euclidean $\ell^q$ norms on $\bR^d$
 for regular placements of a finite graph $G$  which lie in the complement of the variety 
 \[\V(G):=\left\{x\in\mathbb{R}^{d|V(G)|}: \prod_{vw\in E(G)}\,\prod_{i=1}^d \, (x_{v,i}-x_{w,i})=0\right\}\]

\begin{thm}\label{t:rigiditythm2}
Let $(G,p)$ be a finite bar-joint framework in $(\bR^d,\|\cdot\|_q)$ where $q\in (1,2)\cup(2,\infty)$.
If $p\notin \V(G)$ then the following statements are equivalent.
\begin{enumerate}[(i)]
\item
$(G,p)$ is continuously rigid and regular.
\item
$(G,p)$ is infinitesimally rigid.
\end{enumerate}
\end{thm}

\proof
$(i)\Rightarrow (ii)$.
As $p\notin \V(G)$ there exists a neighbourhood of $p$ on which $f_G(x)$ is a $C^1$ mapping and, since $(G,p)$ is regular, $df_G(x)$ is constant on a neighbourhood of $p$. 
By Lemma \ref{t:rigiditythm1}, $V(G,p)\cap U$ is a $C^1$-manifold and we can identify the tangent space at $p$ with $\F(G,p)$. 
If $u=(u_v)_{v\in V(G)}\in \F(G,p)$ then there exists a continuous path $\alpha:[-1,1]\to V(G,p)\cap U$ with 
$\alpha(0)=p$ and $\alpha'(0)=u$. 
Now the collection of component functions $\{\alpha_v:v\in V(G)\}$ is a continuous flex of $(G,p)$.
By Lemma \ref{ContinuousRigidMotions}, there exists a continuous path $c:[-1,1]\to \bR^d$ and a positive real number $\delta$ such that $\alpha_v(t)=p_v+c(t)$ for all $|t|<\delta$.
Hence $u_v=\alpha_v'(0)=c'(0)$ for all $v\in V(G)$ and so $u$ is a trivial infinitesimal flex of $(G,p)$.

$(ii)\Rightarrow (i)$.
Suppose $(G,p)$ is minimally infinitesimally rigid and let $\{\alpha_v:v\in V(G)\}$ be a continuous flex of $(G,p)$.
Choose a vertex $v_0\in V(G)$ and 
define \[\tilde{\alpha}_v:[-1,1]\to V(G,p), \,\,\,\,\,\,\, 
\tilde{\alpha}_v(t)=\alpha_v(t)-(\alpha_{v_0}(t)-p_{v_0})\]
Note that $\{\tilde{\alpha}_v:v\in V(G)\}$ is a continuous flex of $(G,p)$ which  determines a continuous path
$\tilde{\alpha}:[-1,1]\to V(G,p)$, $t\mapsto (\tilde{\alpha}_v(t))_{v\in V(G)}$.
By Proposition \ref{t:rigiditythm1}$(ii)$ we have 
\[\ker df_G(p) = \F(G,p) = \T(G,p)=\{(a,\ldots,a)\in \bR^{d|V(G)|}:
a\in \bR^d\}\] 
and since $(G,p)$ is minimally infinitesimally rigid it follows that
\[|E(G)|=\rank df_G(p)=d|V(G)|-d\] 
Define  \[F:\bR^{d|V(G)|}\to \bR^{d|V(G)|}, \,\,\,\,\,\, x=(x_v)_{v\in V(G)}\mapsto (f_G(x)-f_G(p), \, x_{v_0}-p_{v_0})\]
Then $dF(p)$ is injective and so 
$F(x)$ is injective on a neighbourhood $U$ of $p$.
Note that $F(\tilde{\alpha}(t))=F(p)=0$ and so 
$\tilde{\alpha}(t)=p$ for all $t\in \tilde{\alpha}^{-1}(U)$.
Now $\alpha_v(t)=p_v+(\alpha_{v_0}(t)-p_{v_0})$  for all $t\in \tilde{\alpha}^{-1}(U)$.
Hence by Lemma \ref{ContinuousRigidMotions}, $\{\alpha_v:v\in V(G)\}$ is a trivial continuous flex of $(G,p)$.
We conclude that $(G,p)$ is continuously rigid.
More generally, if $(G,p)$ is infinitesimally rigid then $G$ contains a vertex-complete subgraph $H$ such that $(H,p)$ is minimally infinitesimally rigid.
By the above argument $(H,p)$ is continuously rigid and so $(G,p)$ is also continuously rigid.

\endproof

From this equivalence and our remarks following Theorem \ref{genericCauchy} it follows  that
a generic convex simplicial polyhedron $\P$ is continuously flexible, as a bar-joint framework, in any non-Euclidean spaces. Also $\P$ becomes continuously rigid following the addition of three non-incident internal bars. In analogy with Example \ref{f:non-Euclideanflexes}, for $1<q<\infty$, $q\neq 2$, any connected component of placements of $\P$ with a vertex fixed at the origin is a $3$-dimensional embedded manifold in $\bR^{3|V(G)|}$.

\subsection{Nonequivalence for infinite frameworks.} 
That a countable  infinitesimally flexible framework  can be continuously rigid in a generic realisation can be seen  by exploiting the triangle inequality to arrange that a vertex complete inclusion chain of finite subframeworks has diminishing flexibility, tending to zero. We see polytope examples of this below. On the other hand
the following example shows that a continuously flexible regular framework can be infinitesimally rigid. 

\begin{eg}\label{InfRigCtsFlx}{\bf Infinitesimally rigid and 
continuously flexible.} Consider  the Euclidean planar framework $(G,p)$ suggested by Figure \ref{f:InfRig}.
\begin{center}
\begin{figure}[ht]
\centering
\includegraphics[width=7cm]{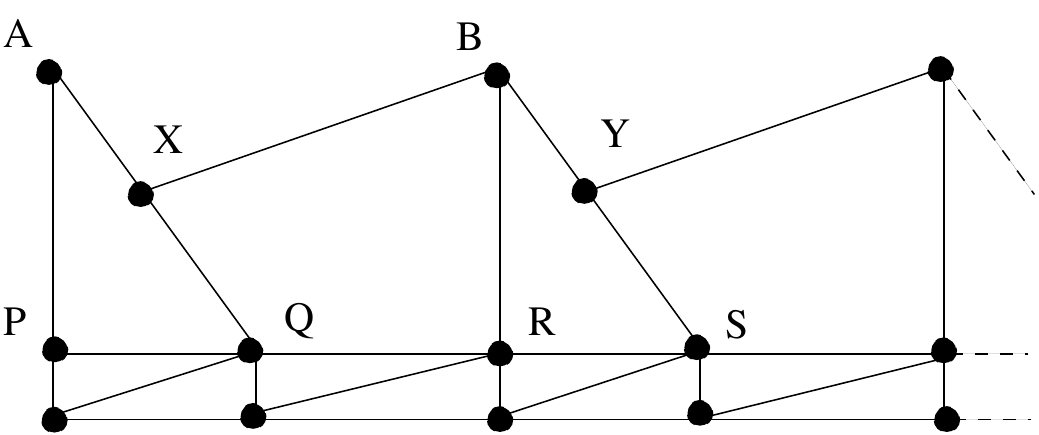}
\caption{Infinitesimally rigid and continuously flexible.}
\label{f:InfRig}
\end{figure}
\end{center}
The triples of points $\{A,X,Q\}, \{B,Y,S\}, \dots $ are  collinear and the placement is  periodic to the right. 
Suppose, by way of contradiction, that $u$ is a nonzero infinitesimal flex 
of  which assigns zero velocities to the
the vertices lying on and below the line through PQRS, with velocities $u_X, u_Y, u_Z, \dots $ for the vertices $X,Y,Z, 
\dots $.  One of these velocities must be nonzero and without 
loss of generality we may assume  $u_X\neq 0$. Thus $u_B\neq 0$ and since $B,Y,S$ are collinear $u_B$ is in the direction of the 
positive $x$-axis. However, $u_S=0$ and $B,Y,S$ are collinear and so 
this is a contradiction since there is no choice of finite velocity $u_Y$ to satisfy the flex conditions for the edges $BY$ and $YS$. 

To see that the framework can be continuously flexible  assume that  $|XQ| < |QR|$ and  $|QB| >|XB| >|RB|$. Consider the 
finite subframework,
$(G_1,\pi)$ say, supported by the labelled vertices and the four vertices below $P,Q,R,S$.
For this subframework, consider the joints $P,Q,R,S$ as fixed. As $AP$ 
rotates
through a clockwise angle $t>0$ the bar $XQ$ rotates clockwise 
achieving a horizontal position for the angle $t=t_1$ say. 
The induced angular position $\theta(t)$ of $BR$  
first achieves a local maximum, $\theta_{max}$, when $QX$ and $XB$ are 
collinear, and then retreats through positive values to a final value 
$\theta_{fin}=\theta(t_1)$. In view of the geometry 
$\theta_{fin} >0$.  Thus the angular range of the continuous function 
$\theta$ is included in the angular range of its argument $t$.
It follows, by iteration of this inclusion principle, that 
the resulting continuous flex $\pi(t)$ of $(G_1,\pi)$, with flex parameter  $t= 
t_1$, is extendible, uniquely, to a continuous flex $p(t)$ of any finite strip and indeed of the entire 
infinite framework. 

One might say, colloquially, that the paradox of this example arises because
we do not admit infinitesimal flexes  with infinite magnitudes while we do
admit continuous flexes $p(t)$ for which $t \to p_k(t)$ may fail to be differentiable at $t=0$ for some joint $p_k$. 
\end{eg}

\subsection{Infinitesimal rigidity with respect to continuous velocity fields}
\label{velocityfield}
There can be placements of countable graphs, including generic placements, which while failing to be infinitesimally rigid are nevertheless infinitesimally rigid with respect to well-behaved or admissible velocity fields. In fact this is already a feature in the rigidity theory of  periodic graphs and crystallographic frameworks where periodicity of some form is required, leading  to properties such as  periodic rigidity, affinely periodic rigidity and phase-periodic rigidity. 

The next definition is based on the requirement that there is a continuous variation of velocity vectors over the framework.
One might say that such frameworks are continuously infinitesimally rigid but for clarity we refrain from doing so.
For the remainder of this section we   consider only placements in Euclidean spaces.

\begin{defn}
Let $(G,p)$ be a countable bar-joint framework in the Euclidean space $\bR^d$
for the placement $p:V(G)\to \bR^d$ and let $\U_c(G,p)$ be the vector space of velocity fields  
$u : V(G) \to \bR^d$ which are continuous for the  topology induced by  $p^{-1}$. 
Then $(G,p)$ is infinitesimally rigid for continuous velocity fields
if
\[
\F(G,p)\cap \U_c(G,p) =\F(G,p)\cap \U_c(G,p)\cap \T(G,p)
\]
\end{defn}

The infinite strip graph $G$ in Figure \ref{strip}, which in fact is the structure graph for  Example \ref{InfRigCtsFlx},
serves to illustrate that the nature of infinitesimal flexes varies considerably according to the geometry of the placement.

\begin{center}
\begin{figure}[ht]
\centering
\includegraphics[width=4.5cm]{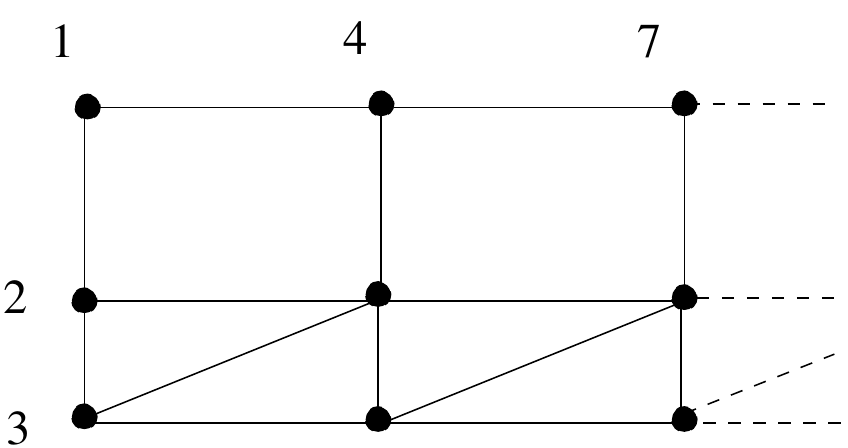}
\caption{The strip graph $G$.}
\label{strip}
\end{figure}
\end{center}

Consider for example the radial placement in which the vertical edges of Figure \ref{strip} are placed on the lines $x=\frac{1}{2^k}$, for $k=1,2,\dots$ and the vertices are placed on the three radial lines $y=-x$, $y=0$ and $y=x$ in the left half-plane. The origin is the unique accumulation  point of the vertex placements.
Then there is  an infinitesimal flex $u = (u_k)$ with $u_k = (1,0)$ for $k=1, 4, 7, \dots $ and $u_k=(0,0)$ otherwise. Any infinitesimal flex in $\F(G,p)$ is a linear combination $\lambda u +\mu w$ with $w $ in $\T(G,p)$ and so it follows that the framework is infinitesimally rigid for continuous velocity fields.

In fact one can construct a  placement $p=(p_k)$ which admits a flex for which the velocities $u_1,u_4,u_7,\dots $  realise an arbitrary sequence of speeds. That is, given positive speeds $s_1, s_4, s_7, \dots $ there is a placement  and an infinitesimal flex $u$ with
$\|u_k\|_2=s_k$ for $k=1, 4, 7, \dots $ and $u_k=(0,0)$ otherwise.

The next example shows how the rigidity matrix can  be used to 
determine infinitesimal rigidity for continuous velocity fields. 
 
\begin{eg}{\bf Whirlpool frameworks.} Let $G$ be the countable graph suggested by Figure \ref{whirlpool}. We define a \emph{whirlpool framework} to be one that arises from
a planar placement of $G$ with a single (finite) accumulation point for the (distinct) vertex placements. Note that the addition of any edge to $G$ gives a countable minimally $2$-rigid  graph and so $\dim_{\rm fl}(G) = 1$. 
If $(G,p)$ is a generic whirlpool framework then $\dim_{\rm fl}(G,p) = 1$ and
every nontrivial infinitesimal flex is uniquely associated with a nontrivial flex of the outer rectangular subframework. Such flexes  $u=(u_k)$ are determined by the rigidity matrix $R_2(G,p)$ and this calculation can be made explicit in the presence of symmetry. We now show this for the  affinely periodic placement indicated in Figure \ref{whirlpool}.
\begin{center}
\begin{figure}[ht]
\centering
\includegraphics[width=4.5cm]{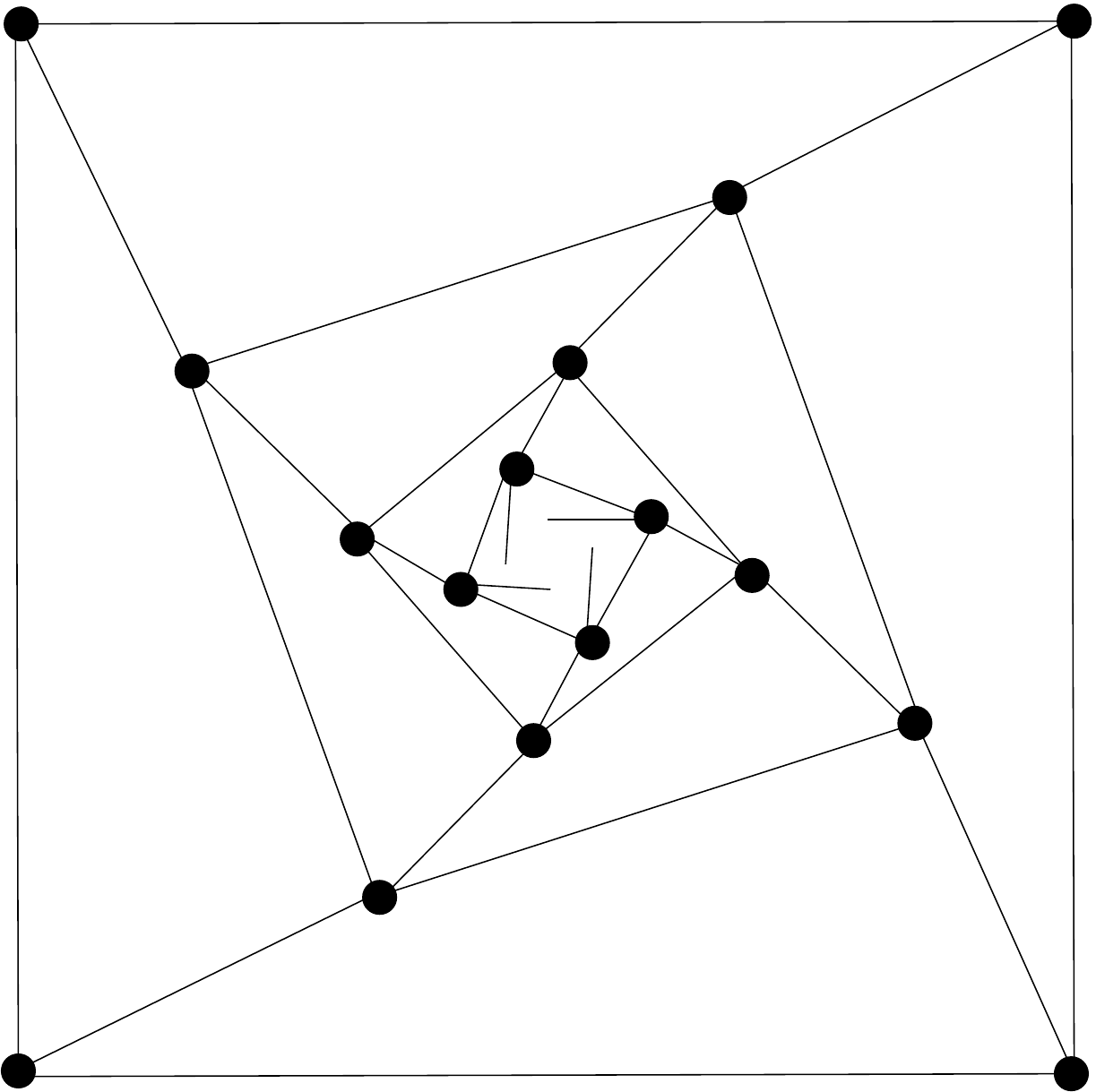}
\caption{An affinely periodic whirlpool framework.}
\label{whirlpool}
\end{figure}
\end{center}

We have
\[
p_1=(3,3), \,\,\,\,\, p_2=(-3,3), \,\,\,\,\,p_3=(-3,-3), \,\,\,\,\,p_4=(3,-3)\]
\[p_5=(1,2),\,\,\,\,\,p_6=(-2,1), \,\,\,\,\,p_7= (-1,-2), \,\,\,\,\,p_8=(2,-1) \]
and successive placements satisfy $p_{k+4}= Up_k$, with
$
U=\frac{1}{3}\left[ \begin {array}{cc} 1&2\\
2&1
\end {array} \right].
$
Here the $p_k$ are viewed as their column vectors.

The finite framework $(G_1,(p_1,\dots ,p_8))$ supported by the first eight vertices has a finite structure graph with edges $e_1,e_2,\dots , e_{12},$ for the pairs \[(12),\,(23),\,(34),\,(41),\,(56),\,(67),\,(78),\,(85),\,(15),\,(26),\,(37),\,(48)\]  The rigidity matrix $R=R(G,(p_1,\dots ,p_8))$ has the form
\[
R=\left[ \begin {array}{cc} R_1&0\\
0&R_2\\
X&-X\end {array} \right] 
\]
where
\[
R_1 = \left[\small{ \begin {array}{cccccccc} 6&0&-6&0&0&0&0&0\\ \noalign{\medskip}0
&0&0&6&0&-6&0&0\\ \noalign{\medskip}0&0&0&0&-6&0&6&0
\\ \noalign{\medskip}0&6&0&0&0&0&0&-6\end {array} }\right]\]
and \[R_2= 
 \left[\small{ \begin {array}{cccccccc} 3&1&-3&-1&0&0&0&0
\\ \noalign{\medskip}0&0&-1&3&1&-3&0&0\\ \noalign{\medskip}0&0&0&0&-3&
-1&3&1\\ \noalign{\medskip}-1&3&0&0&0&0&1&-3\end {array} }\right]
\]
are the rigidity matrices for the outer and inner square frameworks, and 
\[
X= \left[ \small{\begin {array}{cccccccc} 2&1&0&0&0&0&0&0\\ \noalign{\medskip}0
&0&-1&2&0&0&0&0\\ \noalign{\medskip}0&0&0&0&-2&-1&0&0
\\ \noalign{\medskip}0&0&0&0&0&0&1&-2\end {array} }\right] 
\]
Note that $a=((1,1),(1,-1),(-1,-1),(-1,1))$ is an infinitesimal flex of
the outer square, with half turn symmetry. 
Let $b\in \bR^8$ so that $u=(a,b)$ is a velocity vector for the $8$-vertex subframework.
Solving the linear system  $Ru^t=0$ for $b$ leads to an infinitesimal flex and taking $b$ with half-turn symmetry reduces this calculation to one with  $4$ rather than $8$ indeterminates. We have
\[
b = ((3/4,3/2),\,(3/2,-3/4),\,(-3/4,-3/2),\,(-3/2,3/4) 
\]
Also this infinitesimal flex of the inner square has the same form as that for the outer square in that its velocities are directed
along the diagonal directions of the square. In particular, the infinitesimal flex $u=(a,b,\dots )$ of the infinite whirlpool framework is now determined by the affine symmetry.
Since $\|(3/2,3/4)\|_2/\|(1,1)\|_2 = \sqrt{45/32} >1$ we have  $u=(u_k)$ with $\|u_k\| \to \infty $ as $k\to \infty$. Since the framework is generic with flexibility dimension $\dim_{\rm fl}(G,p)=1$ it follows that this divergence property holds for any nontrivial infinitesimal flex. Thus $(G,p)$ is infinitesimally  rigid for continuous velocity fields.

If a continuous flex of a finite  bar-joint framework
is differentiable then the derivative at $t=0$ is an infinitesimal flex. The same assertion therefore holds for infinite frameworks.
Thus we conclude immediately that  the whirlpool framework  does not admit a continuous flex that is smooth in the sense of being the restriction of a differentiable velocity field on $\bR^2$.  In fact one can show  that $(G,p)$ is continuously rigid.

Note that a whirlpool placement $(G,q)$ which is fully symmetric, with symmetry group of the square, has $\dim_{\rm fl}(G,p)=\infty$ since each square $4$-cycle subframework supports an infinitesimal rotation flex. In this case $(G,q)$  is not infinitesimally rigid for continuous velocity fields.
\end{eg}

\subsection{Continuous rigidity with respect to regular motions.}
The continuous motion of a translationally periodic strip framework can be quite chaotic. One way to make this precise is first to recall a theorem of Kempe \cite{kem} to the effect that
an arbitrary algebraic finite curve in the plane can be "drawn" by a finite mechanism. That is, the curve agrees with the trace of a joint in a (standardized) flex of uniquely continuously flexible finite bar-joint framework. In particular one can realise a chaotic logistic map $t \to \eta(t) = at(1-t)$ by a mechanism $M$ where $t$ and $\eta(t)$ are angle changes. Figure \ref{f:periodicmechanism} indicates a periodic mechanism built from such a finite mechanism in which the angular deviation $t$ from the vertical position at $A$ induces an angular deviation $\eta(t)$ at $B$.
\begin{center}
\begin{figure}[ht]
\centering
\includegraphics[width=7cm]{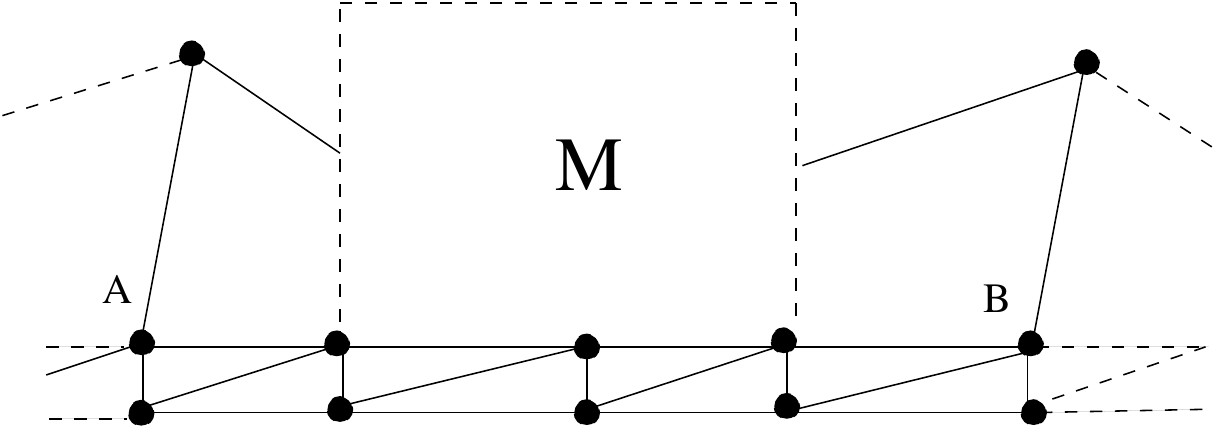}
\caption{Periodic mechanism framework.}
\label{f:periodicmechanism}
\end{figure}
\end{center}
The deviation induced by the increase of $t$ from $0$ to $t_0$ at the $n^{th}$ repeat, is
$\eta \circ \dots \circ \eta(t_0)$ ($n$ compositions).
It follows in particular that the sequence of speeds  of the placed vertices at any positive time is an unbounded sequence.

The same unbounded speeds property holds for the two-way infinite version of the strip framework of 
Figure \ref{f:InfRig}. However this framework does not exhibit chaotic continuous motion. In fact there are  two translationally periodic positions that may occur in any continuous flex (which fixes the base framework) namely the initial placement (given in Figure \ref{f:InfRig}) and  the periodic position in which all inclination angles
are at the fixed point value for the propagation function $\theta$.
For all other positions the inclination angles tend to $0$  to the left, while to the right the inclination angles tend towards the positive fixed point angle in an oscillatory fashion.

It follows that these frameworks can be regarded as being continuously rigid with respect to regular motions in the sense that the vertex speeds are uniformly bounded at each time instant. 

We also remark that there is  a countable bar-joint framework variant of Kempe's theorem for  the trace realisation of any finite continuous curve  $t \to \eta(t)$ (Owen and Power \cite{owe-pow-kempe}). In these realisations there are a small number of vertices of countable degree. It follows that with an appropriate infinite black box mechanism one may in principle construct countable graphs with periodic  placements with
diverse chaotic  motion. It also follows from such elaborate constructions that there exist countable graphs with generic placements which are mechanisms (uniquely continuously flexible frameworks) whose motion is nowhere differentiable.

\subsection{Continuously rigid infinite polytopes.}\label{ContinuousRigidity}
We now construct  continuously rigid  placements of simplicial graphs of finite connectivity in the Euclidean space $\bR^3$.

For a general dimension $d$ and sufficiently large finite graph $G$ with placement $(G,p)$    write $\hat{V}(G,p)$ for the normalised variety of equivalent frameworks $(G,q)$ with the  property that a specific set of $d(d+1)/2$
coordinates of $p$ agree with the corresponding set for $q$. This might be considered as the variety of standard placements equivalent to  $p$. Thus, for generic $p$ the set $\hat{V}(G,p)$ is finite if and only if $G$ is $d$-rigid. Also, for 
a sufficiently small open ball $B(p,\delta)$ the intersection
$
\hat{V}(G,p)^\delta = \hat{V}(G,p)\cap B(p,\delta)
$
is a manifold. 

It is also convenient to write  $\hat{V}_0(G,p)$ for the connected component
in  $\hat{V}(G,p)$ which contains $p$. Thus $\hat{V}_0(G,p)=\{p\}$ if and only if the framework is infinitesimally rigid, and otherwise there are continuous flexes taking values in  $\hat{V}_0(G,p)$.

We also write $\hat{V}(G,p)$ and $\hat{V}_0(G,p)$ for the corresponding subsets of the configuration space when $G$ is infinite.

\begin{lem}\label{l:diameter}
Let  $\{G_k:k\in\bN\}$ be an edge-complete tower  for the countable simple graph $G$, and 
for $k>j$ let  $\pi_{j,k}: \bR^{d|V(G_k)|}\to \bR^{d|V(G_j)|}$ be the canonical linear projection.
Suppose that  for sufficiently small $\delta$ the diameter of $\pi_{j,k}(\hat{V}_0(G_k,p)^\delta)$
tends to zero as $k\to \infty$ for each $j\in\bN$.
Then $(G,p)$ is continuously rigid.
\end{lem}

\begin{proof}
If $(G,p)$ is continuously flexible then there is a continuous flex taking values in $\hat{V}_0(G,p)$ which is not constant on $(G_j,p)$ for some index $j$. Moreover the sets $\pi_{j,k}(\hat{V}_0(G_k,p))$, for all $k>j$
contain the range of the  restriction of the flex to $(G_j,p)$. In view of the assumptions this is a contradiction for sufficiently large $k$.
\end{proof}

The next lemma concerns  the approximate rigidity of 
what might be referred to as the {\em slack placements} of a finite graph.
In such placements a number of edge incidence conditions are relaxed to proximity conditions.
In the case of
incidence relaxation at a single  vertex $v_1$ the vertex $v_1$ is removed but not the edges incident to it.
The incidence relation is recorded in terms of a reconnection map $\pi$ and this map also  features in placements that satisfy an approximate incidence condition.

To be more 
precise, let $G=(V,E)$ be a finite simple graph with $V=\{v_1,\dots ,v_n\}$ and suppose that
 $v_1$ has degree $m-1$ and incident edges $v_1v_2,\dots, v_1v_m$.
Let $\tilde{G}=(\tilde{V}, \tilde{E})$ where 
\[
\tilde{V} = (V\backslash \{v_1\}) \cup \{w_2,\dots ,w_m\}\]
\[\tilde{E} =
 (E\backslash \{v_1v_2,\dots ,v_1v_m\}) \cup \{w_2v_2,\dots ,w_mv_m\}
\]
so that  $\tilde{G}$ arises by disconnecting the edges of $G$ at $v_1$.
The reconnection map is the map $\pi : \tilde{V} \to {V}$ with  $\pi :w_i \to v_1$ for $2\leq i\leq m$. 

More generally the disconnection graph
$\tilde{G}$ may be defined over a finite subset, $V_2$ say, where $V=V_1\cup V_2$, by repeating of the disconnection operation. 
The associated reconnection map $\pi: \tilde{V} \to V$   is a surjection which induces a bijection $\tilde{E} \to E$.
With slight abuse of notation we may let $V_1$ denote the subset in $\tilde{V}$
corresponding to $V_1$,
so that $\tilde{V}=V_1\cup \pi^{-1}(V_2)$.

If
$p:V \to \bR^d$ is a placement for a bar-joint framework $(G,p)$  then there is a (non-injective) placement 
$\tilde{p}:V \to \bR^d$
which we call the derived placement, for which $\tilde{p}_v=p_v$ for $v\in V_1$ and  $\tilde{p}_w=p_{\pi(w)}$ for $ w\in \pi^{-1}(V_2)$.

\begin{lem}\label{l:Wdelta}
Let $G=(V,E)$ be a finite simple graph with $V=V_1\cup V_2$, let $p:V(G)\to \bR^d$ be a rigid, not necessarily  generic, placement and assume that $\hat{V}_0(G,p)$ is the standardised configuration space for a choice of $d(d+1)/2$
coordinates of the points $p_v$ with $v\in V_1$. Let $(\tilde{G},\tilde{p})$ be the derived placement for the disconnection graph of $G$ over $V_2$ and for $\delta>0$ let $W_\delta$ be the set of normalised $\delta$-placements of $\tilde{G}$ associated with reconnection map $\pi$;
\[
W_\delta=\{\tilde{q}\in 
\hat{V}(\tilde{G},\tilde{p}):\|\tilde{q}_w-\tilde{q}_{w'}\|\leq \delta \mbox{ for } w,w'\in V_2, \pi(w)=\pi(w')\}
\]
Then for $\epsilon >0$ there exists $\delta >0$ such that
\[
W_\delta \subseteq B(\tilde{p},\epsilon)
\]
\end{lem}

\begin{proof}It suffices to note
that the sets $W_\delta$ are compact and nested, and that their intersection is the singleton set $\{\tilde{p}\}$.
\end{proof}

Lemmas \ref{l:diameter} and \ref{l:Wdelta} may be used in a number of ways to  construct  continuously rigid placements for various countable graphs which are obtained by the refinement of $3$-rigid finite graphs.  
In particular, let $\P$ be a simplicial polytope in $\bR^3$ with $\kappa$ vertices and let $G$ be an associated countable simplicial graph $G$ of topological connectivity $\kappa$ arising from a triangulation of the faces. Then one can construct a continuous convex placement $(G,p)$ with polytope body $\P(G,p)=\P$ which is continuously rigid.
We omit the proof since it is a straightforward variant of the proof of
Theorem \ref{infinitepolytoperigid}.

Let us say that a continuous semi-convex placement $(G,p)$ of a countable
simplicial graph $G$ of connectivity $\kappa\geq 1$ is a \emph{flat placement} if the compact infinite polytope $\P(G,p)$ has well-defined tangent planes at the $\kappa$
points of accumulation of the vertex placements. Flatness is of interest as a  property for strictly convex
infinitely-faceted polytopes which is not present for finite simplicial polytopes. 
The next theorem shows that flatness is not in itself a handicap to continuous rigidity. 

\begin{thm}\label{infinitepolytoperigid}
Let $G$ be a simplicial graph of refinement type $\rho$ and finite connectivity $\kappa\geq 1$. Then there is a  convex flat continuous placement $(G,p)$ in $\bR^3$ which is continuously rigid. 
\end{thm}

\begin{proof}
For notational convenience we assume that $\kappa = \rho$. 
The general case follows similarly.
Let $G_1\subseteq G_2\subseteq \dots $ be an inclusion tower of edge-complete (vertex-induced) finite subgraphs which are simplicial, of finite connectivity $\kappa$ and span $G$. We may suppose moreover that for each $k$ the graph
$G_{k+1}$ has vertex set equal to the union of the vertices of $G_k$ and the vertices adjacent to $G_k$.

Let $G_k^+$ be the connected component containing $G_k$ of the disconnection of $G_{k+1}$ over $V(G_{k+1})\backslash V(G_k)$. Also let $G_k'$ be the multi-graph obtained by identifying the vertices of $G_k^+$ that are not in $G_k$.
Figure \ref{gk_sequence} indicates the relatedness of placements in $\bR^3$ of the graphs  $G_1, G_2, G_1^+$ and $G_1'$ in the case $\kappa=1$.

\begin{center}
\begin{figure}[ht]
\centering
\includegraphics[width=11cm]{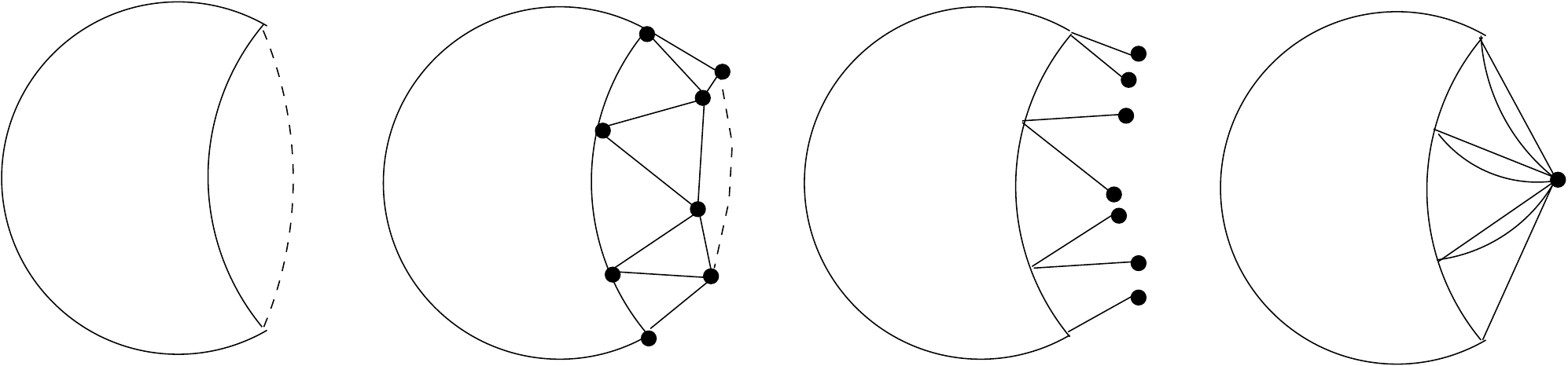}
\caption{Placements in $\bR^3$ of the graphs $G_1, G_2, G_1^+$ and $G_1'$.}
\label{gk_sequence}
\end{figure}
\end{center}
We have
\[
V(G_1')=V(G_1)\cup \{v_1',\dots , v_\kappa'\}
\]
and $G_1^+$ is the disconnection of $G_1'$ over the vertices 
$\{v_1',\dots v_\kappa'\}$. Note that $G_1'$ is a simplicial graph of connectivity zero with
possibly some edges repeated.

We proceed by making progressive placements of the graphs $G_k$. For the placement of $G_2$ following the placement of $G_1$ we make use of a temporary placement of the $3$-rigid multi-graph $G_3'$ as a guide. In this step, which also provides the argument for subsequent placements, the placement of $G_1$ is extended to a placement of $G_2$ so that the diameter of the projection to $G_1$ coordinates of the standardised variety of $G_2$ has diameter less than $\epsilon.$ 

Let $\epsilon>0$. Place $G_1$ generically on the unit sphere to create $(G_1, p(1))$.
Since $G_1'$ is a simplicial graph of connectivity zero with
possibly some edges repeated there is an extended generic placement $p'(1)=(p(1),a_1,a_2,\dots , a_\kappa)$ of
$G_1'$ on the unit sphere which is continuously rigid. By Lemma \ref{l:Wdelta}
it follows that there exists $\delta >0$ such that
\[
W_{\delta}(G_1^+,p'(1)) \subseteq B(p'(1),\epsilon)
\]

With this control of the diameter of the standardised configuration space of the $\delta_1$-slack placements of $(G_1^+,p'(1))$ we now specify a placement of $G_3$ which extends the placement of $G_1$. We denote this as $(G_3,p(3))$ where
$p(3)=(p(1),q(3))$ and we do this in such a manner that the length of each  perimeter of the nontriangular faces of $G_2$ is less than $\delta$.
This ensures that the projection of the standardised variety
$\hat{V}(G_2,p(3)|_{G_2})$ to $\bR^{3|V(G_1)|}$ is a subset of the corresponding projection of $W_{\delta}(G_1^+,p'(1))$ and hence of  the corresponding projection of $B(p'(1),\epsilon_1)$.

Similarly, with $G_3$ playing the role of $G_1$ we may determine an extended placement of $G_4$ which reduces the diameter of the projection
to $G_2$ coordinates of the standardised variety of $G_3$. 
For a sequence $\epsilon_k>0$ which tends to zero we may progressively determine a placement $p$ for $G$ so that  the diameter of the projection to $G_k$ coordinates of the standardised variety of the placement of $G_{k+1}$ has diameter less than $\epsilon_k$. That $(G,p)$ is continuously rigid follows from Lemma \ref{l:diameter}.
\end{proof}

\subsection{Remarks}
Generic placements do not play a significant role in the argument above since the guiding placements of the $3$-rigid  simplicial polyhedral graphs $G_k'$ 
will be continuously rigid when given a convex placement.
In particular  
it follows that the polytope frameworks in Examples \ref{eg:diamond} and \ref{polytopeframeworks} have latitude specifications which give continuously rigid placements. The determination of whether arbitrary semi-convex continuous placements of simplicial polytope graphs of finite connectivity are continuously rigid is an interesting issue which requires a more refined  investigation of relative continuous rigidity. 

The bar-joint frameworks of strip graphs and their two-way infinite translationally periodic variants are also considered in Owen and Power \cite{owe-pow-crystal}. 

We note that Servatius and Servatius \cite{ser-ser} have recently shown
that a sufficiently winding continuous flex of a finite framework $(G_1,q)$ need not extend fully to \emph{any} extension framework $(G_2,q^+)$, where $G_2$ is obtained from $G_1$ by a Henneberg $2$ move
and $q^+$ is the augmentation of $q$ by the new vertex placement. This shows that some caution is required if one wishes to define a continuous flex of an infinite framework by means of extensions through an infinite a construction sequence of  Henneberg moves.


\bibliographystyle{abbrv}
\def\lfhook#1{\setbox0=\hbox{#1}{\ooalign{\hidewidth
  \lower1.5ex\hbox{'}\hidewidth\crcr\unhbox0}}}

\end{document}